 \documentclass[10pt]{amsart}
\usepackage{graphicx}
\usepackage{psfrag}
\usepackage{float}
\usepackage{subfigure}
\usepackage{psfrag}
\usepackage{epsfig}
\usepackage{epstopdf}
\usepackage{amsmath}
\usepackage{amscd}

 \newcommand{\dodo}{{\mathtt D}}

\newcommand{\et}{{\mathfrak t}}
\newcommand{\pp}{{\mathtt p}}

\newcommand{\ro}{{\varepsilon^2}}

\def\tet1{{\theta}}
\usepackage{color}
\definecolor{orange}{rgb}{1,.549,0}
 \definecolor{GreenYellow }{rgb}{ 0.15,   0.69, 0} 
\definecolor{Yellowone}{rgb}{ 0, 1., 0} 
\definecolor{Goldenrod }{rgb}{  0, 0.10, 0.84} 
\definecolor{Dandelion }{rgb}{ 0, 0.29, 0.84} 
\definecolor{Apricot }{rgb}{ 0, 0.32, 0.52} 
\definecolor{Peach }{rgb}{ 0, 0.50, 0.70} 
\definecolor{GreenYellow}{cmyk}{0.15,0,0.69,0}
\definecolor{RoyalPurple}{cmyk}{0.75,0.90,0,0}
\definecolor{Yellow}{cmyk}{0,0,1,0}
\definecolor{BlueViolet}{cmyk}{0.86,0.91,0,0.04}
\definecolor{Goldenrod}{cmyk}{0,0.10,0.84,0}
\definecolor{Periwinkle}{cmyk}{0.57,0.55,0,0}
\definecolor{Dandelion}{cmyk}{0,0.29,0.84,0}
\definecolor{CadetBlue}{cmyk}{0.62,0.57,0.23,0}
\definecolor{Apricot}{cmyk}{0,0.32,0.52,0}
\definecolor{CornflowerBlue}{cmyk}{0.65,0.13,0,0}
\definecolor{Peach}{cmyk}{0,0.50,0.70,0}
\definecolor{MidnightBlue}{cmyk}{0.98,0.13,0,0.43}
\definecolor{Melon}{cmyk}{0,0.46,0.5,0}
\definecolor{NavyBlue}{cmyk}{0.94,0.54,0,0}
\definecolor{YellowOrange}{cmyk}{0,0.42,1,0}
\definecolor{RoyalBlue}{cmyk}{1,0.50,0,0}
\definecolor{Orange}{cmyk}{0,0.61,0.87,0}
\definecolor{Blue}{cmyk}{1,1,0,0}
\definecolor{BurntOrange}{cmyk}{0,0.51,1,0}
\definecolor{Cerulean}{cmyk}{0.94,0.11,0,0}
\definecolor{Bittersweet}{cmyk}{0,0.75,1,0.24}
\definecolor{Cyan}{cmyk}{1,0,0,0}
\definecolor{RedOrange}{cmyk}{0,0.77,0.87,0}
\definecolor{ProcessBlue}{cmyk}{0.96,0,0,0}
\definecolor{Mahogany}{cmyk}{0,0.85,0.87,0.35}
\definecolor{SkyBlue}{cmyk}{0.62,0,0.12,0}
\definecolor{Maroon}{cmyk}{0,0.87,0.68,0.32}
\definecolor{Turquoise}{cmyk}{0.85,0,0.20,0}
\definecolor{BrickRed}{cmyk}{0,0.89,0.94,0.28}
\definecolor{TealBlue}{cmyk}{0.86,0,0.34,0.02}
\definecolor{Red}{cmyk}{0,1,1,0}
\definecolor{Aquamarine}{cmyk}{0.82,0,0.30,0}
\definecolor{OrangeRed}{cmyk}{0,1,0.50,0}
\definecolor{BlueGreen}{cmyk}{0.85,0,0.33,0}
\definecolor{RubineRed}{cmyk}{0,1,0.13,0}
\definecolor{Emerald}{cmyk}{1,0,0.50,0}
\definecolor{WildStrawberry}{cmyk}{0,0.96,0.39,0}
\definecolor{JungleGreen}{cmyk}{0.99,0,0.52,0}
\definecolor{Salmon}{cmyk}{0,0.53,0.38,0}
\definecolor{SeaGreen}{cmyk}{0.69,0,0.50,0}
\definecolor{CarnationPink}{cmyk}{0,0.63,0,0}
\definecolor{Green}{cmyk}{1,0,1,0}
\definecolor{Magenta}{cmyk}{0,1,0,0}
\definecolor{ForestGreen}{cmyk}{0.91,0,0.88,0.12}
\definecolor{VioletRed}{cmyk}{0,0.81,0,0}
\definecolor{PineGreen}{cmyk}{0.92,0,0.59,0.25}
\definecolor{Rhodamine}{cmyk}{0,0.82,0,0}
\definecolor{LimeGreen}{cmyk}{0.50,0,1,0}
\definecolor{Mulberry}{cmyk}{0.34,0.90,0,0.02}
\definecolor{YellowGreen}{cmyk}{0.44,0,0.74,0}
\definecolor{RedViolet}{cmyk}{0.07,0.90,0,0.34}
\definecolor{SpringGreen}{cmyk}{0.26,0,0.76,0}
\definecolor{Fuchsia}{cmyk}{0.47,0.91,0,0.08}
\definecolor{OliveGreen}{cmyk}{0.64,0,0.95,0.40}
\definecolor{Lavender}{cmyk}{0,0.48,0,0}
\definecolor{RawSienna}{cmyk}{0,0.72,1,0.45}
\definecolor{Thistle}{cmyk}{0.12,0.59,0,0}
\definecolor{Sepia}{cmyk}{0,0.83,1,0.70}
\definecolor{Orchid}{cmyk}{0.32,0.64,0,0}
\definecolor{Brown}{cmyk}{0,0.81,1,0.60}
\definecolor{DarkOrchid}{cmyk}{0.40,0.80,0.20,0}
\definecolor{Tan}{cmyk}{0.14,0.42,0.56,0}
\definecolor{Purple}{cmyk}{0.45,0.86,0,0}
\definecolor{Gray}{cmyk}{0,0,0,0.50}
\definecolor{Plum}{cmyk}{0.50,1,0,0}
\definecolor{Black}{cmyk}{0,0,0,1}
\definecolor{Violet}{cmyk}{0.79,0.88,0,0}
\definecolor{White}{cmyk}{0,0,0,0}

 \definecolor{rltred}{rgb}{0.75,0,0}
   \definecolor{rltgreen}{rgb}{0,0.5,0}
   \definecolor{oneblue}{rgb}{0,0,0.75}
   \definecolor{marron}{rgb}{0.64,0.16,0.16}
   \definecolor{forestgreen}{rgb}{0.13,0.54,0.13}

   \definecolor{purple}{rgb}{0.62,0.12,0.94}
   \definecolor{dockerblue}{rgb}{0.11,0.56,0.98}
   \definecolor{freeblue}{rgb}{0.25,0.41,0.88}
   \definecolor{myblue}{rgb}{0,0.2,0.4}

 \definecolor{Melon}{rgb}{ 0.46, 0.50, 0}

 \definecolor{Melone}{rgb}{ 0, 0.46, 0.50}
\usepackage{amsfonts,amsxtra,amssymb,amscd}
\usepackage[all]{xy}
\xyoption{2cell}
\usepackage{enumerate}
\newcommand{\Gama}{{\Gamma_S}}
\newcommand{\er}{{\mathtt r}}

\newcommand{\GA}{{\mathcal A}}
\newcommand{\Co}{{\mathbb C}}
\usepackage{pdfsync}

 \usepackage[unicode]{hyperref}

\font\teneufm=eufm10
\font\seveneufm=eufm7
\font\fiveeufm=eufm5
\newfam\eufmfam
\textfont\eufmfam=\teneufm
\scriptfont\eufmfam=\seveneufm
\scriptscriptfont\eufmfam=\fiveeufm

\newcommand{\vgot}{{\mathtt v}}

\theoremstyle{plain}
\newtheorem{lemma}[subsection]{Lemma}
\newtheorem*{theorem*}{Theorem}
\newtheorem{theorem}{Theorem}
\newtheorem*{proposition*}{Proposition}
\newtheorem{proposition}[subsection]{Proposition}
\newtheorem*{corollary*}{Corollary}

\newtheorem{corollary}[subsection]{Corollary}
\theoremstyle{definition}
\newtheorem*{definition*}{Definition}

\newtheorem{definition}[subsection]{Definition}
\newtheorem*{example*}{Example}

\theoremstyle{remark}
\newtheorem*{remark*}{Remark}
\newtheorem{remark}[subsection]{Remark}

\newcommand{\T}{{\mathbb T}}
\def\frec#1{{\stackrel{#1}\rightarrow}}

\newcommand{\N}{{\mathbb N}}

\newcommand{\be}{\begin{equation}}
\newcommand{\ee}{\end{equation}}

\newcommand{\e}{\varepsilon}

\newcommand{\al}{\alpha}

\renewcommand{\a }{\alpha }
\renewcommand{\b }{\beta }
\newcommand{\s }{\sigma }
\newcommand{\ii }{{\rm i} }
\renewcommand{\d }{\delta }

\renewcommand{\l }{\lambda }

\newcommand{\ome}{{\omega}}
\newcommand{\Ome}{{\Omega}}
\renewcommand{\O }{\mathcal O }

\newcommand{\Z}{\mathbb{Z}}
\newcommand{\R}{\mathbb{R}}

\renewcommand{\O }{\mathcal O }

\begin{document}
 \title[Reducible quasi-periodic solutions  for the NLS]{Reducible quasi-periodic solutions for the Non Linear Schr\"odinger equation\thanks{ The first author is  supported by ERC  project HamPDEs, under FP7 
 }}

\author{ M. Procesi \and 
 C. Procesi }
 
\maketitle
\begin{abstract}
 The present paper is devoted to the  construction of  small reducible quasi--periodic solutions for the completely resonant   NLS equations on a $d$--dimensional  torus $\T^d$. The main point is to prove that prove that the normal form is reducible, block diagonal  and satisifes the second Melnikov condtiton block wise. From this we deduce the result by a KAM algorithm.
 \keywords{ KAM theory for PDEs \and Quasi-periodic solutions.}
MSC: 37K55, 35Q55
\end{abstract}

\maketitle

\section{Introduction}\label{intro}
 The present paper is devoted to the  construction of  small reducible  (see Definition \ref{reducible}) quasi--periodic solutions for the completely resonant   NLS equations on a $d$--dimensional  torus $\T^d$:
\begin{equation}\label{scE}iu_t-\Delta u=\mathtt k |u|^{2q}u +\partial_{\bar u}G(|u|^2).\end{equation} 
Here $u:= u(t,\varphi)$, $\varphi\in \T^d$, $\Delta$ is the Laplace operator  and  $G(a)$ is a real analytic function whose Taylor series starts from degree $q+2,\ q\geq 1$. Finally $\mathtt k$ is a coupling constant which can be normalised to $\pm 1$. Since in our results the sign is irrelevant we will   set it equal to one. 
Note that we have restricted our attention to $\varphi$-independent non-linearities $G$. The corresponding symmetries imply the presence of $d+1$ constants of motion given by the $L^2$ norm and the momentum (translation invariance).

A relevant feature is that the NLS equation is completely resonant near $u=0$ (i.e. all the linear solutions are periodic), hence we look for  our quasi-periodic solutions close to some {\em periodic}   solutions  \begin{equation}\label{linsol}
\sum_{i=1}^n\sqrt{\xi_i} e^{\ii  |\mathtt j_i|^2 \,t }e^{ \ii\mathtt j_i\cdot\varphi}
\end{equation} of the linear equation  involving $n$   frequencies $S=\{\mathtt j_1,\ldots, \mathtt j_n\}\subset \Z^d$. It is well known, see for instance \cite{KT}, that  due to the presence of resonances, there exist choices of the frequencies $S$ for which the solutions of the non--linear system differ drastically from the ones of the linear system. In order to avoid  such phenomena we restrict to {\em generic} choices of $S$, where generic  means that this list of vectors does not satisfy some explicit, although quite complicated,  polynomial equations which express the {\em resonances} to be avoided.

  Our results are obtained  by exploiting the Hamiltonian structure of equation \eqref{scE}, and by studying  a simplified Hamiltonian,  denoted $H_{Birk}$ see Formula \eqref{Ham2},  obtained from  $H $ by removing all  terms of degree $2q+2$  which do not Poisson commute with the quadratic part. Its Hamiltonian vector field is tangent to infinitely many subspaces obtained by setting some of the coordinates equal to  0 (cf. \cite{PP}, Prop. 1).  On infinitely many of them  the restricted system is  completely integrable, thus  the next step  consists in choosing such a subset $S$ which, for obvious reasons, is called of {\em tangential sites}.  
  
  In this  step  the linear solution of \eqref{linsol}  deforms to a quasi--periodic solution  \begin{equation}\label{linsol1}
\sum_{i}\sqrt{\xi_i} e^{\ii t (|\mathtt j_i|^2+\omega^{(1)}_i(\xi))  }e^{ \ii\mathtt j_i\cdot\phi },\quad \text{for}\quad  \omega^{(1)}_i(\xi)\quad\text{see Formula \eqref{ome1}}.
\end{equation}We  then apply   a KAM algorithm, starting from  the small quasi--periodic orbits  parametrised by a suitable domain of parameters $\xi$.  To be precise
  \begin{theorem}\label{ilteorema}  For any choice of  $n$ generic  frequencies $S=\{\mathtt j_1,\ldots, \mathtt j_n\}\subset \Z^d$,  and for $\e$ sufficently small, there exists a compact set $\mathcal O_\infty$ contained in  $ \{(\xi_1,\ldots,\xi_n)\},$ $ \e^2/2\leq \xi_i\leq \e^2 \}$ of measure of order $\e^{2n}$, parametrizing bijectively a set of quasi--periodic solutions  of  \eqref{scE}  which are a small perturbation  of the  solutions of type \eqref{linsol1} of the  equations associated to the hamiltonian $H_{\rm Birk}$.
  Moreover  the quasi-periodic solutions for all $\xi\in \mathcal O_\infty$ are reducible KAM tori, see Definition \ref{reducible}. 
\end{theorem}
  
  As is well known KAM algorithms require strong {\em non-degeneracy} conditions   not always valid, even for finite dimensional systems, this has for long time been an obstacle for applications to PDEs on tori. Indeed existence results for quasi-periodic solutions (with no control on the reducibility) for such equations were proved, starting from the late '$90$, by "multiscale" techniques (see \cite{Bo3},\cite{BBhe1}). In the case of the resonant NLS we mention the paper \cite{W1} which covers our equation \eqref{scE} and provides an existence result. 
  The breakthrough result in KAM theory was in the paper \cite{EK}, where the authors proved reducibility for a class of non-resonant NLS equations (see also \cite{PX} and \cite{GY}, \cite{EGK} for the beam equation).
  In the case of resonant PDEs the first problem that arises in KAM schemes is to prove (when it holds)  reducibility for the Birkhoff normal form, for the NLS this was done in \cite{PP1}. Then, in order to proceed with the KAM scheme one needs further  {\em non-resonance} assumptions on the normal form (the so called Melnikov conditions) which are in general much harder to prove that in the non-resonant cases.       \smallskip

 In the case of the cubic NLS, i.e. $q=1$,  in \cite{PP3} we have discussed in detail the KAM algorithm and proved the existence of families of stable and unstable quasi-periodic solutions. This required a very subtle combinatorial analysis (performed in \cite{PP1}) which enabled us to prove the {\em second Melnikov conditions} (which amounts to proving that the NLS equation linearised at an approximate solution has distinct eigenvalues on the space of quasi-periodic functions). This combinatoric is not available for $q>1$  except in dimension $d\leq 2$ (see. \cite{Van}).

  In the present paper we discuss the general case $d\geq 1, q\geq 1$ and prove that,  for a  generic choice of the excited frequencies $S$,   the  multiplicity of the eigenvalues of the corresponding linearised system  is uniformly bounded, and moreover there is a normal form with a block diagonal  non--degeneracy, see Proposition \ref{teo2}.
  
  Using the properties of  this  normal form in     section \S \ref{aKT} we explain how a KAM  algorithm  leads to existence and reducibility of quasi-periodic solutions. This requires some minor variations in the KAM scheme of \cite{PP3} in order to take into account the block diagonal structure (essentially one needs to control more derivatives in the $\xi$ variables). Since  the material is quite standard, but very heavy and lengthly, we only give a schematic proof, contained Propositions \ref{kamforma} and \ref{misura}, see \ref{prilte}. The same kind of problems appear in the preprint \cite{EGK} where the authors study the non-linear beam equation.\medskip
    
\section{The Hamiltonian formalism} Passing to the Fourier representation \begin{equation}\label{FR}
u(t,\varphi):= \sum_{k\in \Z^d} u_k(t) e^{\ii (k, \varphi)}\ 
\end{equation} 
we have, up to a rescaling of $u$ and of time, 
 in coordinates,   dropping the part of  $G$ which will be placed in the perturbation, 
 the Hamiltonian: 
\begin{equation}\label{Ham}H:=\sum_{k\in \Z^d}|k|^2 u_k \bar u_k +  \sum_{k_i\in \Z^d:\ \sum_{i=1}^{2q+2}(-1)^i k_i=0}\hskip-30pt u_{k_1}\bar u_{k_2}u_{k_3}\bar u_{k_4}\ldots u_{k_{2q+1}}\bar u_{k_{2q+2}} . \end{equation}

The complex symplectic form is $ i \sum_{k}d u_k\wedge d \bar u_k$, and we work on the scale of complex Hilbert spaces $(u,\bar u)\in {\bf{\bar \ell}}^{(a,p)}\times {\bf{\bar \ell}}^{(a,p)}$, where:
\begin{equation}\label{scale}
{\bf{\bar \ell}}^{(a,p)}:=\{ u=\{u_k \}_{k\in \Z^d}\;\big\vert\;|u_0|^2+\sum_{k\in \Z^d} |u_k|^2e^{2 a |k|} |k|^{2p}:=||u||_{a,p}^2 \le \infty \},\end{equation}
Where $a>0,\ p>d/2$.
Note that both $H$ and its Hamiltonian vector field $X_H$ are analytic functions on these spaces\footnote{it is well known that the NLS is  locally well posed under much weaker regularity conditions. This is not the purpose of the present paper.} 
 We denote as usual by $\{A,B\}$ the associated Poisson bracket and, if we want to stress the role of one of the two variables, we also write $ad(A)$ for the linear operator $B\mapsto \{A,B\}$.\footnote{$ad$ stands for {\em adjoint} in the language of Lie algebras.}
 
 \smallskip

We   systematically apply the fact that we have $d+1$ conserved  quantities:
the  $d$--vector {\em momentum} \ $\mathbb M:=\sum_{k\in \Z^d}   |u_k|^2k$ and the scalar {\em mass}  $\mathbb L:= \sum_{k\in \Z^d} |u_k|^2$
with 
\begin{equation}
\{\mathbb M,u_h\}=\ii   u_hh,\ \{\mathbb M,\bar u_h\}=-\ii   \bar u_hh,\ \{\mathbb L,u_h\}=\ii u_h,\ \{\mathbb L,\bar u_h\}=-\ii \bar u_h.\ 
\end{equation}\smallskip
The terms in equation \eqref{Ham} commute with $\mathbb L$. The conservation of momentum is expressed by the constraints $\sum_{i=1}^{2q+2}(-1)^i k_i=0$.

  We partition 
$
\Z^d= S\cup S^c,\quad S:=(\mathtt j_1,\ldots,\mathtt j_n)$
where 
the elements of  $S$ play the role of {\em tangential sites} and those of the complement $S^c$ the {\em normal sites}. We divide $u\in \bar\ell^{a,p}$ in two components $u=(u_1,u_2)$, where $u_1$ has indexes in $S$ and $u_2$ in $S^c$.   Here we  always assume that $S$ is subject to the constraints which make it {\em generic} and which are fully discussed in \cite{PP} and finally refined in \cite{PP1}. Further constraints will appear later in this paper.

 We apply a standard {\em semi-normal form} change of variables  
$\Psi^{(1)}:=e^{ad( F_{Birk})}$, which is
 well defined and analytic:    $   B_{\epsilon_0}\times B_{\epsilon_0} \to B_{2{\epsilon_0}} \times B_{2\epsilon_0}$,  for $\epsilon_0$ small enough, see \cite{PP}.   
 
 The map   $\Psi^{(1)}$  brings (\ref{Ham}) to  the form
 $H= H_{Birk} +P^{2q+2}(u)+P^{4q+2}(u)$ where $P^{2q+2}(u)$ is of degree $2q+2$  in $u$ and  at least cubic in $u_2$
  while $P^{ 4q+2 }(u)$ is analytic of degree at least $4q+2$ in $u$, finally 
\begin{equation}\label{Ham2}H_{Birk}:=\sum_{k\in \Z^d}|k|^2 u_k \bar u_k +  \sum_{\alpha,\beta\in\mathcal C'}  \binom{2}{\alpha}\binom{2}{\beta}u^\alpha\bar u^\beta.
\end{equation} 
\begin{equation}\label{Ham2b}  \mathcal C': \alpha,\beta\in (\Z^d)^\N \,\left | \begin{array}{ll} & |\alpha_2|+|\beta_2|\leq 2;\  |\alpha|=|\beta|=q+1\,,\\  &\sum_k (\alpha_k-\beta_k)k=0\,,\;\sum_k (\alpha_k-\beta_k)|k|^2=0.\end{array}\right.
\end{equation} 
The constraint $ |\alpha_2|+|\beta_2|\leq 2$ comes from the definition of  $P^{2q+2}(u)$, 
the other three constraints in this formula express the conservation of $\mathbb L$, $\mathbb M$ and  of the {\em quadratic energy}: 
\begin{equation}
\label{kappabb}\mathbb K:= \sum_{k\in \Z^d}|k|^2 u_k \bar u_k .
\end{equation}
 Switching to polar coordinates we set \begin{equation}
\label{chofv}u_k:= z_k \;{\rm for}\; k\in S^c\,,\quad u_{\mathtt j_i}:= \sqrt {\xi_i+y_i} e^{\ii x_i}= \sqrt {\xi_i}(1+\frac {y_i}{2 \xi_i }+\ldots  ) e^{\ii x_i}
\end{equation}  for $i=1,\dots,n$. Here we conside  the $\xi_i>0$ as parameters $|y_i|<\xi_i$ while $y,x,w:=(z,\bar z)$ are dynamical variables\footnote{To be completely formal one should think of $z,\bar z$ as independent dynamical variables, to this purpose they are often denoted by $z^+,z^-$, then one shows that the {\em real} subspace, where $\bar z^+=z^-$ is invariant w.r.t. the dynamics}.   We  denote   by   $ {\bf \ell}^{(a,p)}:= {\bf\ell}^{(a,p)}_S $  the subspace of $\bf{\bar \ell}^{(a,p)}\times\bf{\bar \ell}^{(a,p)} $  of the sequences $u_i,\bar u_i$ with indices in $S^c$ and denote the coordinates  $w=(z,\bar z) $.
We define
 $$ \Lambda:=\left[\frac1 2, \frac32\right]^n.$$ 
We choose $\e>0$ small and note that for all $\xi\in \e^2\Lambda$ and  for all $r<    \e/2$, formula \eqref{chofv} is an analytic and symplectic change of variables $\Phi_\xi$ in  the  domain 
\begin{eqnarray}\label{domain}  & & D_{a,p}(s,r) = D(s,r)
:= \nonumber \\
& &   \{   x,y,w\,:\   x\in \T^n_s\,,\  |y|\le r^2\,,\  \|w\|_{a,p}\le r\}\subset \T^n_s\times\Co^n\times {\bf{\ell}}^{(a,p)}.
\end{eqnarray}
 Here $\e>0$, $s>0$ and $ 0<r<\e /2$ are auxiliary parameters while $\T^n_s$ denotes the compact subset of the complex torus $\T_{\Co}^n:=\Co^n/2\pi\Z^n$ where   $ x\in\Co^n,\ |$Im$(x)|\leq s$. Moreover  there exist universal constants $c_1<1/2 ,c_2$ such that if \begin{equation}   r<c_1\e\quad\text{ and }\quad 
\label{bapa}\sqrt{2 n} c_2 \kappa^{p} e^{ ( s+ a\kappa)} \e  < {\epsilon_0}\,, \quad \kappa:=\max(|\mathtt j_i|)
\end{equation} the change of variables  sends $D(s,r)\to B_{{\epsilon_0}}$  so we can apply it to our Hamiltonian.

We thus assume that the parameters $\e,\, r,\, s$  satisfy \eqref{bapa}.
Formula \eqref{chofv}  puts  in action angle variables  $(y;x)= (y_1,\dots, y_n;x_1,\dots, x_n) $ the tangential sites, close to the action $\xi$,  parameters for the system.  
The  symplectic form is now $ dy \wedge dx + i \sum_{k\in S^c} dz_k\wedge d \bar z_k $.   
 By abuse of notations we still call $H$ the composed Hamiltonian $H\circ \Psi^{(1)}\circ \Phi_\xi$.

Remark that, in polar coordinates the  Hamiltonians $\mathbb L,\ \mathbb M,\ \mathbb K$, after dropping some constant terms  which Poisson commute with everything become
\begin{eqnarray} \label{prLM}
\mathbb M&=&\ \sum_i   y_i \mathtt j_i+\sum_{k\in S^c}k|z_k|^2,\  \mathbb L= \sum_i   y_i  +\sum_{k\in S^c} |z_k|^2\,, \nonumber\\ \mathbb K&=& \mathtt j^{(2)}\cdot y +\sum_{k\in S^c} |k|^2|z_k|^2\,,\
 \mathtt j^{(2)}:=(|\mathtt j_1|^2,\dots,|\mathtt j_n|^2).
\end{eqnarray}
 \paragraph{ The  standard form\label{stanfor}}
 By the rules of Poisson bracket, on the real space spanned by $z,\bar z$, we have $\{\bar a,\bar b\} =-\{  a,b\}=\overline{\{  a,b\}}$ so $\{a,\bar b\} =-\{\bar  a,b\}= \{b,\bar  a \}$ is imaginary: \begin{definition}
 $(a,b):=\ii \{a,\bar b\}$ is a real symmetric form, called the {\em standard form}.
\end{definition}
 For the variables  we have $(z_h,z_k)=\delta^h_k,\ (\bar  z_h,\bar  z_k)=-\delta^h_k$, so the form is positive definite on the space spanned by the $z$, which give an orthonormal basis, negative on the space spanned by the $\bar z$  and of course indefinite if we mix the two types of variables.  Thus we may say that an element $a$ in the real space spanned by $z,\bar z$ is {\em of type $z$ (resp. $\bar z$)} if  $(a,a)=1$ resp.   $(a,a)=-1$. For a  quadratic real Hamiltonian $\mathcal H=\bar{\mathcal H}$ we have  $\{\mathcal H,\{a,\bar b\}\}=0$  since $\{a,\bar b\}$ is a scalar. The map $x\mapsto \ii\{\mathcal H,x\}$ preserves the real subspace spanned by $z,\bar z$ hence   by the Jacobi identity $(a, \ii\{\mathcal H,  b\})=  (\ii\{\mathcal H,a\},  b )$ so
  the operator $\ii\{\mathcal H,-\}$ is symmetric with respect to this  form.

\subsection{Functional setting}
Following \cite{Po} we study {\em regular} functions $F:\ro \Lambda\times D_{a,p}(s,r)\to \Co$, that is whose Hamiltonian vector field  $X_F(\cdot;\xi)$ is M-analytic from $D(s,r)\to \Co^n\times\Co^n\times\ell^{a,p}_S$. In the variables $\xi$ we require $C^{5d^2}$ regularity.  Let us recall the  definitions from \cite{BBP}.

\noindent Let us consider the Banach space $
V := \Co^n \times \Co^n \times  \ell^{a, p}_{S}
$
with  $(s,r)$-weighted norm
\begin{equation}\label{normaEsr}
\overrightarrow v =  (x,y,z,\bar z) \in V \, , \
  \|\overrightarrow v\|_{V,s,r}:= \frac{|x|_\infty}{s} + \frac{|y|_1}{r^2}
 +\frac{\|z\|_{a,p}}{r}+\frac{\|\bar z\|_{a,p}}{r}
\end{equation}
where,  $ |x|_\infty := \max_{h =1, \ldots, n} |x_h| $,
$ |y|_1 := \sum_{h=1}^n |y_h| $ and we restrict $r,s$  with  $ 0 < s< 1 $, $0 < r <c_1\e$.

For a vector field  $X:  D(s,r)\to V$, described by the formal Taylor expansion:
$$
 X = \sum_{\nu,i,\a,\b} X_{\nu,i,\a,\b}^{(\vgot)} e^{\ii (\nu, x)} y^i z^\a \bar z^\b \partial_{\mathtt v}\,, \quad \mathtt v= x,y,z,\bar z$$   we define the {\em majorant} and its {\em norm}:
\begin{eqnarray}\label{normadueA}
MX & :=& \sum_{\nu,i,\a,\b} |X_{\nu,i,\a,\b}^{(\vgot)}| e^{s|\nu|} y^i z^\a \bar z^\b \partial_{\mathtt v}\,, \quad \mathtt v= x,y,z,\bar z\nonumber\\
|| X ||_{s,r} & := &
\sup_{(y,z, \bar z) \in D(s,r)} \| M X \|_{V,s,r} \,.  \end{eqnarray}
 The different weights ensure that, if $\Vert X_F\Vert_{s,r}$ is sufficiently small,  then $F$ generates a close--to--identity symplectic change of variables from $D(s/2,r/2)\to D(s,r)$.

In our algorithm we deal with functions which depend in a $C^{\ell}$ way on some parameters $\xi$ in a compact set $ \mathcal O \subseteq \e^2 \Lambda$, the integer $\ell$ in fact will be chosen to be $(2d+1)^2$,  the maximal size of diagonal blocks of the normal form of the NLS. To handle this dependence we introduce weighted  norms   for a map $X:  \O \times D(s,r)\to V$ setting: 
\begin{equation}
\label{weno}   \Vert X(\xi)\Vert^\lambda_{s,r}:=  \!\sum_{k\in \N^n:|k|\leq \ell}\lambda^{|k|}  \|\partial_\xi^kX\|_{s,r} ,\quad \Vert X\Vert^\lambda_{s,r,\mathcal O} :=\sup_{\xi\in \mathcal O}\Vert X(\xi)\Vert^\lambda_{s,r}
\end{equation}  where $\lambda $ is a  parameter of order $\e^{2}$.  Sometimes when $\mathcal O$ is understood we just  write $ \Vert X\Vert^\lambda_{s,r,\mathcal O}= \Vert X\Vert^\lambda_{s,r}$.

Finally we introduce a stronger norm, called  {\em quasi--T\"oplitz}  and denoted  by $\Vert \cdot \Vert^{T}_{\overrightarrow p}$, proposed first in \cite{PX},  which controls  the behaviour of linear Hamiltonian vector fields with  respect to linear stratifications (cf. Definition \ref{lins}). This will be  essential for the KAM algorithm. This norm is defined through  the parameters $\overrightarrow p$ given by $\lambda,s,r,\mathcal O$ and three parameters $K,\vartheta,\mu$ with $K$ a large positive integer and $\frac12 <\vartheta,\mu< 4$.
\begin{definition}
We define by $\mathcal V_{\lambda,s,r,\mathcal O}=\mathcal V_{s,r},$ $ \mathcal H_{\l,s,r,\mathcal O}=\mathcal H_{s,r}$, resp. $\mathcal H_{\overrightarrow p}^T$ with $\overrightarrow p= (\lambda,s,r,\vartheta,\mu,K,\mathcal O)$, the space of $M$--analytic vector fields, resp. regular  analytic and finally quasi--T\"oplitz, Hamiltonians depending on a parameter $\xi\in\mathcal O$ 
with 
the norms\footnote{in fact Hamiltonians should be considered up to scalar summands and then this is actually a norm}
\begin{equation}\label{normasr} \Vert X\Vert^\lambda_{s,r},\quad 
\| F\|^\lambda_{s,r} := \Vert X_F\Vert^\lambda_{s,r}<\infty,\quad \text{resp}\quad  \Vert X_F\Vert^{T}_{\overrightarrow p}<\infty  \,,
\end{equation} 
where $ \overrightarrow p= (\lambda,s,r,\vartheta,\mu,K,\mathcal O).$
\end{definition} The main properties of the majorant norm, contained in  \cite{BBP1}, Lemmata  2.10,\ 2.15,\  2.17, express the compatibility of the norm with projections and Poisson brackets. Similarly the main properties of the quasi--T\"oplitz norm may be found in \S  8.17.1 and in Proposition 9.2 of \cite{PP3}.

Since $\mathcal H^T_{\overrightarrow p}$ is a space of analytic functions it is naturally spanned by the monomials $e^{\ii \nu\cdot x}y^i z^\a \bar z^\b$.
It is natural to give degree $0$ to the angles $x$, degree $2$ to $y$ and $1$ to $w$. In this way   each element  $F\in\mathcal H^T_{\overrightarrow p}$ is expanded as $F=\sum_{j=0}^\infty F^{(j)}$. We will need the projections onto various subspaces, in particular those defined by the degree, which are all continuous w.r.t the majorant norm.
\begin{definition}
A linear operator $L$ on $\mathcal H^T_{\overrightarrow p}$ has degree $j$ if it maps elements of degree $k$ into elements of degree $k+j$ for all $k$.
\end{definition}
Notice that the degree of the composition of two linear operators of degree  $i,j$ is their sum $i+j$.
\begin{remark}\label{graa}
If $A\in \mathcal H_{s,r}$ has degree $j$ the linear operator ad$(A)$ has degree $j-2$. By the Cauchy estimates (Lemma 2.17 of \cite{BBP1}) this is continuous as operator $\mathcal H_{s,r}\to \mathcal H_{s',r'}$ with $s'<s$ and $r'<r$.
\end{remark}  
\begin{definition}\label{puzza}
The {\em normal form} $\mathcal N$   collects all the terms of $H_{Birk}$ of degree $\leq 2$ (dropping the constant terms). We then set $P:= H-\mathcal N$ or $H=\mathcal N+P$.
\end{definition}
For a basic {\em finiteness property} of the normal form we need the following:
\begin{definition}\label{lins}
A {\em   linear stratification} of $\mathbb R^n$  is a finite decomposition $\mathbb R^n=\bigcup_j Y_j$  where the closure $\bar Y_j$ of each $Y_j$ is a linear affine subspace, and $Y_j$ is obtained from $\bar Y_j$ by removing a finite number of proper linear affine subspaces.  

A  given stratum $Y$ lies in a minimal affine space $\bar Y$ and the group $T_Y$ of   translations of $\bar Y$ is the {\em  group of translations  of $Y$.}

\end{definition} 
\begin{remark}\label{lst} A linear stratification is obtained by choosing a finite list of linear affine  subspaces $A_i$.  Then to each $A_i$ is associated, as stratum, the set of points in $A_i$ which do not lie in any of the affine spaces $A_j$ of the list which do  not contain $A_i$.

\end{remark}
Thus $Y$ is obtained  from any of its elements $y$  as $y+T_Y$ and then removing the lower dimensional strata.
\begin{remark}\label{remark}
 
 If each $Y_i$ is defined by linear equations over $\mathbb Z$ we speak of an {\em integral linear stratification},  this induces by intersection a linear stratification of  $\mathbb Z^n$. 
 
 The integral points satisfying $\ell$  independent linear equations with coefficients in $\Z$  form a subgroup  $\Gamma$ of $\Z^d$  isomorphic  to  $\Z^{d-\ell}$ (with an integral basis which can be extended to a basis of $\Z^d$).  
 Correspondingly  the intersection of an affine space  $\bar Y$ (defined over $\Z$)  of dimension $k=d-\ell$ with $\Z^d$  is of the form  $\Gamma+u$  where $u\in  \Z^d$  and $\Gamma\subset \Z^d$ is a group of integral translations  isomorphic to $\Z^k$ (usually $k$ is called {\em rank}).
 Each stratum   has thus a rank and is obtained from such a translate by removing a finite number of translates of subgroups of strictly lower rank.
\end{remark}
The linear stratifications  appearing in this paper come from  two  sources, the   theory of graphs and  block diagonal structure of the Normal form and the Theory of cuts, see \S\, \ref{basst} and \ref{bycuts}.
\section{Properties of the normal form}
Here we state the main properties used in the KAM algorithm.   These properties are detailed and proved in  \S \ref{NoF}.  

We need to explain the following notation and fact.  First given a sign  $\s=\pm 1$ and a  variable  $z_k$ we write  $z_k^\s=z_k$ if $\s=1$ and $\bar z_k$ if $\s=-1$. 

Denote by $\mathcal F$ the space of Hamiltonians of degree $\leq 2$ (which form a Lie algebra under Poisson bracket).  In our treatment appear certain {\em blocks  $\dodo_\et$} indexed by a set $\mathfrak T$ divided into an infinite set $\mathfrak T_s$ and a finite set $\mathfrak T_f$. For each $\et\in  \mathfrak T $  we have a  finite set of indices $k\in S^c$ and, for the finitely many elements in $\mathfrak T_f$ also a sign function $s(k)=\pm1$ for $k\in \dodo_\et$. When we have such a sign we need to divide the corresponding variables $\{z_k,\bar z_k\}_{k\in \dodo_\et}$   in the set $z_k^{s(k)}$ and its conjugate.   The variables span a symplectic space and the two sets  each span a lagrangian subspace. 
\begin{definition}\label{lagr}
 An $x$ independent quadratic Hamiltonian on the space $\{z_k,\bar z_k\}_{k\in \dodo_\et}$ will  be called {\em Lagrangian} if under Poisson bracket it preserves  these two  spaces, this means that we never have in its expression a term of type  $z_hz_k$ or $\bar z_h\bar z_k$ (with $h,k\in \dodo_\et$) if  $s(k)=s(h)$  or $z_h\bar z_k$  if $s(k)=-s(h)$. In matrix terms this is the standard embedding of $m\times m$ matrices into the Lie algebra of the symplectic group of dimension $2m$.
\end{definition}
In fact, by Formula \eqref{turandot},  this  Lagrangian structure is equivalent to conservation of momentum (or of quadratic energy). We then  have the main properties of the normal form summarized as follows:
\begin{proposition}\label{teo2} For any {\em generic} choice of the tangential sites $ S=\{\mathtt j_1,\dots,\mathtt j_n\}$, there exists a  homogeneous algebraic hypersurface $\mathfrak A$, whose complement in $\R_+^n$ is  union of simply connected open regions  $\mathcal R_\al$ 
with the property:
\smallskip

For each $\al$    there is an analytic family of symplectic changes of variables
$ \mathcal R_\al\times D(s,r) \mapsto B_{\epsilon_0}$ which conjugates the NLS Hamiltonian to the following form $H=\mathcal N^{(s)}+\mathcal N^{\rm nil}+P$ where:
 \begin{eqnarray}\label{prinor}
\mathcal N^{(s)}&:=&  \omega(\xi)\cdot y + \sum_{\et\in \mathfrak T_s}\Omega_\et(\xi)\sum_{k\in \dodo_\et}|z_k|^2+ \sum_{\et\in \mathfrak T_f}\Omega_\et(\xi)\sum_{k\in \dodo_\et}s(k)|z_k|^2\,,\nonumber\\  \mathcal N^{\rm nil}&=& \sum_{\et\in \mathfrak T_f}\mathcal Q_\et^{\rm nil}.
\end{eqnarray}
The set $\mathfrak T_s$ is a denumerable index set while  $\mathfrak T_f$ is {\em finite}. Each index $\et$ is identified by  a  point  $\er_\et\in S^c $ and an algebraic function $\theta_\et$, homogeneous of degree $q$, chosen   from a finite list $\Upsilon$  (see Formula \eqref{Tht}). To $\et$ is associated    a finite set $\dodo_\et\subset S^c$ . 

The   cardinality  $d_\et$ of the set $\dodo_\et$ is  $d_\et\leq 2d+1$ for  $\et \in \mathfrak T_f$ and $d_\et\leq d+1$ for  $\et \in \mathfrak T_f$ (note that one may have that $\er_\et\notin \dodo_\et$). 

Moreover  the sets $\dodo_\et$ give a disjoint decomposition of  the normal  sites $S^c=\Z^d\setminus S$.   
\smallskip

\noindent i) {\bf Non-degeneracy} We have $\ome(\xi)=\mathtt j^{(2)}+\ome^{(1)}(\xi)$ where  $\omega^{(1)}(\xi)$  is homogeneous of degree $q$ in the variables $\xi$. The map $(\xi_1,\ldots,\xi_m)\mapsto (\ome_1(\xi),\ldots,\ome_m(\xi))$ is an algebraic local diffeomorphism for  $\xi$ outside some real algebraic hypersurface.\smallskip

\noindent ii) {\bf Asymptotic of the normal frequencies:} We have $$\Omega_\et(\xi)= |\er_\et|^2 +\theta_\et(\xi)$$ where all the functions $\theta_\et$ are chosen from the finite list $\Upsilon$.

Finally $s(k)=\pm 1$ while  $\mathcal Q^{\rm nil}_\et$ is a Lagrangian  (see Definition \ref{lagr}) nilpotent quadratic Hamiltonian in the variables $\{z_k,\bar z_k\}_{k\in \dodo_\et}$ independent of $x$.

\noindent iii) {\bf Translation invariance}:
The infinite set $\mathfrak T_s$ decomposes into finitely many components  $\mathfrak T_s(i)$  with the following property.

For all  $\et\in \mathfrak T_s(i)$ the dimension $d_\et$ is the same.

For each given $\mathfrak T_s(i)$,    the elements $\er_{\et_2 }-\er_{\et_1},\  \et_1,\et_2\in \mathfrak T_s(i)$ generate    a {\em subgroup  of translations}  $\Gamma_i\subset  \Z^d$ of   rank $k_i= d-d_\et+1, $ (that is   isomorphic to some $\Z^{k_i},\ k_i\leq d$).  Choosing an element $ {\et_0}\in \mathfrak T_s(i)$ gives  a subset  $\Gamma_i^0\subset  \Gamma_i $ obtained from the group $\Gamma_i $ by removing finitely many translates of subgroups of lower rank so that \begin{align}\label{trasl}\nonumber
\forall \et_1,\et_2\in \mathfrak T_s(i),\ \dodo_{\et_2}=&\dodo_{\et_1}+\er_{\et_2 }-\er_{\et_1},\   \er_{\et_2 }-\er_{\et_1}\in \Gamma_i \\
\bigcup_{\et\in\mathfrak T_s(i)} \dodo_{\et}=&\bigcup_{a\in \Gamma_i^0}   \dodo_{\et_0}+a.
\end{align}
\noindent iv) {\bf Affine structure}:
 Formula \eqref{trasl}  determines an  integral linear stratification $\Sigma_0$ of $\Z^d$  such that the set of roots  $\er_\et$  is a union of strata.  

We take all points in some $\dodo_{\et}, \et\in \mathfrak T_f$ to be zero dimensional strata.

Then  for each $i$, the set $\bigcup_{\et\in\mathfrak T_s(i)} \dodo_{\et}$ is the union of the parallel strata $\Gamma_i^0+k, \ k\in  \dodo_{\et_0}$.   

\smallskip
\noindent  v) {\bf Piecewise T\"oplitz:} The   function  $\theta_\et(\xi) $  is  the same for all $ \et\in\mathfrak T_s(i) $.

\smallskip
\noindent  vi) {\bf Constants of motion:} in the new coordinates, mass and momentum are:
\begin{equation}\label{turandot}
\mathbb L= \sum_i  y_i + \sum_k   s(k) |z_k|^2\,,\quad \mathbb M=\sum_i  y_i \mathtt j_i+\sum_{\mathfrak t\in \mathfrak T} \er_{\mathfrak t}(\sum_{k\in \dodo_\et} s(k)  |z_k|^2)\,,
\end{equation}
moreover $\mathcal N$ has as further constant of motion {\em the quadratic energy}:
\begin{equation}\label{turandot2}
 \mathbb K= \sum_i |\mathtt j_i|^2 y_i+  \sum_{\mathfrak t\in \mathfrak T} |\er_{\mathfrak t}|^2(\sum_{k\in \dodo_\et} s(k)  |z_k|^2).
\end{equation}
\smallskip

\noindent vii)   {\bf Smallness:} Consider any compact domain $\mathcal O$ in $\mathcal R_\al\cap \e^2\Lambda$.   If $\e^3<r<c_1\e$,  the perturbation $ P $ in the new variables is  small, more precisely 
 we have   the bounds:
\begin{equation}\label{pertu}
\Vert X_{ P}\Vert^{T}_{\overrightarrow p}\leq C (\e^{2q-1} r + \e^{2q+3} r^{-1}) \,, 
\end{equation}  where $\overrightarrow p=(\lambda, s,r,\vartheta_0,\mu_0,N_0,\mathcal O)$ with $\vartheta_0=1$, $\mu_0=2$ and $N_0= 8 (d+1)! \kappa^{d+1}$ (see \eqref{bapa} for $\kappa$)  while $C$ is independent of $s,r$ and depends on $\e,\lambda$ only through $\lambda/\e^{2}$. 
\end{proposition}
\begin{proof}
Most of these statements are part of  Proposition \ref{ilcambio} and Corollary \ref{diagp}. The smallness condition in the $\lambda$ norm is the content of Theorem 1 of \cite{PP}, the corresponding estimate for the quasi--T\"oplitz norm follows  verbatim from part 4 of \cite{PP3} extending Proposition 11.19. \qed\end{proof}
\begin{remark}\label{generico}
Here {\em generic} means that the list of excited frequencies $S$, thought of  as a vector in $\mathbb Z^{dn}$,  is not a solution of a  (complicated) polynomial $P$ (in $dn$ variables) with integer coefficients, the polynomial $P$ is the product of explicit polynomials associated to a finite list of graphs  which represent resonances to be avoided.  Of course this list grows exponentially with the dimension and with $q$ so, although the polynomial is described explicitly, it is not  possible to write down the polynomial  in a file. Nevertheless the {\em density} of generic choices  clearly tends to 1 (cf. \cite{PP3}).
\end{remark}
\begin{remark}\label{connessione}
The index set $\mathfrak T$ and the corresponding decomposition of the space as well as the affine structure depend only on the chosen connected component.

In the case of the cubic  the picture simplifies drastically since there is no nilpotent part and moreover,  all the sets $\dodo_\et$ reduce to one element. This depends on the fact that the normal form is diagonalizable with distinct eigenvalues, while here we have to take into account the multiplicities.
\end{remark}
\begin{proposition}[Melnikov conditions]\label{melone}
For any $\alpha$ and for all $\xi\in \mathcal R_\alpha$, the kernel of ad\,$\mathcal N^{(s)}$ in the space  $\mathcal F$ is given by the subspace $\mathcal F_{{\rm ker}}$, of hamiltonians of the form
$$
 \psi(\xi)\cdot y + \sum_{\et\in \mathfrak T}Q_\et(\xi,w)
$$
where $\psi(\xi)\in \R^n$ and $Q_\et(\xi,w)$ are $x$ independent Lagrangian quadratic Hamiltonians depending only on the variables $\{z_k,\bar z_k\}_{k\in \dodo_\et}$.
\end{proposition}
\begin{proof}
The proof is in Proposition \ref{secM}.
\qed\end{proof}
This proposition induces a decomposition\footnote{Roughly speaking this is just dividing a space of matrices into   some block diagonal subalgebra and the stable off-diagonal part.} of  $\mathcal F$ into:
\begin{definition}\label{kerrg}
The subalgebra $\mathcal F_{{\rm ker}}$ and its unique  complement $\mathcal F_{{\rm rg}}$ stable under adjoint action of $\mathcal F_{{\rm ker}}$.  
\end{definition}

\begin{definition}\label{doppia}  We consider the free abelian group $\mathbb Z^n$ with  canonical basis  $e_i$ and define the following linear maps $\eta: \mathbb Z^n\to \mathbb Z$, $\pi:\mathbb Z^n\to  \mathbb Z^d$, $\pi^{(2)}:\mathbb Z^n\to  \mathbb Z$
\begin{equation}\label{folm}
 \eta: e_i\mapsto 1,\quad \quad \pi:e_i\mapsto \mathtt j_i,\quad \pi^{(2)}:e_i\mapsto  |\mathtt j_i|^2.
\end{equation}
\end{definition} {\bf Warning}\quad In $\Z^n$ we always use as norm $|l|$ the $L^1$ norm $\sum_{i=1}^n|l^{(i)}|$. On the other hand  in $\Z^d$, and hence in $S^c$, we use the euclidean $L^2$ norm.
\medskip

The space $\mathcal F$ decomposes as a direct sum of its parts of degree $0,1,2$  respectively with basis
$$0)\quad e^{-\ii  \nu\cdot x} ,\ \quad 1)\quad e^{-\ii  \nu\cdot x}z^\s_k,\ \quad 2)\quad e^{-\ii  \nu\cdot x}z_h^{\s_1}  z_k^{\s_2}, \quad e^{-\ii  \nu\cdot x}y_i  $$ satisfying mass and momentum conservation, deduced from  Formula \eqref{turandot}. We usually denote  $z^+=z,\ z^-=\bar z$  for convenience.  We exclude the constants. We denote this degree decomposition as $\mathcal F:=F^{0}\oplus F^{0,1}\oplus F^{ 2},\quad F^ {2}=F^{0,2}\oplus F^{1,0}$.

We now decompose $F^{0,1}$     as orthogonal sum (with respect to the symplectic form),  of subspaces (blocks)  $(\dodo_\et,\nu)^{\pm 1}$ with $\et\in\mathfrak T$ and $\nu\in\Z^n$ costrained by the conservation laws. Each  $(\dodo_\et,\nu)^{+}$ has as basis the elements  $e^{\ii   \nu\cdot x}z_k^{s(k)}$  where  $k\in \dodo_\et $.

For each index  ${\mathfrak t},\nu$ the contraints 
 \begin{equation}\label{consmo0}\pi(\nu)+\er_{\mathfrak t}=\nu\cdot\mathtt j+\er_{\mathfrak t}=0,\quad \eta(\nu)+1=0\end{equation}
  express respectively   the conservation of momentum and   mass. 

These blocks are obviously stable  under $ad(\mathcal N)$ and  on each block this action is invertible.  On the blocks $(\dodo_\et,\nu)^{\s}$ $ad(\mathcal N)$ is self adjoint for the standard form. If $\et\in\mathfrak T_s$ the standard form is positive definite if $\s=1$ negative if $\s=-1$.
\smallskip

The previous decomposition of     $F^{0,1}$  induces a block decomposition of   $F^{0,2}$.
The product of two blocks $(\dodo_{\mathfrak t_i},\nu_i)^{\pm 1},\ i=1,2$  produces a quadratic block in $F^{0,2}$ stable under  $ad(\mathcal N)$ with basis  the products of basis  elements. In order to avoid repetitions we index different quadratic blocks by $(\nu,\et_1,\et_2)_{\s_1,\s_2}$  where the conservation laws are:
 \begin{equation}\label{consmo}
\eta(\nu)+1+\s_1\s_2=0\,,\quad \pi(\nu)+\er_{\mathfrak t_1}+ \s_1\s_2 \er_{\et_2}=0\,
\end{equation}
 while the basis elements are  $e^{\ii  \s_1 \nu\cdot x}z_k^{\s_1s(k)}z_h^{\s_2s(h)}$ where  $k\in \dodo_{\et_1} $ and  $h\in \dodo_{\et_2} $.
For $\s_1\s_2$ fixed the blocks come in conjugate pairs  and we usually exhibit the one with $\s_1=1$.

\noindent The action of $ad(\mathcal N)$ is invertible on all the blocks different from $(0,\et,\et,\s,-\s)$ which add up to the part of $\mathcal F_{{\rm ker}}$ in $F^{(0,2)}$.

\section{A KAM theorem\label{aKT}}
Having now prepared our normal form we need to explain how to perform a successful  KAM algorithm.
 
 We follow the same path as in \cite{PP3} but with some differences due to the fact that our normal form is no more diagonal with different eigenvalues but rather block diagonal with  uniformly bounded blocks corresponding to different eigenvalues.  Recall that in \cite{PP3} the KAM algorithm  requires  several steps which correspond to different sections of the paper in particular Parts 2,3.
 
\noindent Part 2 of \cite{PP3} is the part on {Quasi  T\"oplitz functions} can be transported verbatim to our setting.  More precisely
Section 7 on  optimal presentations, cuts and good points  requires only the existence,  for all $N$ large, of a stratification  of 
$\mathbb Z^d$  which refines the one of Proposition \ref{teo2}  items {\it iii)} and {\it iv)}  so it needs no changes.
 Finally  Section 8 on quasi- T\"oplitz functions  is a general discussion which holds in our setting. For completeness we give a very brief overview of the relevant definitions and properties in Appendix \ref{prilone}

 \noindent Part 3 of \cite{PP3}  needs some changes  since the Hamiltonians involved in the present paper do not satisfy all the axioms of the {\em compatible Hamiltonians} introduced there (they are only block diagonal).  The main issue is  how to perform the measure estimates. We discuss the needed changes in detail in the following subsections.
 
\noindent Finally Part 4 of \cite{PP3} was dedicated to proving that  the NLS hamiltonian is quasi--T\"oplitz and that the various changes of variables necessary to obtain  the Normal form Hamiltonian \eqref{prinor} preserves the quasi--T\"oplitz property. Since the changes of variables that we are using are exactly of the same type as those of \cite{PP3} also this section holds verbatim and is only needed in order to prove item {\it vii)}  of Proposition \ref{teo2}.     
 
 \medskip
 
 The starting point for our KAM Theorem is a class of Hamiltonians $H$ defined in $ D(s,r) \times  \mathcal O$, where we take $\mathcal O\subseteq \ro \Lambda
 $   a compact domain of diameter of order $\varepsilon^2$ contained in one of the components $\mathcal R_\alpha\cap \e^2\Lambda$ where the local diffeomorphism $\ome(\xi)$ is also injective.  
 This class of Hamiltonians is obtained as a small correction of the class described in Proposition \ref{teo2} from which we maintain all the notations.
  \begin{equation}\label{hamH}   H:= N +P,\quad  N:=  \omega(\xi)\cdot y +\sum_{\mathfrak t\in \mathfrak T_b\cup \mathfrak T_g} \mathcal Q_{\mathfrak t}(\xi,z,\bar z)\,, \quad P=
 P(x,y,z,\bar z, \xi).\end{equation} 
Here $\omega(\xi)= \mathtt j^{(2)}+ \omega^{(1)}(\xi)+ \tilde \omega(\xi)$ where $\mathtt j^{(2)},\omega_1$ are the ones of the NLS normal form, defined in Proposition \ref{teo2} item {\it ii)},  while $\tilde\omega(\xi)$ is  small of the order of Formula \eqref{pertu}.
The finite (at most $2d+1$ degrees of freedom) dimensional quadratic Hamiltonians $\mathcal Q_{\mathfrak t}$ depend on the variables $z_k,\bar z_k$ with $k \in \dodo_\et$. By the lagrangian structure of the blocks $(\dodo_\et,\nu)^\pm$ we represent these Hamiltonians by the  matrices, denoted by $\Omega_\et(\xi)$, acting on the blocks $(\dodo_\et,\nu)^+$. Note that these are independent of $\nu$. We have that for all $\et\in \mathfrak T_g$ the matrix $\Omega_\et(\xi)$ is self-adjoint (the basis is orthonormal for the standard form).  
  For all $\et\in \mathfrak T$ we have:\begin{equation}\label{asymp1} \Omega _\et(\xi)=(|\er_\et|^2+\theta_\et(\xi))I_\et+ \Omega^{{\rm nil}}_\et+ \tilde
\Omega _\et(\xi).
\end{equation} 
where $\er_\et$ is the {\em root}  of the stratum indexed by $\et\in \mathfrak T$. $\Omega^{\rm nil}_\et$ is a nilpotent matrix chosen from a finite list of matrices depending on $\xi$ in an analytic way and homogeneous of degree $q$, similarly for $\theta_\et$. Moreover the functions $\theta_\et$ satisfy the {\em piecewise T\"oplitz} property (iv) of Proposition \ref{teo2}.   Finally the  $\tilde
\Omega _\et(\xi)$ define a quadratic Hamiltonian which is quasi T\"oplitz  and  small of the order of Formula \eqref{pertu}. The same properties hold for the perturbation $P$.
\smallskip 

It is well known that the Hamiltonian equations of motion
for the normal form $ N$
admit the special solutions  $(x, 0, 0, 0)\to (x+\omega(\xi) t,
0,0,0)$, that  correspond  to
  invariant tori in the phase space, for each $\xi\in \mathcal O$. These special solutions are solutions also of the perturbed system if and only if  the Hamiltonian vector field associated to $P$ vanishes on these  invariant tori.
\smallskip  

Our aim is to prove that,  under suitable hypotheses, there is a set   $\mathcal O_\infty\subset \mathcal O$   of positive Lebesgue measure, so that,   for all
 $\xi \in \mathcal O_\infty$  the Hamiltonians $H$ still admit
 invariant tori (close to the ones of the unperturbed system)  given  (in the new variables) by the equations $y=z=0$. 
  Moreover, the associated Hamiltonian vector field  $X_H$  restricted to these tori is $\sum_{i=1}^n\omega^\infty_i(\xi)\partial_{x_i}$ while $X_H$  linearised at a torus is block-diagonal in the normal variables with $x$-independent block matrices $\Omega^\infty_\et(\xi)$. 
 \begin{definition}[reducible KAM tori] \label{reducible}
 A quasi-periodic solution of frequency $\omega^\infty$  is a  reducible KAM torus if  (in appropriate sympectic variables) it is expressed  by   the equations $y=z=0$ and moreover the associated Hamiltonian vector field  $X_H$  restricted to the torus  is $\sum_{i=1}^n\omega^\infty_i(\xi)\partial_{x_i}$,  while $X_H$  linearized at the torus is block-diagonal in the normal variables with $x$-independent block matrices  $\Omega^\infty_\et(\xi)$ of uniformly bounded dimension.
\end{definition}
\subsection{The algebraic algorithm}  Let us explain the algebraic part of  the  algorithm. We have seen that the space  $\mathcal F:=F^{0}\oplus F^{0,1}\oplus F^{ 2},\quad F^ {2}=F^{0,2}\oplus F^{1,0}$ can be canonically decomposed in an orthogonal sum of two parts,    $\mathcal F=\mathcal F_{{\rm rg}}\oplus \mathcal F_{{\rm ker}}$  where, for generic values of $\xi$ the space $\mathcal F_{{\rm rg}}$ is the range  of  the normal form operator $ad(\mathcal N^s)$  and $\mathcal F_{{\rm ker}}$ the Kernel.  Denote by  $\Pi_{{\rm rg}},\ \Pi_{{\rm ker}}$ the corresponding projections.   

Recall the degree projections $\Pi^{\leq j},\, \Pi^{  j},\, \Pi^{> j}$ onto polynomials of degree $\leq j,\ j,  >j$, we also denote by $A^{\leq j}:=  \Pi^{\leq j} A $ for any Hamiltonian $A$. We shall also need to perform a {\em ultraviolet cut}, that is  separate the subspace where the frequency $\nu$ is bounded by some $|\nu|\leq K$ and the rest, these projections will be denoted by putting $K$ as pedex as in $\Pi_{{\rm rg},\leq K},\ \Pi_{{\rm ker},\leq K},\ \Pi_{\leq K},\ \Pi_{>K}$ etc..
 
Up to now we have considered $\mathcal F$ as an algebraic object by describing a basis  but we will soon need to consider infinite linear combinations of these basis elements, defining regular quadratic Hamiltonians on some of the regions $D(r,s)$, moreover these will depend on parameters $\xi$ on suitable compact domains, nevertheless the projections still make sense.

In fact it is convenient to define  $\Pi_{{\rm rg}},\ \Pi_{{\rm ker}}$ on the entire space of  series of which  $\mathcal F$ is the part of degree $\leq 2$.   
We hence decompose  \begin{equation}\label{hrs}
\mathcal H_{s,r}= \mathcal F_{{\rm ker}}\oplus  \mathcal F_{{\rm rg}}\oplus \mathcal H_{s,r}^{>2}.
\end{equation}\begin{remark}\label{theg}[The goal]
By definition, the normal form $\mathcal N$ of Proposition \ref{teo2} is in $\mathcal F_{{\rm ker}}$.
In general, the condition for  a hamiltonian $  H=  N+  P,\   N=\Pi_{{\rm ker}}   H$ to have  KAM tori is $\Pi_{{\rm rg}}   H=\Pi_{{\rm rg}}   P=0$. So our goal is to find a
  symplectic transformation $  \mathfrak  S$ so that $$ \Pi_{{\rm rg}}      \mathfrak S(  H)=0. $$\end{remark}
The strategy is to construct this as a limit of   a {\em quadratic Nash-Moser algorithm}.  

{\bf Warning}  We deviate from  the standard notation  to put in evidence the intrinsic decomposition  of $\mathcal H_{r,s}$ given by Formula  \eqref{hrs}. In particular for a Hamiltonian $  H\in \mathcal H_{r,s}$  we will  denote:  
\begin{equation}\label{notazioni}
  N:=\Pi_{{\rm ker}}(  H)\,,\quad 
  P_{{\rm rg}}:=\Pi_{{\rm rg}}(  H)\,,\quad   P^{>2}:=\Pi^{>2}(  H)\,,
\end{equation}
so that $ H=   N+   P_{{\rm rg}} +  P^{>2}$.
 The first two terms correspond to the canonical decomposition of the Lie algebra (under Poisson bracket), $\mathcal F$ into the subalgebra   $\mathcal F_{{\rm ker}}$ and its complement    $\mathcal F_{{\rm rg}}$, which is stable under the action of   $\mathcal F_{{\rm ker}}$. \smallskip

By Remark \ref{graa} if
 $A$ has degree $i$  we have that $ad(A)$  has degree  $i-2$  so as soon as $i\geq 3$ these  operators have positive degree and are nilpotent on $\mathcal F$, the cube is always 0.

We start  with $  H_{0}:=   N_0 +   P_{{{\rm rg}}, 0} +   P^{>2}_0$, where $  N_0$ is close to $\mathcal N$ and $   P_{{\rm rg}, 0}$ is appropriately small. We wish to find a convergent sequence of changes of variables,   dependent on a sequence $K_m$ of ultraviolet cuts,
$   \mathfrak S_{m+1}:= e^{ad F_{m }}    \mathfrak S_{m }$ so that at each step $  H_{m+1}=   \mathfrak S_{m+1}(  H_{0})$ 
is such that $   N_{m }$ stays close to $\mathcal N$,  $  P^{>2}_m$ stays bounded while $   P_{{\rm rg}, m}$ converges to zero (super--exponentially). 

At a purely formal level we would like that  $   P_{{\rm rg}, m+1}$ is 
{\em quadratic} w.r.t. $   P_{{\rm rg}, m }$.
The generating function $F_{m}\in \mathcal F_{{\rm rg},\leq K_{m+1}}$ is fixed by solving the homological equation
 \begin{equation}\label{hoeq}
\{   N_{m},F_{m}\}+ \Pi_{{\rm rg},m} \{P^{>2}_{m},F_{m}\}=\Pi_{\leq K_{m+1}} P_{{\rm rg},m}\,,\quad  \Pi_{{{\rm rg}},m}:= \Pi_{{{\rm rg}},\leq K_{m+1}}
\end{equation}
which uniquely determines $F_{m}$ as a {\em linear} function of $  P_{{\rm rg},m}$ provided that the linear operator:
$$L_{m}:= {\rm ad}(  N_{m}) + \Pi_{{{\rm rg}},m} {\rm ad}(P^{>2}_{m})= {\rm ad}(  N_{m}) + \Pi_{{{\rm rg}},m} {\rm ad}(P^{3}_{m})+ \Pi_{{\rm rg},m} {\rm ad}(P^{4}_{m}) , $$ is invertible on $\mathcal F_{{\rm rg},\leq K_{m+1}}$ (clearly we also need some quantitative control on the inverse).   
\begin{remark}\label{nilp}
On $\mathcal F$ the operators 
$ {\rm ad}(  N_{m}) , \Pi_{{\rm rg},m} {\rm ad}(P^{3}_{m}), \Pi_{{\rm rg},m} {\rm ad}(P^{4}_{m})$ have respectively degree 0,1,2  so it should be be clear that $L_{m}$ is invertible if and only if ${\rm ad}(  N_{m})$ is invertible and in this case  one inverts $$L_{m}=  {\rm ad}(  N_{m}) \bigg(1+{\rm ad}(  N_{m}) ^{-1} \Pi_{{\rm rg},m} {\rm ad}(P^{>2}_{m})\bigg)$$ by inverting the second factor. This is  of the form $1+A$ with $A$ a sum of two linear operators of degree 1,2 respectively, so $A^3=0$ and we invert $1+A$ with the 3 term Neumann series $1-A+A^2$.
\end{remark} 
We now justify our choice  by computing one {\em KAM step}, for notational convenience we drop the pedex $m$ in $  H_{m }$ etc..  and substitute $  H_{m+1}$ with $  H_+$ etc...\,.

Let us compute $  H_+:= e^{ad F}  H$.  
First split the operator  $e^{ad(F)}=1+ad(F)+E_F$, by definition $E_F$ is quadratic in $F$ and hence quadratic in $  P_{{\rm rg}}$. Regarding the term
$$(1+ad(F))(  N+   P_{{\rm rg}}+  P^{> 2})=   N+   P_{{\rm rg}}+  P^{> 2}-\{  N+  P^{> 2},F\} +\{F,  P_{{\rm rg}}\} $$
we first notice that, since $F$ is linear w.r.t. $  P_{{\rm rg}}$  then the last summand is quadratic moreover since $F$ solves the homological equation we have
$$
  P_{{\rm rg}}-\{  N+  P^{> 2},F\}= (\Pi_{{\rm ker}}+\Pi_{>2}+ \Pi_{>K}\Pi_{{\rm rg}})\{  P^{> 2},F\}+\Pi_{>K} P_{{\rm rg}}.
$$
Then we deduce that
$$
\Pi_{{\rm ker}} e^{ad F}  H:=  N_{+}=   N+\Pi_{{\rm ker}}\{  P^{> 2},F\}+\Pi_{{\rm ker}}Q(  P_{{\rm rg}})\,,$$
  $$ \Pi_{{\rm rg}} e^{ad F}  H:=  P^{\leq 2}_{+}= \Pi_{>K} (P_{{\rm rg}}+ \Pi_{{\rm rg}}\{  P^{> 2},F\} )+\Pi_{{\rm rg}}Q(  P_{{\rm rg}})\,,
$$
$$
 \Pi_{>2} e^{ad F}  H:=   P^{>2}_+=   P^{> 2}+\Pi_{>2}\{  P^{> 2},F\}+\Pi_{>2}Q(  P_{{\rm rg}})
$$
where $Q(  P_{\rm rg})\,,$  is quadratic in $  P_{{\rm rg}}$ and collects the terms from $E_F(  H)$ and $\{F,  P_{{\rm rg}}\} $.
  \subsection{The quantitative estimates\label{quaes}}  {\em In the course of  our  computations we  will often have a statement for functions $f,g$  
of the following type:

\noindent {\em there is a constant $C$ depending only on some of the parameters $d,n,q,t$  with $|f|\leq C|g|$}  we then will  replace this statement  by the formula
$$|f|\lessdot_{d,n,q,t}|g|,\quad \text{or just}\quad  |f|\lessdot |g|.$$
Notice that the relation $|f|\lessdot |g| $ satisfies all the usual properties of the    relation $\leq$,  that is it is a partial order compatible with sums and products.
}\bigskip

 At a more quantitative level, we start by choosing a domain $\mathcal O_0\subset \mathcal R_\alpha\cap \e^2 \Lambda$ where some appropriate quantitative non-degeneracy conditions (see  Lemma \ref{2melquant}) hold. 
 
 By construction the measure of  $\mathcal O_0$ is  of the order of $\e^{2n}$. Then we choose two positive parameters $\tau,K_0$, bounded from below by some estimate depending only on $n,d,q$ and the set $\mathcal O_0$.
 Such a choice determines bounds on $\e,r$ through formula \ref{pertu} of the type $r \e^{-1}+\e^3 r^{-1} < C K_0^{-\tau}$. When these bounds hold our algorithm converges for all $\xi$ in some set whose complement, in the starting domain $\mathcal O_0$, has measure  of order $\e^{2n} K_0^{-f(\tau)}$ where $f$ is some linear increasing function of $\tau$, see Proposition \ref{misura}. 

 Each Hamiltonian $H$ is constrained by several parameters, $s,r,\Theta,L,M,a,$ $\vartheta,\mu,K$, which  control the quantitative structure. The frequency cut-off parameter $K$ grows exponentially while the other parameters shall be proved to be  {\em telescopic}. \begin{definition}\label{telpa}
We say that a positive parameter $b$ is telescopic if for each step $m$ we have $b_0/2< b_m< \frac32 b_0$ (usually $b_m$ is either an increasing or a decreasing sequence). 
\end{definition}
Note that a sufficient condition is that
\begin{equation}\label{tele}
\sum_{m=0}^\infty |b_{m+1}-b_m|\leq \frac{b_0}{2}.
\end{equation}
Finally the domain $\mathcal O$ shrinks at each step but its measure remains bounded from below.
\begin{definition}\label{coste}
The  parameters $s,r$ control the decreasing domain $D(s,r)$  of definition in the dynamical variables.

The  parameters $\vartheta,\mu,K$ control the quasi--T\"oplitz norm of $P$ and of the quadratic Hamiltonian associated to the $\tilde\Omega_\et$, and we set $\|P\|^T_{\overrightarrow p}\leq \Theta$. 

The parameters $M,L $ control  the  derivatives up to order $\ell =(2d+1)^2$ of $\omega$ and $\Omega_\et$, the fact that $\omega$ is invertible and an estimate on the dilation factor. 
$$
|\partial_\omega \xi|_\infty \leq L \e^{-2(q-1)}\,,\quad |\omega-\mathtt j^{(2)}|_\infty\,,\, |\Omega_\et -|\er_\et|^2 I_\et|_\infty\leq M \e^{2q}
$$
 \begin{equation}\label{omelip}
|\partial_\xi^\alpha \ome|_\infty+|\partial_\xi^\alpha \theta_\et|+|\partial_\xi^\alpha \Ome^{\rm nil}_\et|_\infty+|\partial_\xi^\alpha \tilde\Ome_\et|_\infty  \leq \e^{2(q-|\alpha|)}M^{|\alpha|}\,. \end{equation}  
for all $\alpha \neq 0$ such that $|\alpha|\leq \ell$.

Finally the parameter $a$ is defined by 
\begin{eqnarray}\label{piffedue}
&|&(\omega\cdot \nu)^{-1}|,\leq \e^{-2q} a\,,\quad | (\omega\cdot \nu + \Omega_\et)^{-1}|\leq \e^{-2q} a\, ,\nonumber\\ & |& (\omega\cdot \nu +  L(\Omega_\et)+\s\s' R (\Omega_{\et'}))^{-1}|\leq \e^{-2q} a \,,\ \forall \nu:\, |\nu|\leq S_0
\end{eqnarray}
where $S_0$ is some fixed positive number, see Lemma \ref{2melquant}. Here $\Omega_\et$ is defined by \eqref{asymp1}  and  $L(\Omega_\et)$, resp. $R (\Omega_{\et'})$, are the operators of left respectively right multiplication by the matrices  $ \Omega_\et,$ resp. $\Omega_{\et'} $, on $d_{\et}\times d_{\et'}$ matrices.  Finally $|\cdot |$ is the $L^2$ operator norm.
\end{definition} 

We choose as in \cite{PP3}:
\begin{align}\label{srK}
 r_{m+1}= (1-2^{-m-3})r_m\,,\quad &s_{m+1}= (1-2^{-m-3})s_m\,,\quad \nonumber \\
\vartheta_{m+1}=\vartheta_m + \vartheta_0 2^{-m-2},\quad  &\mu_{m+1}=\mu_m - \mu_0 2^{-m-2}\,,\quad K_m=4^m K_0 .
\end{align}

This choice allows us to perform the Cauchy and ultraviolet estimates.

We choose any sequence of nested domains $\mathcal O_m$  as follows\footnote{ of course we will need to show that we can make non-empty choices}. The domain $\mathcal O_{m+1}$ is a subset of $\mathcal O_{m}$ where not only \eqref{hoeq} admits a unique solution but moreover \begin{equation}\label{stF}
\|F_{m}\|^{T}_{\overrightarrow { p}^{'}_m} \leq \e^{-2q}K_{m}^{ \tau }\|P_{{\rm rg}}\|^{T}_{\overrightarrow p_m }\,,\end{equation}  where $\overrightarrow p_m=(\l_m,s_m,r_m,\vartheta_n,\mu_m,\mathcal O_{m})$, $\overrightarrow { p}_m'=(\l_m,s'_m,r'_m,\vartheta'_m,\mu'_m,\mathcal O_{m+1})$ with
$b'_m= (b_m+b_{m+1})/2$, for $b=s,r,\vartheta,\mu$, and  
 $\lambda_m= \e^2M_m^{-1}$.


\begin{proposition}\label{kamforma}
For $K_0,\tau$ large and for all $\e,r$ such that $r \e^{-1}+\e^3 r^{-1} \lessdot K_0^{-\tau}$, at each step $m$, in the set $\mathcal O_m$, one has the estimate 
\begin{equation}\label{superex}
\e^{-2q}\|P_{{{\rm rg}},m}\|^{T}_{\overrightarrow p_m} \leq e^{-\frac32^m}\,.
\end{equation}  Moreover all the constants $s,r,\Theta,L,M,a,\vartheta,\mu$ are telescopic. Finally our algebraic algorithm converges on the set $\cap_m\mathcal O_m$, and we obtain a Hamiltonian $H_\infty$ which has reducible KAM tori.
\end{proposition}
\begin{proof} These estimates are  all standard and  follow the same lines as the corresponding ones in  \cite{PP3}.
The procedure is recursive, suppose we have reached some step $m$ and computed all the necessary estimates, then using for instance Lemma 10.20 of \cite{PP3}, one proves  the estimate \eqref{superex} for $F_m$.

Then $F_{m }$ defines a symplectic change of variables and one proves, see Proposition 10.16 and section 10.17   of \cite{PP3}, that in the new Hamiltonian $H_{m+1}$ both the  perturbation  $P_{m+1}$ and  the quadratic Hamiltonian associated to the $\tilde\Omega^{(m+1)}_\et$ are well defined in the domain $D(s_{m+1},r_{m+1})$ and quasi--T\"oplitz with parameters $(K_{m+1},\vartheta_{m+1},\mu_{m+1})$. 
Note that by definition of $E_F$ and since $F$ solves the Homological equation we have that  $Q(  P_{{\rm rg},m})$ is quasi-T\"oplitz.
The variation in the parameters $s,r,\vartheta,\mu$ is due to {\em Cauchy estimates}, namely it is used in order to control the  norm of the Poisson bracket of two Hamiltonians, following Proposition 10.16 of \cite{PP3}.

The corrections of the parameters  $L,M,a,\Theta$ are obtained from the coordinate change, which is very close to the identity due to the super--exponential decay of the norm of $F$, this implies easily the telescopic  nature of the parameters used,  as $|b_{m+1}-b_m|<$ const $b_0 e^{-(3/2)^m}$ see   Lemma 10.20   of \cite{PP3}.  

%
Then one easily sees that the algorithm converges.

By the nature of the majorant norm, for all $j$,  the part  of homogeneous degree $j$ of $ H_m$ converges to the part  of homogeneous degree $j$ of $ H_\infty.$

In particular, since  $P_{{{\rm rg}},m} $ converges to 0 this implies that  $ N_m$ and $P^{>2}_m $   converge to $ N_\infty$ and $P^{>2}_\infty  $ respectively.
\qed\end{proof}

The main problem  is thus to exhibit some domain of parameters $\xi$ where the estimates \eqref{stF} hold and prove that at the end of the algorithm we are left with a set $\mathcal O_\infty:=\cap_m\mathcal O_m$ (of Cantor type) of positive measure (in particular non--empty) of order $\e^{2n}$.   

\subsection{Measure estimates}
For the remaining discussion  we  drop  the indices $_m$   and  assume we are working at some undefined step, with $F,H,N$ etc. the hamiltonians at this given step.
  By remark \ref{nilp} the linear operator $F\mapsto\Pi_{{\rm rg}}\{F,P^{>2}\}$ is nilpotent of order three so in order to invert the homological equation we only need to invert the operator $F\mapsto{\rm ad}( N) (F)$ which is of degree 0 hence block diagonal on the decomposition $F^{0}\oplus F^{0,1}\oplus F^{2}$.

\noindent Indeed, by the formulas \eqref{asymp1} and \eqref{hamH},  ${\rm ad}\, N$ acts on $F^0$ and on $F^{1,0}$ as $\omega\cdot\partial_x$ and
on each block $(\nu,\mathfrak t)_\s$ of $F^{0,1}$ as
\begin{equation}\label{scoccia1}
 \s( \omega\cdot \nu + \Omega_\et).
\end{equation}
Finally on each block $(\nu,\mathfrak t,\et')_{\s,\s'}$ of $F^{0,2}$, thought of as $d_{\et}\times d_{\et'}$ matrices, ${\rm ad}\, N$ acts as
\begin{equation}\label{scoccia2}
\s( \omega\cdot \nu + L(\Omega_\et)+\s\s' R (\Omega_{\et'})) 
\end{equation}
where the notations are those of \eqref{piffedue}. 

%
%
In each case we have to show that these linear operators (which we represent as  matrices of dimension resp. one, $d_\et$ and $d_\et d_{\et'}$) are invertible and moreover estimate the norm (we choose the $L^2$ operator norm)  of the inverse.   
This will be done, when the matrices are self--adjoint (i.e. both $\et$ and $\et'$ in $\mathfrak T_g$), by lower  bound estimates on the eigenvalues. 
 By momentum  conservation, the remaining set of matrices is finite and we estimate it by  lower  bounds on the  determinants. We impose as constraint on the parameters $\xi$  to have estimates of the form:
\begin{equation}\label{piffe}
|(\omega\cdot \nu)^{-1}|,\ | (\omega\cdot \nu + \Omega_\et)^{-1}| ,\ | (\omega\cdot \nu +  L(\Omega_\et)+\s\s' R (\Omega_{\et'}))^{-1}|\leq \e^{-2q}K^{\varrho}
\end{equation}
for all $\nu$ such that $|\nu|\leq K$ and
for some ($K$ independent) value  of $ \varrho$. 
\begin{lemma}\label{scelta}
 Fixing $\varrho=\varrho(\tau):=(\tau-1)/3((2d+1)^2+4)$, in the set of the  $\xi$ for which  \eqref{piffe} 
holds  one has the estimate \eqref{stF}.
\end{lemma}
\begin{proof}
The result follows directly from  Lemma \ref{orrore}. We apply it  with $s,r=(s_m,r_m)$, $\delta=2^{-m-5}$, $P\rightsquigarrow P_m$, $F\rightsquigarrow F_m$ and
$K\rightsquigarrow K_{m+1}= 4^{m+1} K_0$. Then by our choices $\delta^{-2} \leq K$ (provided $K_0 \geq 2^9$).
\qed\end{proof}
\begin{definition}\label{insiemi}
We choose $\mathcal O_{m+1}$ to be the set of   $\xi\in \mathcal O_m$ for which  \eqref{piffe} 
holds with $\varrho=(\tau-1)/3((2d+1)^2+4)$, $\omega\rightsquigarrow \omega_m$ and $\Omega\rightsquigarrow \Omega_m$ for all $\et,\et'$.
\end{definition}
\begin{proposition}\label{misura}
i)\quad The measure of   $\mathcal O_m\setminus \mathcal O_{m+1}$    is bounded by $ \e^{2n} K_{m+1}^{-\frac{\tau}{b}+n+1} $ with $b= 3(4d)^{d+4}$. 

ii)\quad Provided that $\tau\geq b(n+3) $, we have  $|\mathcal O_\infty|\geq |\mathcal O_0|- \e^{2n}K_0^{-\frac{\tau}{b}+n+2},$
and thus is a parameter set of positive measure  by taking $K_0$ large. \end{proposition}
The proof of this Proposition occupies the rest of this section and concludes the proof of  Theorem \ref{ilteorema}.

\begin{proof}[Proof of Theorem \ref{ilteorema} ]\label{prilte}
By Proposition \ref{kamforma} we have reducible KAM tori for all $\xi\in \mathcal O_{\infty}=\cap_m \mathcal O_m $ for all choices of the sets $\mathcal O_m$ for which one has \eqref{stF}. By Lemma \ref{scelta} the choice done in Definition \ref{insiemi} satisfies this requirement. Then Proposition \ref{misura} ensures that $ \mathcal O_\infty$ has positive measure provided $K_0$ is large enough. \qed\end{proof}In order to prove Proposition \ref{misura} we drop the index $m$ and first give some notation. 
\begin{definition}\label{resse}[Resonant sets]
For all $|\nu|\leq K$, $\et,\et'\in \mathfrak T$  we define the {\em resonant} sets $\mathfrak R^\varrho_{\nu,\et,\et',\s,\s'}$  with $\s,\s'= 0,\pm 1$ as the sets of $\xi\in \mathcal O$ where
\begin{equation}\label{reson}
\ | (\omega\cdot \nu + \s L(\Omega_\et)+\s' R (\Omega_{\et'}))^{-1} |>\e^{ -2q}K^{\varrho}.
\end{equation}
\end{definition}Note that this formula includes all cases of  the complement of the region given by Formula \eqref{piffe}. When one (or both) of the $\sigma$ equals zero then we drop it and the corresponding symbol $\et$, for instance  $\mathfrak R^\varrho_{\nu,\et,\et',\s,0}=\mathfrak R^\varrho_{\nu,\et,\s}$.

\smallskip

\begin{proof}[Proof of Proposition \ref{misura}]
Part ii) of Proposition \ref{misura}, follows trivially from part i)  
since $K_m=  4^mK_0$ and $|\mathcal O_0|\sim \e^{2n}$. Indeed by hypothesis $-\frac{\tau}{b}+n+1\leq -2$ and hence $|\mathcal O_\infty|\geq |\mathcal O_0|- \sum_m |\mathcal O_{m-1}\setminus\! \mathcal O_{m}| \geq$
$$
\geq |\mathcal O_0|-\e^{2n} K_0^{-\frac{\tau}{b}+n+1}  \sum_m 4^{-m} \geq |\mathcal O_0|- \e^{2n}K_0^{-\frac{\tau}{b}+n+2}.
$$
 In order to prove part i), the first thing (Lemma \ref{bababa})  is to show that for all choices of $\varrho$ and for any $|\nu|<K,$ and $\mathfrak t,\et',\s,\s'$, we may impose the corresponding estimate \eqref{piffe}, by only removing, from the domain $\mathcal O$,  a   resonant  set of parameters  $\mathfrak R^\varrho_{\nu,\et,\et',\s,\s'}$ of measure $\leq C\e^{2n} K^{-c (\varrho-1)}$ for some $C,c$ depending only on $q,d,n$.

Then we evaluate the measure of the union of resonant sets.  Since $\et,\et'$ vary in an infinite set it is not sufficient to simply count the resonant sets.
In Lemma \ref{banali} we prove that for all $\varrho$ and for all  ${\nu,\et,\et',\s,\s'}$ such that $|\nu|\leq K$ and $\s\s'\neq -1$  one has that the non-empty sets $\mathfrak R^\varrho_{\nu,\et,\et',\s,\s'}$  belong to a finite list   $G_K$ with cardinality
 $|G_K|< K^{p}$. Hence we need to remove a set of measure  $\leq C\e^{2n} K^{-c (\varrho-1)}K^{p}$ where we choose $c$ to be the minimum over the possible $c$ and $C$ the maximum. 
 
 
 In the case of the second Melnikov condition with $\s\s'=-1$ we proceed in a slightly  subtler way. We use the quasi-T\"oplitz property which is defined in terms of (among others) a free parameter $\varrho_0$, then we prove, in Lemma \ref{Rcont}, that fixing $\varrho_{d+2}= (4d)^{d+2}\varrho_0$ there exists a set $R$ of measure at most $\e^{2n}K^{-\varrho_0+n}$ such that for all $\nu,\et,\et'$ : 
 $$
 \mathfrak R^{ \varrho_{d+2}}_{\nu,\et,\et',\s,-\s} \subset R.
 $$
In conclusion fixing $\varrho_0$ so that $ \varrho_{d+2}=\varrho$ we have proved that, provided $\varrho$ is sufficiently large, by removing a small measure resonant set we may impose all the Melnikov conditions \eqref{piffe}.
We then verify (Corollary \ref{meare}) that the measure is of order $\e^{2n} K^{-\varrho/(4d)^{d+2}+n+1}$. When we substitute $\varrho=\varrho(\tau)$ we easily obtain item {\it i)} of Proposition \ref{misura}.\qed\end{proof}

\begin{lemma}\label{bababa}
For $|\nu|\leq K, \et,\et'\in \mathfrak T,\s,\s'=0,\pm 1$, we have: 
\begin{equation}\label{palleres}
| \mathfrak R^\varrho_{\nu,\et,\et',\s,\s'}| \leq C \e^{2n}  K^{-c (\varrho-1)}\,,
\end{equation}
 where $c=1$ if either 
 $\s=\s'=0$ or if $\s'=0$ and $\et\in \mathfrak T_g$ or if $\et,\et'\in \mathfrak T_g$.
 
\noindent If $\s'=0$ and $\et\in \mathfrak T_f$ then  $c=d_\et ^{-1}$, otherwise $c=(d_\et d_{\et'})^{-1}$.
 \end{lemma}
 
 \begin{proof}  We give a brief overview of the proof, the details are in the appendix.
 
  The first remark is that all the resonant sets with $|\nu|\leq S_0$ are empty since   formula \eqref{piffedue} implies \eqref{piffe} provided that $K_0^\rho $ is large enough.
 
 Moreover if  $|u|:= |\omega\cdot \nu + \s|\er_\et|^2+\s'|\er_{\et'}|^2|>  2(2d+1)^2 M \e^{2q}$  then the bounds \eqref{piffe} hold trivially and the resonant sets are empty.
 
  Indeed,  by \eqref{omelip}, we have   $|\Omega_\et-|\er_\et|^2|_{\infty}\leq M\e^{2q} $ and  we have to invert some matrix of the form  $u+Z$ with $Z$ a matrix of size $\pp \leq (2d+1)^2$ deduced from $\Omega_\et$  in such a way that also $|Z|_\infty \leq   M\e^{2q} $.  Since the operator norm  of $Z$ is bounded by $\pp    M\e^{2q} $ if  $ 2|Z|_2 \leq 2\pp    M\e^{2q}\leq  |u| $ we invert in Neumann series and obtain the estimate $| (u+Z)^{-1}|_2\leq 2 |u| ^{-1} \leq$ cost $  \e^{-2q}$. 
  So we only need to consider $S_0\leq |\nu|\leq K$ and given $\nu,\et,\et',\s,\s' $ we only need to  work in the subset of $\omega$ where 
  
  \begin{equation}\label{numeratore}
|\omega\cdot \nu+\s|\er_\et|^2+\s'|\er_{\et'}|^2|< 2(2d+1)^2 M\e^{2q}.
\end{equation} 
\begin{remark}\label{dilf}
The condition $|\nu|>S_0$ is used so that one can perform the measure estimates not on the set $\mathcal O$ but on its image $\tilde{\mathcal O}$ under $\omega$, by multiplying by the dilation constant. For this,  an upper bound for the absolute value of the determinant of the Jacobian of the inverse map is estimated,  by Formula \eqref{omelip},  as  $L ^n\e^{-2n(q-1)}$.
This is justified by two  remarks, first $\tilde{\mathcal O}$ is contained in a hypercube of length $2\e^{2q}$ and second  the dilation constant  is  uniformly bounded  in both ways.

Then the basis of our measure estimates is Lemma \ref{tecnico} which shows how to control the resonant set $ |f(\omega)|\leq \alpha$ of a $C^k$ function $f$ provided that one has a uniform lower bound on some derivative $|\partial_\omega^h f|$ with $|h|\leq k$.
\end{remark}\smallskip

%
%
%
%
%

We now discuss each of the three inequalities of \eqref{piffe} separately, the technical parts are in the appendix.

 On $F^0$  (i.e. $\s=\s'=0$) the estimate is classical, see the appendix.

 \smallskip

On $F^{0,1}$ we treat separately $\et\in \mathfrak T_f$ and $\et\in \mathfrak T_g$.  In the second case, since by hypothesis  the $\Omega_\et$ are self-adjoint one uses the typical KAM estimates on eigenvalues:
\begin{equation}\label{reso}
|\omega\cdot \nu + \lambda^{(i)}_\et|\geq \e^{2q} K^{-\varrho} \,,\quad \forall \lambda^{(i)}_\et \in {\rm spec}(\Omega_\et)\,,
\end{equation}
which are equivalent to the bounds \eqref{piffe} on the operator norm of the inverse matrix.
To estimate the measure of the resonant set of this case, one just passes to the  variables $\omega$ and then shows that the derivative of the small divisor is bounded from below by $S_0/n$, see the appendix.

Regarding the $\et\in \mathfrak T_f$, we cannot bound the $L^2$ norm with the eigenvalues since they do not dominate the $L^2$ norm and moreover may not be regular functions of $\xi$.


We can obtain the  bound \eqref{piffe},    by requiring   \begin{equation}\label{reso2}
|{\rm det}(\omega\cdot \nu + \Omega_\et)|\geq  \e^{2qd_\et} K^{-\varrho+1} 
\end{equation}
here  $d_\et$ is the dimension of the matrix $\Omega_\et$.

The $L_\infty$ norm of the matrix we are inverting is estimated as $\lessdot M \e^{2q}$ by the two bounds \eqref{numeratore} and \eqref{omelip}.
By Cramer's rule,  we have that  \eqref{reso2} implies the bounds \eqref{piffe} provided $K_0$ is large enough (we use the exponent $\varrho-1$ in order to absorb the constants).  Thus the measure of $\mathfrak R_{\nu,\et,\s}$ is dominated by the measure of the resonant set where \eqref{reso2} does not hold.

In this case again we obtain the measure estimates by a lower bound on some derivative.  We need to perform $d_\et$ derivatives, thus we use the fact that all our functions are at least $C^{d_\et}$. \smallskip

We are left with the {\em second Melnikov conditions} $\s,\s'=\pm 1$.
%
If  at least one of the couple $\et,\et'$ belongs to $\mathfrak T_f$,
following the same reasoning as in the case of $\s'=0$, we impose 
\begin{equation}\label{reso3}
|{\rm det}(\omega\cdot \nu + L(\Omega_\et)+\s\s' R (\Omega_{\et'}))|\geq  \e^{2qd_\et d_{\et'}} K^{-\varrho+1} 
\end{equation} 
By Cramer's rule, since $|\ome\cdot \nu + \s|\er_\et|^2+\s'|\er_{\et'}|^2|$ is small and by the second bound on \eqref{omelip}, we have that  \eqref{reso2} implies the bounds \eqref{piffe} provided $K_0$ is large enough.  This means that the measure of $\mathcal R_{\nu,\et,\s}$ is dominated by the measure of the resonant set where \eqref{reso2} does not hold.

In this case in order to obtain the measure estimates, by a lower bound on some derivative, we need to perform $d_\et d_{\et'}$ derivatives, thus we use the fact that all our functions are at least $C^{d_\et d_{\et'}}$).

When both $\et,\et'\in \mathfrak T_g$ then the matrices we need to control are self-adjoint and we bound their inverse by controlling the eigenvalues, namely we need a condition 
\begin{eqnarray}\label{reso2mel}
& &|\omega\cdot \nu + \lambda^{(i)}_\et+ \s \lambda^{(j)}_{\et'}|\geq \e^{2q} K^{-\varrho} \,,\nonumber \\ 
& &\ \forall\; \s=\pm\,,
\ \lambda^{(i)}_\et \in {\rm spec}(\Omega_\et)\,,\lambda^{(j)}_{\et'} \in {\rm spec}(\Omega_{\et'})\,,
\end{eqnarray}
This measure estimate is done just as in formula \eqref{reso}, indeed one just passes to the  variables $\omega$ and then shows that the derivative of the small divisor is bounded from below by $S_0/n$, see the appendix.
\qed\end{proof}
\begin{lemma}\label{banali}
 For  all $\varrho$ and for all  ${\nu,\et,\et',\s,\s'}$ such that $|\nu|\leq K$ and $\s\s'\neq -1$  one has taht the non-empty sets $\mathfrak R^\varrho_{\nu,\et,\et',\s,\s'}$  belong to a finite list   $G_K$ with cardinality $|G_K|< K^{p}$ with $p=n+\frac{d}{2}+1$.
\end{lemma}
\begin{proof}
For the sets  $\mathfrak R_{\nu,\et,\s}^\varrho$ we recall the conservation of momentum  $\er_\et= -\pi(\nu)$, then a set is determined by  $|\nu|<K$ and $\s=0,\pm 1$. We have at most $3(2K)^n$ sets of this type.

In the sets $ \mathfrak R_{\nu,\et,\et',\s,\s'}^\varrho$ we preliminarily notice that  there are a finite number of cases where at least one of the pair $\et,\et'$ belongs to $\mathfrak T_f$. Indeed let us suppose that $\et\in \mathfrak T_f$. Then $|\er_\et|<C_0\ll K_0$,  $\er_{\et'}   $ is fixed by momentum conservation, $\theta_{\et'}$ is chosen in the finite list $\Upsilon$ (see formula \eqref{Tht}) and the pair identifies the index $\et'$.  Consequently we have at most $|\Upsilon|^2(2C_0)^d(2K)^n$ sets of this type.

We are left with $ \mathfrak R_{\nu,\et,\et',\s,\s'}^\varrho$ with both $\et,\et'\in \mathfrak T_g$. 
We notice that
\begin{equation}\label{ilam}
\lambda^{(i)}_\et= |\er_\et|^2+\theta_\et +\tilde\lambda^{(i)}_\et\,,
\end{equation} then if $\s'\s=+$  the left hand side of \eqref{reso2mel} is trivially $\geq 1/2$ for all $\xi\in \mathcal O$, if $|\er_{\et}|^2+|\er_{\et'}|^2 \geq  4K\geq 2|\omega||\nu|$and we only need to consider a finite number of $\et$ such that
$|\er_{\et}|^2\leq 4K$. Consequently  we have at most  $|\Upsilon|^2 (4K)^{d/2}(2K)^n$ {\em non-empty} sets of this type.\qed \end{proof}

Treating the case $\s '\s=-$ 
 is a well-known problem in the study of NLS on $\T^d$ which is overcome by exploiting  the { \em quasi--T\"oplitz} property of the NLS Hamiltonian. 

As explained in Appendix \ref{prilone}, given $K$ we can construct a stratification $  \Sigma^{(K)}$ of $\Z^d$ with the following properties. 
\begin{proposition}\label{lepro}
\begin{enumerate}\item To each stratum $\Sigma$ is associated an order parameter $\varrho_{_{\Sigma}}\geq \varrho_0$  
which can be taken from the finite list $ \varrho_{i}=(4d)^i\varrho_0,\ i=1,\ldots,d+1.$

\item For each $\et\in \mathfrak T_g$, $\dodo_\et$ belongs to a set $\Sigma_\et $ union of parallel strata each meeting $\dodo_\et$ in a single point,  all   of codimension $\ell\leq d_\et-1$ whose union is a union of  translates $\dodo_{\et'}=\dodo_{\et}+u$ and with the same cut $\varrho_{_{\Sigma}}$, moreover the eigenvalue $\theta(k),\ k\in\Sigma_\et$ depends only on $\et$. 

\item If $\ell=0$ one has  that the corresponding matrix  \eqref{scoccia2} is   invertible with the bounds \eqref{piffe}.  \item If  $\ell \neq 0$ but $  \pi(\nu)\cdot v  $ is not constant on each stratum again the corresponding matrix  \eqref{scoccia2} is   invertible with the bounds \eqref{piffe}.
\end{enumerate} 
\end{proposition}
\begin{proof}
i)  The order parameter is that of the cut  $\varrho_j$ associated to the points of $\Sigma$.  ii) Each $\dodo_\et$ determines  in the initial stratification a stratum for each  of its points  of codimension $d_\et-1$. The stratification $  \Sigma^{(K)}$ is a refinement of  he stratification $  \Sigma_0$ of Proposition \ref{teo2} so the properties follow from that. iii)  If $\ell=0$  this means that for every $v\in B_K$ we have $| v\cdot \er_\et | >4K^{4d\varrho_0}$,  in particular  we have $\pi(\nu)\in B_K$  so $ |  \pi(\nu)\cdot \er_\et   |> 4K^{4d\varrho_{0}}$. iv)  In this case  by definition we have $ |  \pi(\nu)\cdot \er_\et   |> 4K^{4d\varrho_{_\Sigma}}\geq 4K^{4d\varrho_{0}}$. So in both cases we need to observe that the matrix $(\omega\cdot \nu +   L(\Omega_\et)- R (\Omega_{\et'})$ has integral part 
 $\mathtt j^{(2)} \cdot \nu +|\er_\et|^2-|\er_{\et'}|^2=\mathtt j^{(2)} \cdot \nu - |\pi(\nu)|^2- 2\pi(\nu)\cdot \er_\et $ since   by conservation of momentum $\pi(\nu)+ \er_\et - \er_{\et'} =0 $.

 Now , if $ |  \pi(\nu)\cdot \er_\et   |> 4K^{4d\varrho_{0}}$, we have  $||\pi(\nu)|^2+  \pi(\nu)\cdot \er_\et   -\mathtt j^{(2)} \cdot \nu|\geq |  \pi(\nu)\cdot \er_\et   |-|\pi(\nu)|^2- |\mathtt j^{(2)}||\pi(\nu)|> 4K^{4d\varrho_0}-2K^2$ is a large positive integer and  since the remainder is small with $\e^{2q}M$ by formula \eqref{omelip}, we can invert the matrix for all $\xi\in \mathcal O_m$ by  Neumann series with a good estimate. 
\qed\end{proof}
For each  stratum $\Sigma$  of the stratification of $\mathfrak T_s$, given in Remark \ref{strT},  we choose  {\em one representative} $\et_\Sigma$  (by fixing $\dodo_{\et_{_\Sigma}}$).

Then for each  $|\nu|\leq K$ such that $ \pi(\nu)\cdot x$ is constant for $x$ on the stratum, we proceed as follows:
\begin{itemize}
 \item For each $\Sigma$ we determine  (if it exists) the unique  $\et'  $  so that $(\er_{\et'},\theta_{\et'})= (\er_{\et_{_\Sigma}}+\pi(\nu),\theta')  $;

\item  if $\et' $ exists we impose
$$
| (\omega\cdot \nu +  L(\Omega_{\et_\Sigma})+\s\s' R (\Omega_{\et' }))^{-1}|\leq \e^{-2q}K^{2 d \varrho_{_\Sigma}}\,.
$$
\end{itemize}
The complementary of this set is by definition $\mathfrak R_{\nu,\et_\Sigma,\et' }^{2 d \varrho_{_\Sigma}}$ and has measure of order $\e^{2n}K^{-2d \varrho_{_\Sigma}}$, by Lemma \ref{bababa}.
Now by Lemma \ref{ecchec} we have less than $K^{(2d -2)\varrho_{_\Sigma}}$ strata of this type so setting \begin{equation}\label{lR}
R:=\bigcup_{|\nu|\leq K}\bigcup_{\Sigma,\theta'} \mathfrak R_{\nu,\et_\Sigma,\et' }^{2 d \varrho_{_\Sigma}}\end{equation}
 we have
$$
\e^{-2n}K^{-n}|R| \leq  \sum_{\Sigma,\theta'} K^{-2 d\varrho_{_\Sigma}}\leq   |\Upsilon|\sum_{j=0}^{d+1} K^{-2d \varrho_{j}}K^{(2d -2)\varrho_{j}} \leq K^{-\varrho_{0}}.
$$ 
\begin{lemma}\label{Rcont}
For all $\nu,\et,\et'$ we have
 $
\mathfrak R^{ \varrho_{d+2}}_{\nu,\et,\et',\s,-\s} \subset R$.
\end{lemma}
\begin{proof}
For each stratum $\Sigma$ of $\mathfrak T_s$   we have chosen a representative $\et_\Sigma$.   Now for all $\nu $ for which a small divisor may occur we have for the scalar products $ \pi(\nu)\cdot \er_{\et_{_\Sigma}} =  \pi(\nu)\cdot \er_{\et } $ for all   $\et $ in this stratum. It then follows that $$|\Ome_{\et }- \Ome_{\et_\Sigma}|_\infty= |\tilde\Ome_{\et }- \tilde\Ome_{\et_\Sigma}|_\infty\leq \e^{2q} M K^{-4 d \varrho_{_\Sigma}},\quad\forall\et\in\Sigma$$by applying Lemma \ref{eccec0}. Consequently for each pair $\et,\et'$ with $\et\in \Sigma, \et'\in \Sigma'$  we have
$$
| (\omega\cdot \nu +  L(\Omega_{\et })- R (\Omega_{\et' }))^{-1}|= |(\omega\cdot \nu +  L(\Omega_{\et_\Sigma})- R (\Omega_{\et_{\Sigma'}}))^{-1}||(1+ A)^{-1}|
$$
with $$
A=(\omega\cdot \nu +  L(\Omega_{\et_\Sigma})- R (\Omega_{\et_{\Sigma'}}))^{-1}(\Ome_{\et }- \Ome_{\et_\Sigma})\,,\quad |A| \leq M K^{-2d \varrho_{_{\Sigma}}} 
$$
so that 
$$
| (\omega\cdot \nu +  L(\Omega_{\et })- R (\Omega_{\et' }))^{-1}| \leq  2\e^{-2q}K^{2 d \varrho_{_\Sigma}}  \leq  \e^{-2q}K^{2 d \varrho_{_\Sigma}+1}
$$
In conclusion, since $\varrho_{_\Sigma}\leq \varrho_{d+1}$, we have proved our claim.\qed\end{proof}
\begin{corollary}\label{meare} For $\varrho$ large enough, the measure of the union of all resonant sets $\mathfrak R^\varrho_{\nu,\et,\et',\s,\s'}$ with $|\nu|\leq K$ is bounded by
$$ \e^{2n}K^{-\frac{\varrho}{(4d)^{d+2}}+n+1}$$  
\end{corollary}\begin{proof}
This follows by Lemmas  \ref{bababa} \ref{banali} and \ref{Rcont}. Indeed  we fix $\varrho_{d+2}= \varrho$ so that $\varrho_0= \varrho/(4d)^{d+2}$. Then the Lemmas imply that we   need to remove a region of order $$K^{-\frac{\varrho}{(2d+1)^2}+ n +\frac d2+1} + K^{-\frac{\varrho}{(4d)^{d+2}}+n} \leq K^{-\frac{\varrho}{(4d)^{d+2}}+n+1},$$ provided $\varrho$ and $K$ are both large.
\end{proof}
 \section{The normal form\label{NoF}}
We will work with many quadratic Hamiltonians in the variables $w$ (thought of as a row vector). We  represent a quadratic Hamiltonian m $\mathcal F$ by a matrix $F$ as
 \begin{equation}\label{represQ}
\mathcal F(w)= \frac 12  w\cdot wJF^t =-\frac1 2  w F Jw^t\,,
\end{equation} where $J:=-\ii \{w^t,w\}$ is the standard matrix of the symplectic form.

\smallskip

By explicit computation, and under simple genericity conditions, we have:
\begin{equation}
\label{Pno}\mathcal N= \ome(\xi)\cdot y +\sum_{k\in S^c} |k|^2 |z_k|^2 +{\mathcal Q}(\xi;x,w) \,,\quad \ome_i(\xi):= |\mathtt j_i|^2+ \ome^{(1)}_i(\xi)
\end{equation} 
 here    ${\mathcal Q}(\xi;x,w)$ is a quadratic Hamiltonian in the variables $w$ with coefficients trigonometric polynomials in $x$ given by Formula ($30$) of \cite{PP}.
 The frequency modulation $\ome^{(1)}$ is homogeneous of degree $q$ in $\xi$ and given by the following explicit   formula.
 We introduce 
\begin{equation}\label{symme}
A_r(\xi_1,\ldots,\xi_m):= \sum_{ \sum_i k_i=r} {\binom{r}{k_1,\dots, k_m}}^2 \prod_i \xi_i^{k_i}.
\end{equation}
 and we have
 \begin{equation}\label{ome1}
 \ome^{(1)}(\xi) = \nabla_\xi  A_{q+1}(\xi)- (q+1)^2 A_q(\xi)\underline 1.
 \end{equation}
 By applying the results of \cite{PP}  we decompose this   very complicated infinite dimensional quadratic
Hamiltonian into infinitely many decoupled finite dimensional systems reduced to constant coefficients by an explicit symplectic change of variables.
 
 Since this construction is needed in the following we recall   quickly   Theorem   1   of \cite{PP}.
 \begin{theorem}\label{teo1}[Theorem   1   of \cite{PP}] For all {\em generic} choices $S=\{\mathtt j_1,\dots,\mathtt j_n\}\in\nobreak  \Z^{nd}$ of the tangential sites,    there exists   a {\em phase shift map} 
$L:S^c\to \Z^n ,\ L: k\mapsto L(k) \,,\quad |L(k)|\leq 4qd$  such that the analytic symplectic change of variables:
\begin{equation}\label{labella}
 \Psi:\ z_k= e^{-\ii   L(k)\cdot x   }z_k' ,\ y=y'+\sum_{k\in S^c}  L(k)  |z_k'|^2,\ x=x',
\end{equation} 
from $ D(s,r/2) \to D(s,r)$ 
has the property that $\mathcal N$ in the new variables:
\begin{equation}
\label{Sno}\mathcal N\circ\Psi=  \ome(\xi)\cdot y'  +\sum_{k\in S^c}\tilde\Ome_k|z'_k|^2 +\tilde {\mathcal Q}(w')\,,
\end{equation}
  has {\em constant coefficients},  where $\omega(\xi)$ is defined in \eqref{Pno} and furthermore:
  
\smallskip
\noindent i) {\bf Non-degeneracy}  The map $(\xi_1,\ldots,\xi_m)\mapsto (\ome_1(\xi),\ldots,\ome_m(\xi))$ is a local diffeomorphism for  $\xi$ outside a homogeneous real algebraic hypersurface. 
\smallskip

\noindent ii) {\bf Asymptotic of the normal frequencies:} We have $\tilde\Ome_k= |k|^2 +\sum_i |\mathtt j_i|^2 L^{(i)}(k)$.

\noindent iii) {\bf Reducibility}: The matrix $\tilde{ Q}(\xi)$  (see formula \eqref{represQ}) of the quadratic form	$\tilde{\mathcal Q}(\xi,w')$ (see formula \eqref{tildeQ})   depends  only on the variables $\xi$,  
 its entries are homogeneous of degree $q$ in these variables.  It is   block--diagonal  and satisfies  the following properties: 

\quad   All of the blocks except a finite number  are self adjoint   of dimension  $\leq d+1$.  The remaining finitely many  blocks have dimension  $\leq 2d+1$. 
 
\quad All the (infinitely many) blocks are  described, through Formula \eqref{represQ},  by a finite list of matrices $\mathcal M(\xi)$.

\noindent iv)   {\bf Smallness:} \ If $\e^3<r<c_1\e$,  the perturbation $\tilde P:= P\circ \Psi $ is  small, more precisely 
we have   the bounds:
\begin{equation}\label{pertu2}
\Vert X_{\tilde P}\Vert^\lambda_{s,r}\leq C (\e^{2q-1} r + \e^{2q+3} r^{-1}) \,, 
\end{equation}  where $C$ is independent of $r$ and depends on $\e,\lambda$ only through $\lambda/\e^2$. 
\end{theorem}
Regarding   $\omega^{(1)}$    given by Formula  \eqref{ome1}   we will need the following
\begin{lemma}\label{omest} For each $i$ the polynomial $\omega^{(1)}_i(\xi)$  equals $ -q(q+1) \xi_i^q$ plus terms which contain at least two variables.\end{lemma}\begin{proof} We do it for $i=1$ by symmetry. 
From $\nabla_\xi  A_{q+1}(\xi)$ the terms $\xi_j^q$ are obtained  by
  $$\partial_{\xi_1}(\xi_1^{q+1}+ (q+1)^2\sum_{j\neq 1}\xi_j^{q }\xi_1) =(q+1)\xi_1^{q }+  (q+1)^2\sum_{j\neq 1}\xi_j^{q }.$$  From the second summand we subtract  $(q+1)^2\sum_j\xi_j^q$  getting the desired formula.   
\qed\end{proof}

\subsection{The matrix blocks and the geometric graph $\Gamma_S$\label{gra}} 
The quadratic Hamiltonian $\tilde{\mathcal Q}$ is described, as we shall see in Formula \eqref{tildeQ},  in terms of a {\em 2--colored} marked graph $\Gamma_S$ with vertices in $\Z^d$ and labels in $\Z^n$ which encodes  the possible interactions between the normal frequencies $k\in S^c$. 
In order to describe this in a combinatorial way we recall some notation from \cite{PP}, see  Definition \ref{doppia} for the meaning of the maps $\eta,\pi,\pi^{(2)}$.
 \begin{definition}[edges] \label{edges}
  Consider a finite set $X$ of  elements in $\Z^n$.\begin{equation}
\label{glied}X :=\{\ell =\sum_{i=1}^n \ell_ie_i,\ \ell_i\in\mathbb Z \quad \ell\neq 0,-2e_i\; \forall i\,,\quad \eta(\ell)\in \{0,-2\}\}.
\end{equation}  Notice the {\em mass constraint}   $\sum_i\ell_i=\eta(\ell)\in \{0,-2\}$.  We call all these elements respectively the {\em black, $\eta(\ell)=0$} and {\em red $\eta(\ell)=- 2$}    {\em edges}  and denote them by $X^{(0)},X^{(-2)}$ respectively.

Each edge carries a {\em quadratic energy} (which is an integer): $$
K(\ell):= \frac{1+\eta(\ell)}2 (|\pi(\ell)|^2+\pi^{(2)}(\ell)),\quad  \frac{1+\eta(\ell)}2=\pm \frac12.
$$
Of particular interest for the $q$--NLS is 
 \begin{equation}
\label{glied1}X_q:=\{\ell:=\sum_{j=1}^{2q} \pm  e_{i_j}=\sum_{i=1}^m \ell_ie_i,\  \quad \ell\neq 0,-2e_i\; \forall i\,,\quad \eta(\ell)\in \{0,-2\}\}.
\end{equation}  
\end{definition}
Choose now a finite set of integral vectors $S:=\{\mathtt j_1,\ldots,\mathtt j_n\}$ in $\mathbb Z^d$. For the NLS these are the tangential sites but the construction is purely geometric.
\begin{definition}\label{pl}
Given $\ell \in X^{(0)}$  denote by $\mathcal P_\ell$  the set of pairs $h,k\in\Z^d$ satisfying:
\begin{equation}\label{fico}
\sum_{j=1}^{n}\ell_j \mathtt j_j+h-k=0,\quad  \sum_{j=1}^{n}\ell_j |\mathtt j_j|^2+|h|^2-|k|^2=0\ \quad \phantom{-}\ell\in X^{(0)}.
\end{equation}  If $\ell\in X^{(-2)}$  we denote by $\mathcal P_\ell$  the set of unordered pairs $\{h,k\},\ h,k\in\Z^d$ satisfying: \begin{equation}\label{fico1} 
\sum_{j=1}^{n}\ell_j \mathtt j_j+h+k=0,\  \quad \sum_{j=1}^{n}\ell_j |\mathtt j_j|^2+|h|^2+|k|^2=0\quad \phantom{-} \ell\in X^{(-2)}.
\end{equation} 
 It is convenient to formalise this definition in the language of graphs with two types of edges corresponding to the two conditions.
  When $h,k$ satisfy formula \eqref{fico} we join them by an oriented {\em black } edge, marked $\ell$, with source $h$ and target $k= h+\sum_{j=1}^{m}\ell_j \mathtt j_j $.  Formula \eqref{fico1} is symmetric and we join $h,k$ by an unoriented {\em red} edge marked $\ell$.
  \end{definition}

 Note that the two conditions have a geometric meaning expressed through the quadratic energy. The first means that $h $ lies on the hyperplane 
 \begin{equation}\label{lanera}
H_\ell:\quad   x\cdot \pi(\ell) =K(\ell)   \end{equation} while $k$ lies on the parallel hyperplane $H_{-\ell}$.

 The second condition means that $h,k$ are opposite points on the sphere 
 \begin{equation}\label{laros}
 S_\ell:\quad |x|^2+  x\cdot \pi(\ell) =K(\ell).
\end{equation}
Our main object of interest are the connected components of this graph, called {\em geometric blocks}.
  In \cite{PP} we have   shown that for a generic choice of $S$ the set $S$ is itself a connected component, called the {\em special component}  all other components thus decompose $S^c$.  \begin{definition}\label{grafg} For a given choice of edges $X$ and vectors $S$  we denote by $\Gamma_{S,X}$  the resulting graph.

In the NLS, when $q$ is fixed,   we denote $\Gamma_{S }:= \Gamma_{S,X_q}$
\end{definition}
\begin{remark}\label{incg} Note that if $S\subset S'$ and $X\subset X'$ we have $\Gamma_{S,X}\subset  \Gamma_{S',X'}$.

\end{remark}
\begin{proposition}\label{Gen}
For generic choices of $S$ the connected components of $ \Gamma_{S,\, X }$ have at most $2d+1$ vertices.

We have only finitely many components containing red edges,  moreover the components which do not contain any red edge have  vertices which are affinely independent (and hence at most $d+1$) and are  naturally decomposed in a finite family of subsets in each the connected components are obtained from one by translation.
\end{proposition}
\begin{proof} 
Our graph   is a subgraph of  the graph defined in  \cite{PP} relative to the edges of some $X_p$ provided that we choose $p$ large enough. Then by Theorem 3 of  \cite{PP} for generic choices of $\mathtt j_i$  we have the desired bound.
  \qed\end{proof}
 
   \smallskip
 
 In \cite{PP}, Definition 17, we have associated to each  connected component  $A$ a purely combinatorial graph  
 $\GA$  which contains only the markings of the edges   of $A$  and hence encodes the information on the equations which the vertices of $A$ must solve (equations (61) of \cite{PP}) of type as \eqref{lanera} and \eqref{laros}. We then have that there are only finitely many  such combinatorial graphs, called {\em combinatorial blocks},  while there may be infinitely many $A$ which have the same $\GA$.
In case the graph has no red edges  the equations for the vertex $x$ are all linear.
This implies that the  connected components of  $\Gamma_S$  which correspond to a given combinatorial graph with only black edges and with a chosen vertex      are all obtained from a single one by translations by vectors which are orthogonal to the edges of the graph.

 We also have chosen   a preferred  vertex, called the {\em root}, in each connected component,  in such a way that  the  roots of the translates are the translates of this root.  
 Formalizing:\begin{itemize}\item we have a map $\er:S^c\to S^c$ with image the chosen set  $S^{c,\er}$ of roots. \item  The fibers of this map are the connected components  of the graph $\Gamma_S$.  
\item When we walk from the root $\er(k)$ to $k$   the parity $\pm 1$ of the number of red edges on the path is independent of the path and we denote it by $\sigma(k)$ (the {\em color} of $k$).  
\item There are only finitely many elements $k$  with $\s(k)=-1$, the finitely many corresponding roots are exactly the roots of the components with red edges. \item  In any case the color of the root is always 1 ({\em black}).
\end{itemize}\smallskip

The vector $L(k)=\sum L_i(k) e_i$, with $k\in S^c$, appearing in Theorem \ref{teo1} is defined in Lemma 10 of \cite{PP} through the edges appearing in a path from the root to $k$. It depends only on the combinatorial graph.  The vector $L(k)$ tells  us how to go from the root of the component $A$ of the  graph  $\Gamma_S$ to which $k$ belongs, to $k$, namely:
\begin{eqnarray}\label{defL}
k+\sum_iL_i(k)\mathtt j_i&=&\sigma(k)\er(k)\,,\ |k|^2+\sum_iL_i(k)|\mathtt j_i|^2=\sigma(k)|\er(k)|^2\,, \nonumber \\ 1+\sum_iL_i(k)&=& \sigma(k).  
\end{eqnarray}
 This definition is well posed even if $A$ is not a tree, so that one can {\em walk} from $\er(k)$ to $k$  in several ways,  from   our genericity conditions.

\subsubsection{The NLS graph}  We now apply the previous  geometric construction to the NLS and the tangential sites $S$ of Theorem \ref{teo1}.  With the previous notations we have
\begin{eqnarray}\label{tildeQ}
\tilde{\mathcal Q}&=& \sum_{k\in S^c}\omega^{(1)}(\xi)\cdot L(k)|z'_k|^2+\sum_{\ell\in X_q^0} c_q(\ell)\sum_{(h,k)\in \mathcal P_\ell }z'_h\bar z'_k
\nonumber\\&+& \sum_{\ell\in X_q^{-2}}c_q(\ell) \sum_{\{h,k\}\in \mathcal P_\ell }[ z'_h  z'_k
 +  \bar z'_h\bar z'_k].
\end{eqnarray} 
 where,
given  an edge  $\ell$, we set $\ell=\ell^+-\ell^- $ and define:
\begin{equation}\label{cl}
 c_q(\ell):=\left\{\begin{array}{ll}\displaystyle (q+1)^2\xi^{\frac{\ell^++\ell^-}{2}}\sum_{\alpha\in\N^m\atop{|\alpha+\ell^+|_1=q}}\binom{q}{ \ell^++\alpha}\binom{q}{ \ell^-+\alpha }    \xi_i^\alpha  & \ell\in X_q^0\\\displaystyle
(q+1)q\xi^{\frac{\ell^++\ell^-}{2}}\sum_{\alpha\in\N^m\atop{|\alpha+\ell^+|_1=q-1}}\binom{q+1}{\ell^-+\alpha}\binom{q-1}{ \ell^++\alpha}  \xi_i^\alpha &\ell\in X_q^{-2}  
 \end{array}\right.
\end{equation}

We see now that the graph has been constructed in order to {\em decouple}  $\tilde{\mathcal Q}=\sum_A \tilde{\mathcal Q}_A$. The sum runs over all geometric blocks $A\in \Gama$ and, if $E_b(A), E_r(A)$ denotes the set of resp. black and red edges in $A$: 
 \begin{equation}\label{qtila}
\tilde{\mathcal Q}_A:=\sum_{k\in A}\omega^{(1)}(\xi)\cdot L(k)|z'_k|^2+\sum_{\ell\in E_b(A)} c_q(\ell)z'_h  \bar z'_k +  \sum_{\ell\in E_r(A)} c_q(\ell)[ z'_h  z'_k
 +  \bar z'_h\bar z'_k] 
\end{equation} 
 is a quadratic Hamiltonian in the variables $w'_A:=z'_k,\bar z'_k$ with $k$  running over the vertices of $A$, we have $\{\tilde{\mathcal Q}_A,\tilde{\mathcal Q}_B\}=0,\ \forall A\neq B$. In \cite{PP} we have shown that 
 \begin{proposition}\label{diva}
For each geometric block $A$ (with $a$ vertices) we can   divide the corresponding $2a$ dynamical variables $w'_A$  into two conjugate components $w'_A=(u',\bar u')$  each spanning a $\tilde{\mathcal Q}_A$ stable Lagrangian subspace where $u'_k=  z_k $ if $ \sigma(k)=1$, $u'_k= \bar z_k $ if $ \sigma(k)=-1$. In this basis $-\ii \tilde Q_A $ has a  block matrix   $ C_{A}\oplus -C_{A}$. By convention in the first block the root $\er$ corresponds to $z'_\er$.

\end{proposition}\smallskip

Given two vertices  $h\neq k\in A$  
 we have that the matrix element $c_{u'_h,u'_k}$ of $C_A$ is non zero if and only if $h,k$ are joined by an edge  (marked say $(i,j)$) and then 
 \begin{equation}
\label{Lema}c_{u'_h,u'_k}= \sigma(k)c_q(\ell),\qquad c_{u'_k,u'_k}=  \sigma(k)(\ome^{(1)}(\xi),L(k)). 
\end{equation}  

By definition $L(k)$ depends only on the combinatorial graph $\GA$ of which $A$ is a realization, therefore the matrix $C_A=C_\GA$ depends only on the combinatorial block $\GA$ (but the dynamical variables in $-\ii \tilde Q_A $ depend on the geometric block).

We thus finally have a finite list  $\mathcal G:=\{\GA_1,\ldots,\GA_N\}$ of combinatorial graphs which may appear, together with  a list of matrices  $C_{\GA_i}$ which are explicitly described  using   Formula  \eqref{Lema}.  The entries of these matrices  are polynomials in the elements $\sqrt{\xi_i}$  and homogeneous of degree  $q $ in $\xi$. 

\subsubsection{The space $F^{0,1}$}

In the   KAM algorithm we shall need to study in particular the  action by Poisson bracket of $ \mathcal N $ on a special space of functions (linear Hamiltonians on ${\bf{\bar \ell}}^{(a,p)}$)  called $F^{0,1}$ defined in \cite{PP}, for the readers convenience we recall the basic facts. 

It is convenient to write
  a variable  $z$ or $\bar z$  as $z^\sigma$ where $\s $ is  $ 1$ resp. $-1$.
\begin{definition}
We set $F^{0,1}$ to be the space of functions spanned by the basis elements
  $e^{\ii \s \s (k)\nu\cdot x}{z'_k}\,^\sigma=e^{\ii \s ([\s (k)\nu+L(k)] \cdot x)}{z_k}^\sigma,\ k\in S^c$    which preserve mass and momentum. \footnote{we deviate from the notations of \cite{PP} and in  $F^{0,1}$ we also impose zero mass}\end{definition} 
One easily sees that $F^{0,1}$ is a symplectic space under Poisson bracket. The formulas for mass and momentum in the new variables are for $\mathfrak m:=e^{\ii \s \s (k)\nu\cdot x}{z'_k}\,^\sigma.$ 
  \begin{equation}\label{posta2}
  \{ \mathbb L,\mathfrak m \}= \ii \sigma\s (k)(\sum_i \nu_i+1)\,\mathfrak m,\quad  \{ \mathbb M, \mathfrak m\}= \ii\s \s (k) (\sum_i \nu_i \mathtt j_i+ \er(k))\,\mathfrak m.
  \end{equation}
hence the conservation laws  tell us that for an element $e^{\ii \s \s (k)\nu\cdot x}{z'_k}\,^\sigma\in F^{0,1}$     the vector $\nu\in\Z^d$   is constrained  by the fact that $-\sum_i \nu_i \mathtt j_i$ must be in the set of roots in $S^c$ and moreover the mass constraint $\sum_i \nu_i=-1$.

For each connected component  $A$ of the graph $\Gamma_S$ with   root $\er$,  any solution $\nu$ of $\sum_i \nu_i \mathtt j_i+ \er=0$, where the mass of $\nu$ is $-1$, determines in   $F^{0,1}$ a block  denoted $A,\nu$. This is a symplectic space sum of two  Lagrangian spaces $(A,\nu)_+$ and $(A,\nu)_-=\overline{(A,\nu)_+}$. $(A,\nu)_+$ has  basis the elements $e^{\ii\s \sigma(k) \nu\cdot x}{z'_k}\,^{\sigma(k)},\ k\in A$.   
Thus $F^{0,1}$  decomposes  as orthogonal sum (with respect to the symplectic form), of these blocks $A,\nu$. Notice that, for a given geometric block $A$ there are infinitely many  blocks $A,\nu$ as soon as $n>d$.
 \begin{proposition}\label{madiN}
The matrix of $-\ii  \,ad(\mathcal N)$ on the block $(A,\nu)_+$ is the sum of   the matrix $C_\GA$ plus the scalar matrix $[(|\er(m)|^2+\sum_i\nu_i |\mathtt j_i|^2 )+\ome^{(1)}(\xi)\cdot \nu ]\, I_A.$
\end{proposition}\begin{proof}
  The action of $-\ii\tilde {\mathcal Q}$ does not depend on $\nu$,   it is only through $-\ii\tilde {\mathcal Q}_A$ and gives the matrix $C_A$. As for the elements $-\ii[   \ome(\xi)\cdot y'    +\sum_{k\in S^c}\tilde\Ome_k|z'_k|^2]$   by Formula \eqref{defL}, the term $-\ii\sum_{k\in S^c}\tilde\Ome_k|z'_k|^2$ contributes on the first block the scalar $ |\er(m)|^2 $.  As of $-\ii   \ome(\xi)\cdot y'   $  it also contributes by a scalar, this time $ (\sum_i\nu_i |\mathtt j_i|^2  + \ome^{(1)}(\xi)\cdot \nu).$  \qed\end{proof}

\section{Normal form reduction}
\subsection{Fitting decomposition\label{Fit}}
Let us recall some basic definitions of linear algebra,  given a linear operator $A:V\to V$ where $V$ is a finite dimensional vector space over a field $F$  (we will work on function fields so of characteristic 0), we have the {\em Jordan decomposition}  $A=A_s+A_n$  where  $A_s$ is semisimple and $A_n$  nilpotent with $[A_s,A_n]=0$.  Such a decomposition is unique so that $A_s$ is called the semisimple part of $A$ and $A_n$ the nilpotent part.  Semisimple may be defined in several ways but for us means that $A_s$ is diagonalizable, not necessarily over $F$  but in some finite extension field $G\subset F$ which contains the eigenvalues of $A$ (which are in fact the same as the eigenvalues of $A_s$).  The characteristic polynomials of  $A$ and $A_s$ coincide, on the other hand $A_s$ satisfies its minimal polynomial which has coefficients in $F$  and as root all the eigenvalues with multiplicity 1.

If $F$  contains the eigenvalues of $A$, call the distinct eigenvalues $\alpha_1,\ldots,\alpha_i$, each may appear with some multiplicity,    we have the {\em Fitting decomposition}  of $V=\oplus_i V_{\alpha_i}$  where each $V_{\alpha_i}$ is a uniquely determined subspace which is stable under $A$ and where $A$ has the unique eigenvalue $\alpha_i$.  It is easily seen that in fact  $V_{\alpha_i}$  is just the eigenspace of $A_s$ for the eigenvalue   $\alpha_i$.

In our case we are interested in matrices $C(\xi)$ depending in a polynomial way from the parameters $\sqrt{\xi_i}$, we may consider the field $F$ of rational functions in these parameters   so that $A=C(\xi)$ is a matrix  of some size $p$, with coefficients in $F$, denote by $M_p(F)$ the space of these matrices.

Take the  Jordan decomposition  and notice that, in this decomposition,  $A_s$ has no more polynomial entries but may acquire denominators. Then we have some finite field extension $G\supset F$  and a matrix $X\in  M_p(G)$  such that $X C(\xi)X^{-1}$  is a block diagonal form $\oplus  C_i$  with  $C_i\in M_{p_i}(G)$ has a unique eigenvalue $\alpha_i\in G$ and the  various eigenvalues are distinct. In fact $X$ is defined up to a scalar multiplication so that we may further assume that all the entries of $X$  are integral over the  ring of polynomials in the $\xi_i$, that is satisfy each some monic polynomial (dependent of the entry) $t^N+a_1t^{N-1}+\ldots +a_N$ where the coefficients $a_i$ are polynomials in the $\xi$.\smallskip

Now we have to interpret $G$  as field of algebraic functions in the parameters $\xi$, this follows a standard path.   The distinct eigenvalues are solutions of the minimal polynomial, then by removing the discriminant  (which gives a algebraic hypersurface) on the complement these are distinct as functions, then one can further remove some real algebraic variety, which we may assume to be homogeneous, so that the  open components of the complement are all simply connected (this follows for instance from the fact that one has an algebraic triangulation of the sphere  such that it also  triangulates the  intersection of all these algebraic hypersurfaces with the sphere, then one takes the corresponding cones of the  open triangles  \cite{geore}).  
As a result we  have that we may remove from the space of parameters $\xi$  (which for us are real parameters) some algebraic hypersurface (of course real), such that outside  this hypersurface the entries of  $X$ are true functions, which are as we have said algebraic so analytic,  and integral in the sense explained, and finally assume  distinct values for each value  of $\xi$ outside these algebraic hypersurfaces. Of course once we remove an algebraic hypersurface the complement is open and dense.

If $A\in End(V)$  is a linear map and we decompose  $V=\oplus_\alpha V_\alpha$  into a Fitting decomposition, we have for the linear operator $ad(A):X\mapsto [A,X]$,  on the space $End(V)$ of linear maps on $V$, the following facts.
$$ad(A)_s=ad(A_s),\quad ad(A)_n=ad(A_n).$$  The Fitting decomposition of  $End(V)$, under $ad(A)$,  is deduced from the decomposition in blocks induced by $V=\oplus_\alpha V_\alpha$  that is first
$$End(V)=\oplus_{\alpha,\beta}\hom(V_\alpha,V_\beta).$$   The subspace $\hom(V_\beta,V_\alpha)$ is relative to the eigenvalue  $\alpha-\beta$  so that $ad(A)$ on this space acts as $(\alpha-\beta)Id+ad(A_n)$. The map $ad(A_n)$ is nilpotent, in fact if $A_n^k=0$ we have  $ad(A_n)^{2k}=0.$  So for the off diagonal blocks, $\alpha\neq \beta$ we have $ad(A)$ on such a block is invertible with inverse
$$(\alpha-\beta)^{-1}(1+  \sum_{j=1}^{2k-1}(-1)^j [\frac{ad(A_n)}{(\alpha-\beta)}]^j). $$
It is possible that for different  eigenvalues we may have $\alpha_1-\beta_1=\alpha_2-\beta_2$  so that this decomposition in general refines, but need not be equal, to the Fitting decomposition  of $ad(A)$.\bigskip

\subsubsection{Blocks and matrices} We now apply this analysis to  the finitely many matrices   $C_{\GA_i}$ associated to the list $\mathcal G:=\{\GA_1,\ldots,\GA_N\}$ of combinatorial graphs which may appear.

 Thus we have finitely many distinct eigenvalues (with some multiplicity) of this finite list of matrices,  which are algebraic functions  of  the parameters $\xi$ and defined outside some homogeneous algebraic hypersurface. To make the discussion   precise from now on  one may remove  from $\R^n$ a homogeneous real algebraic hypersurface so that the connected components of the complement are all simply connected, and  so that the distinct eigenvalues are functions in each component and with distinct values pointwise (see \S \ref{Fit}). 
 
  We   fix one of these open simply connected regions, call it $\mathcal R_\alpha$ in which these algebraic functions are well defined and then    we denote this list of functions by
\begin{equation}\label{Tht}
\Upsilon:=\{\theta_1,\ldots,\theta_O\}.
\end{equation} Contrary to what happens for the cubic  NLS we may have various kinds of degeneracies in these eigenvalues, one eigenvalue may appear in two different matrices $C_{\GA_i}$ and with multiplicity $>1$.  This is  discussed in \S \ref{Fit} and \ref{Ga1}.

Later we will need to restrict to some compact domain in  one of the $ \mathcal R_\alpha\cap \Lambda$  in order to make  estimates on the derivatives.
\paragraph{Decomposing $\mathcal N$.}
We now want to simplify $\mathcal N$ by applying the Jordan and Fitting decomposition to the quadratic Hamiltonians  $\tilde{\mathcal Q}_A$.
In oder to do this we first write $\tilde{\mathcal Q}_A= \tilde{\mathcal Q}^s_A+\tilde{\mathcal Q}^n_A$ so that $\tilde{\mathcal Q}^s_A$ and $\tilde{\mathcal Q}^n_A$ commute, $\tilde{\mathcal Q}^s_A$ is represented by a semi-simple matrix and $\tilde{\mathcal Q}^n_A$  is represented by a nilpotent matrix.  Note that, since $\tilde{ Q}_A$ is a polynomial in $\sqrt\xi$, then  both its nilpotent and semi-simple part are rational functions in $\sqrt{\xi}$, so they are analytic outside an algebraic hypersurface (cf. \S \ref{Fit}).  

Clearly $\tilde{\mathcal Q}^s= \sum_A \tilde{\mathcal Q}^s_A$ and the matrix $\frac 12 Q^s_A= C_\GA^s\oplus (-C_\GA^s)$ where $C_\GA^s$ is the semi-simple part of $C_\GA$.

Notice that all the $\tilde{\mathcal Q}_A$ such that $A$ does not have red edges are represented by self-adjoint matrices and hence do not contribute to $\tilde{\mathcal Q}^n$.  Moreover since $\mathbb K$ and  $  \omega^{(1)}(\xi)\cdot y'    $ commute with $\mathcal N$ we can decompose $\mathcal N=\mathcal N^s+\mathcal N^n$ (semi-simple and nilpotent parts) where $\mathcal N^n=   \tilde{\mathcal Q}^n$ and $\{\mathcal N^s,\mathcal N^n\}=0$. 

We put $\mathcal N^s$ in  normal form by  reducing each of  the $C^s_\GA$.
This means that we only need to work on a finite number of matrices.  
\begin{proposition}
There exists a homogeneous algebraic hypersurface $\mathfrak A$ such that   the open region $\R^n\setminus \mathfrak A$ decomposes into finitely many simply connected components $\mathcal R_\alpha$  with the following property:

 \noindent For each combinatorial graph $\GA$ 
   the eigenvalues of  the matrix $C_\GA$, are algebraic, hence analytic, functions 
of $\xi$ say $\theta_1,\dots,\theta_k $ which are homogeneous of degree $q$.  

For all $\xi \in \R^n\setminus \mathfrak A$ there exists a linear   symplectic change of coordinates  $ u' \to U_\GA(\xi) u'= u''$ such that:

1. $U_\GA(\xi)$ is unitary with respect to $\Sigma_A$ and it can be chosen homogeneous of degree zero in $\xi$.

2. $U_\GA(\xi)$ is analytic in $\xi$. 

3. $U_\GA(\xi)$  conjugates  $C_\GA$  into a Fitting normal form:
that is  a block diagonal matrix  such that   each block has a unique eigenvalue.

More precisely for  each real eigenvalue $\theta$ , $C^s_\GA$ acts as $\theta I$  on the   eigenspace of $\theta$.  

For each  pair of conjugate  complex eigenvalues $\theta_\pm= a\pm \ii b$, of multiplicity $k$ we  have a real $2k$ dimensional space such that the two complex eigenvectors lie in its complexification. Then we have a  basis of this subspace such that $C^s_\GA$ restricted to this subspace is a $2\times 2$   block matrix $$ \begin{pmatrix}  a I_k & -bI_k \\ b I_k& a I_k\end{pmatrix}. $$
The matrix $\sigma_\GA$ of the standard form  on this basis is diag$(I_k,-I_k)$.
 
 \end{proposition}
 Note that  two eigenvalues which are distinct as functions of $\xi$ can be assumed to be  distinct for all values of the parameters $\xi_i\in  \R^n\setminus \mathfrak A$. 
\paragraph{Change of variables.}
For each geometric block $A$ we have the  two Lagrangian blocks of variables  where the Hamiltonian  $\tilde{\mathcal Q}_A$ acts  by Poisson bracket with the matrix $C_\GA \oplus (-C_\GA )$, we thus perform  on  the two Lagrangian blocks the symplectic change of variables induced by the matrix  $U_{\GA}$  which puts into Fitting decomposition  $C_\GA$.  This is by construction a symplectic change of variables and we call with $w''$  the new variables obtained. In these new variables we have a further decoupling of  each $ \tilde{\mathcal Q}_A$ as a sum  $ \tilde{\mathcal Q}_A=\sum_\theta   \tilde{\mathcal Q}_{A,\theta}$, since the component $A$ is determined by its root $\er$  we denote also $\tilde{\mathcal Q}_{A,\theta}=\tilde{\mathcal Q}_{\er,\theta}$.
  
  For  the $\GA$ with no red edges  the variables $u'_k=z'_k$. So it is obvious that $w''= (U z', \bar U \bar z')$. In general we have always
  $w''=    (U u', \bar U \bar u')$ but we should specify  which of the $u''_k$ have $\sigma(k)=1$ or $-1$. We notice that by definition $\{u'_k,\bar u'_h\}= \ii \delta^h_k\sigma(k)$ and since $U$ is unitary with respect to $\Sigma_A= $diag $(\sigma(k))_{k\in A}$ then also $\{u''_k,\bar u''_h\}=\ii \delta^h_k \sigma(k)$, so we still may say that $u''_k= (z'')_k^{\sigma(k)}$.  By abuse of notation we still call $x,y,z_k,\bar z_k$ the new variables.
  
  Since $U_\GA$ mixes only the variables $u$ and since $\mathbb L, \mathbb M,\mathbb K$  are scalar on the components $u$ then in the new variables we have still
   \begin{equation}\label{LM}
 \mathbb L= \sum_i y_i + \sum_k \s(k)  |z_k|^2 \,,\quad  \mathbb M = \sum_i \mathtt j_i y_i + \sum_k \s(k) \er(k) |z_k|^2\,,
\end{equation}  \begin{equation}\label{kappa}
\quad \mathbb K = \sum_i |\mathtt j_i|^2 y_i + \sum_k \s(k) |\er(k)|^2 |z_k|^2 .
\end{equation} 
 \begin{definition}\label{scrc}
We denote  by $\mathfrak S^c$ the set of
  pairs $\er,\theta$ where $\er$ is a root of a block $A$ in the geometric graph while $\theta$ is one of the eigenvalues appearing in the corresponding Fitting decomposition of $\GA$.  
\end{definition}
 After the change of variables the new basis is still indexed by convention by the set $S^c$. Then $S^c$ is decomposed as a union of finite blocks $S^c_{\er,\theta}$ corresponding  to the  elements of  $\mathfrak S^c$. 
Notice that  only finitely many of the blocks  indexed by $\mathfrak S^c$  contain $k$ with $\sigma(k)=-1$. Therefore all $\tilde{\mathcal Q}_{\er,\theta}$, except  finitely many, have the form, $\tilde{\mathcal Q}_{\er,\theta}=\theta(\sum_{k\in S^c_{\er,\theta}}|z_k|^2)$. 

The    finite {\em bad blocks}  have still a semisimple part $\tilde{\mathcal Q}^s_{\er,\theta}=\theta(\sum_k\s(k)|z_k|^2)$ and $\sum_{k\in S^c_{\er,\theta}}\s(k)| z_k|^2$ acts as the scalar $\ii,-\ii$ on the two corresponding Lagrangian spaces, since a basis of one of them is $z_k^{\s(k)},\ {k\in S^c_{\er,\theta}}$ (and the other the conjugate).
 The  normal form $\mathcal N$ is 
\begin{equation}\label{secN}
\mathcal N= \mathbb K+\mathbb K^1\,,\quad \mathbb K^1= (\omega^{(1)}, y)+  \sum_{(\er,\theta)\in \mathfrak S^c} \tilde{\mathcal Q}_{\er,\theta}\end{equation}   
where the decoupled quadratic Hamiltonian $\tilde{\mathcal Q}_{\er,\theta}$ corresponds to the block of the Fitting decomposition, of eigenvalue $\theta$, of the matrix $C_{\GA}$.  

\begin{remark}\label{cacchio}
For a given $k$ it is useful to denote by $\theta(k)$ the corresponding eigenvalue. Note that the map $k\mapsto \theta(k)$ depends only on the combinatorial graph.

 In order to keep track of the information on the eigenvalue in the geometric and combinatoric graphs, we mark each vertex with its eigenvalue $\theta(k)$.
\end{remark}

\section{the kernel of $ad(\mathcal N)$}  We now study the subspace of regular analytic Hamiltonians of degree $\leq 2$ which Poisson commute with $\mathcal N$, $\mathbb M$, $\mathbb L$ for generic values of $\xi$. Note that they must commute separately with $\mathbb K$ and  $\mathbb K^1$ since they are homogeneous of degree 0,$q$ respectively in $\xi$.

{\bf Degree zero in $w$:} This space, which we denote by $F^0$ can be divided into $F^{0,0}\oplus F^{1,0}$ with basis $e^{\ii \nu\cdot x}$ and $e^{\ii \nu\cdot x} y_l$.  Since monomials are eigenvectors of all the operators we need to verify when all the eigenvalues are identically zero.
$$0= \{ \mathbb L, e^{\ii \nu\cdot x}\}\iff   \ii \sum_i \nu_i=  0\,,\quad 0= \{ \mathbb M, e^{\ii \nu\cdot x}\}\iff   \ii \sum_i \nu_i \mathtt j_i=  0\,,$$ $$0=\{ \mathbb K, e^{\ii \nu\cdot x}\}\iff   \ii \sum_i \nu_i |\mathtt j_i|^2=   0\,,\quad 0=\{ \mathbb K^1, e^{\ii \nu\cdot x}\}\iff    \ii (\omega^{(1)} , \nu)=  0.$$
The last condition implies that $\nu=0$, by \ref{omest}. The same rules hold if the monomial has degree one in $y$. Hence the Kernel of $ad(\mathcal N)$ is $x$ independent and hence of the form $ c  +\overrightarrow C \cdot y$.

 {\bf Degree one in $w$:} We analyze the action of  $ad(\mathcal N)$ on a block  of the space   $F^{0,1}$ given by an element of $\mathfrak S^c$.  Given $k\in S^c_r$ in this block  
 the monomial $e^{\ii \sigma\sigma(k)\nu\cdot x}z_k^{\sigma}$ is an eigenvector for all our operators, by conservation of $\mathbb L,\ \mathbb M$ we have  $\sum_i \nu_i +1=0,\   \sum_i \nu_i \mathtt j_i+ \er(k)=0$ and with eigenvalue for $\mathbb K$:
  $$\{ \mathbb K, e^{\ii \sigma\sigma(k)\nu\cdot x}z_k^{\sigma}\}		\to  \ii \sigma\sigma(k)( \sum_i \nu_i |\mathtt j_i|^2 + |r(k)|^2)\,,$$ as for  $\mathbb K^1$  we have the diagonal contribution $$\{  \omega^{(1)}\cdot y , e^{\ii \sigma\sigma(k)\nu\cdot x}z_k^{\sigma}\}		\to   \ii \sigma\sigma(k) \omega^{(1)}\cdot y $$  plus the contribution of the term $\tilde{\mathcal Q}_{\er,\theta}$.  Now we claim that we may assume that 
    the determinant of the matrix  of  $ad(\mathcal N)$ is non-zero as a function in $\xi$. This is clear if  the scalar part $\sum_i \nu_i |\mathtt j_i|^2 + |\er(k)|^2\neq 0$.  Since we are working in the space where we have the constraint  $\sum_i \nu_i \mathtt j_i+ \er(k) =0$  this  condition amounts to  ask that
   \begin{equation}\label{tangent}
\sum_i \nu_i |\mathtt j_i|^2 + |\sum_i \nu_i \mathtt j_i|^2\neq 0.
\end{equation} We then choose our tangential sites  $\mathtt j_i$ satisfying this constraint for all $\nu$ with coordinates $|\nu_i|\leq a$ where $a\leq 4dq$ is a common bound for the absolute value of  the coordinates of each $L(k)$ (which form a finite list).  We then claim that for the remaining $\nu$  the matrix    $ C_\GA$ has non--zero determinant, as polynomial in $\xi$. For such a $\nu$ we have some coordinate $\nu_i$ with  $|\nu_i|>a$, we may assume that it is $\nu_1$ then set $\xi_j=0,\ \forall j\neq 1$  and verify that still the determinant of the specialyzed matrix is $\neq 0$.  All the off diagonal entries, which are given by formulas \eqref{cl} vanish since they all contain at least 2 variables $\xi$ due to Formula \eqref{glied}.
  
  From Lemma \ref{omest}, when we set $\xi_j=0,\ \forall j\neq 1$ we   get $\omega^{(1)}=-q(q+1)\xi_1^{q } e_1$.

  To prove our claim   we need to analyze  in more detail the matrix    $ C_\GA$ in particular its diagonal part which, from Formula \eqref{qtila}  comes from $\omega^{(1)}(\xi)\cdot L(k)|z'_k|^2$.   Finally after the specialization the diagonal matrix has entries  a non zero constant times $ \xi_1^{q }(\nu_1-L_1(k)) $ which is non--zero by our constraint on $\nu_1$  since the coordinates of each $L(k)$ are in absolute value bounded by  $a$.

  We have proved that for all $\xi$ outside a countable set of algebraic hypersurfaces the kernel of $ad(\mathcal N)$  is zero on functions of degree one commuting with $\mathbb L,\mathbb M$.
  
  Since the determinants  depend only on the semi--simple part  we have the same statement also for $\mathcal N^s$,
  
 \begin{remark}\label{tange}
Note that the genecirity constraints \eqref{tangent} can be imposed since they are not zero as functions of $\mathtt j_i$, indeed the only trivial case is when $\nu= -e_i$ which contradicts $\er\in S^c$. Actually constraints like \eqref{tangent} appeared already in \cite{PP}.
\end{remark}  
   {\bf Degree two in $w$:} We denote this space by $F^{0,2}$. We do not exhibit explicitly such a kernel (since  we do not control the nilpotent part of $\mathcal N$). We observe that ker$($ad$(\mathcal N))$ is contained in ker$($ad$(\mathcal N^s))$ which now depends only on the eigenvalues.
   The eigenvectors of degree two for ad$(\mathcal N^s)$ in $w$ are  obtained by multiplying eigenvectors  of degree one in all possible ways
   and the eigenvalues are sums of two eigenvalues.
   Since each eigenvalue appears together with its negative in the kernel we have the following two cases. This depends on wether we pair $(\nu_1,A_1)_+$ with $(\nu_2,A_2)_+$ (then we set $\nu=\nu_1+\nu_2$ ) or  we pair $(\nu_1,A_1)_+$ with $(\nu_2,A_2)_-$ (and we set $\nu=\nu_1-\nu_2$). 
 
\noindent   Case 1:
   $$\sum_i \nu_i =-2\,,\quad  \sum_i \nu_i \mathtt j_i+\er(h)+\er(k) =0$$
  \begin{equation}\label{caso1}
\sum_i \nu_i |\mathtt j_i|^2+ |\er(h)|^2+|\er(k)|^2 \,,\quad   \omega^{(1)}(\xi)\cdot \nu      +\theta_h+\theta_k= 0\,.
\end{equation}
  
\noindent  Case 2:
  $$ \sum_i \nu_i = 0\,,\quad  \sum_i \nu_i \mathtt j_i+\er(h)-\er(k) =0.$$
 \begin{equation}\label{caso2}
\sum_i \nu_i |\mathtt j_i|^2+ |\er(h)|^2-|\er(k)|^2 =0\,,\quad    \omega^{(1)}(\xi)\cdot \nu      +\theta_h-\theta_k=0\,.
\end{equation} 
\begin{remark}\label{argo}
Of course in the kernel we always have the case $\nu=0$, therefore we are in Case 2  with $\er(h)=\er(k)$ and $h,k$ in the same Fitting block. 
\end{remark} 
  On the other hand  we may have a kernel arising from different Fitting blocks and also with $\nu\neq 0$. 
  We want to encode this phenomenon in a new graph, whose edges we denote by $Y=Y^0\cup Y^{-2}$, and defined  as follows.     \subsection{The graph ${ \Gamma}_{S }^{(f)}$    \label{Ga1}}
   \begin{definition}
We say that $\ell\in \Z^n\setminus \{0\}$   is an edge $\ell\in Y^0$ if   $\eta(\ell)=\sum_i\ell_i=0$  and  we can obtain $  \omega^{(1)}(\xi)\cdot \ell     $ as difference of two eigenvalues of the list \eqref{Tht} (of the matrices $C_\GA$). 

We say  $\ell\in \Z^n$   is an edge $\ell\in Y^{-2}$ if  $\eta(\ell)=\sum_i\ell_i=-2$, $ \ell\neq -2 e_i$ ($i=1,\dots,n$) and we can obtain $  \omega^{(1)}(\xi)\cdot \ell     $ as sum of two eigenvalues  of the list \eqref{Tht}.
\end{definition} 
\begin{remark}\label{connl}
\begin{enumerate}\item Notice that $\ell$ is uniquely determined by the polynomial $  \omega^{(1)}(\xi)\cdot \ell     $ since $-q(q+1)\ell_i$ is the coefficient of $\xi_i^q$, see Lemma \ref{omest}.
\item We notice that $Y$ depends only on the structure of the matrices $C_\GA$ which form a list of matrices associated to finitely many combinatorial graphs and   do not depend on the choice of the $\mathtt j_i$.
\item  Then $Y$ is contained in the set of edges $X_p$  (see Definition \ref{glied}) for  some $p$ sufficiently large.

\item Notice that   in the case of complex eigenvalues we must have that the imaginary parts cancel.
\end{enumerate}
\end{remark}
\paragraph{The basic stratification. \label{basst}} The  geometric graph $ \Gamma_{S,\, X_q\cup Y}$  is  now used to define  the basic stratification of $\Z^d$ of which all the stratifications used in proposition \ref{teo2} form a refinement.

The construction of the stratification is explained in  \cite{PP1}, we take as single point strata all the points in the finite graphs with red edges. Then the graphs with only  black edges come into a finite family  such that the graphs of each family are all isomorphic  and obtained each from the other by a translation.   The set of  translations, for a given graph $\Gamma$ with $a+1$  affinely independent points  is defined as follows, one chooses a root $\er\in\Gamma$ and takes the   subgroup of elements $u\in\Z^d$ orthogonal to the $a$ linearly independent elements $x-\er,\ x\in\Gamma\setminus\{\er\}$.  Then generically  $\Gamma+u$ is still a  connected component of  the graph $ \Gamma_{S,\, X_q\cup Y}$ while for some special $u$  this translate is contained in a larger connected component.
\begin{definition}\label{basstra}
The stratum associate to a point $k$ lying in a  graph $\Gamma$  with only  black edges is then defined by taking all the points $h=k+u$  lying in an isomorphic  graph $\Gamma+u$ connected component of $ \Gamma_{S,\, X_q\cup Y}$. 
\end{definition} 
  \paragraph{The graph ${\tilde\Gamma}_{S, X_q\cup Y}$.
   \label{Ga2}}Starting from the geometric graph $\Gamma_S$ we construct a new geometric graph, which we denote by $\tilde \Gamma_{S,\, X_q\cup Y}$,  with vertices in $ S^c$ and edges $X_q\cup Y$.  This is obtained from   the  geometric graph $ \Gamma_{S,\, X_q\cup Y}$,  
which contains the graph $ \Gamma_S$    by keeping only the new edges  which connect two roots (of $\Gamma_S$)  with the following constraints.

As explained in Remark \ref{cacchio} we have marked each vertex in $S^c$ with its eigenvalue $\theta(k)\in\Upsilon$. Consider any two connected components $A,B$ of $\Gamma_S$ with roots $\er,\er'$:
\begin{definition}\label{edrr}
If $\er,\er'$ are connected  in $ \Gamma_{S,\, X_q\cup Y}$ by  a black oriented edge marked $\ell\in Y^0$, that is
 \begin{equation}\label{ner}
\sum_i \mathtt j_i \ell_i + \er -\er '=0\,,\quad \sum_i |\mathtt j_i|^2 \ell_i + |\er |^2-|\er '|^2=0\,,
\end{equation}  we keep this edge in $ \tilde\Gamma_{S,\, X_q\cup Y}$ if and only if 
  there exists $k\in A$ and $h\in B$ such that
$$   \omega^{(1)}(\xi)\cdot \ell     + \theta(k)-\theta(h)=0.$$
If $\er,\er'$ are connected  in $\Gamma_{S,\, X_q\cup Y}$ by  a red edge marked $\ell\in Y^{-2}$ that is
 \begin{equation}\label{ros}
 \sum_i \mathtt j_i \ell_i + \er +\er '=0\,,\quad \sum_i |\mathtt j_i|^2 \ell_i + |\er |^2+|\er '|^2=0\,,
\end{equation}
 we keep this edge in $\tilde \Gamma_{S,\, X_q\cup Y}$ if and only if  there exists $k\in A$ and $h\in B$ such that
$$   \omega^{(1)}(\xi)\cdot \ell     + \theta(h)+\theta(k)=0.$$
\end{definition}

\begin{remark} If  $\er,\er'$ are connected  by  a black oriented edge marked $\ell$ then $\er',\er$ are connected  by  a black oriented edge marked $-\ell$, for red edges instead the relation is symmetric.

Notice moreover that for most points $k\in S^c$, in fact outside finitely many hyperplanes plus a finite set,  one has $\er (k)=k$ and $\theta_k=0$.  \end{remark}
\begin{proposition}
 The connected components of $ \tilde\Gamma_{S,\, X_q\cup Y}$ for generic choices of $S$ have at most $2d+1$ vertices moreover the components which do not contain any red edge have  vertices which are affinely independent (and hence at most $d+1$).
\end{proposition}
\begin{proof}
Our graph   is a subgraph of $ \Gamma_{S,\, X_q\cup Y}$ and we may apply Proposition \ref{Gen}.
  \qed\end{proof}\paragraph{Some geometry.}
  \begin{remark}\label{alt}
There is an equivalent way of looking at this construction,  the conditions of  Definition \ref{edrr}, define a graph, called the {\em root graph}, with vertices the roots $\er$  and edges in $Y$,   a connected component  of  the graph $ \tilde\Gamma_{S,\, X_q\cup Y}$ is obtained by a connected component $C$ of this graph with vertices roots, by taking the union of the components  of $\Gamma_S$ with root $\er\in C$. \end{remark}
Then we see the following  important fact.  In $ \tilde\Gamma_{S,\, X_q\cup Y}$  the   components  containing only black edges are clearly obtained  by the previous gluing  construction from all, except finitely many, the connected components of $ \Gamma_S$ containing only black edges. In other words we have a set  $S^c_g$, complement of a finite set, formed by all vertices of the connected components of $\tilde \Gamma_{S,\, X_q\cup Y}$   containing only black edges.

   We have that these components  come into a finite family of usually infinite sets  each obtained by translations  from some  subgroup of  a given component.

  We then have  a linear stratification of  $\mathbb Z^n$, which we still denote by $Y_j$,  in which the roots  of  the components are  a union of strata.  
  
  So a stratum $Y$ of roots  of dimension $k$ determines a subgroup $\Lambda$ of $\mathbb Z^n$   of translations, of its closure,  of rank $k$ and  for any $\er\in Y$  we have that $Y=\Lambda\setminus  \Lambda'+\er$ where $ \Lambda'$   is a finite union of subgroups of $ \Lambda$ or rank stricly lower than $k$.

  For a given stratum $Y$ of roots, the graphs having roots in $Y$ are all isomorphic as combinatorial graphs marked by the elements $\theta$ under translation  by the difference of the roots their union form a union of strata parallel to $Y$ (i.e. inside cosets of $\Lambda$).  
\begin{definition}\label{figa}[Final graph]  
Notice that this construction induces a graph $ \Gamma^{(f)}_S$  also  on the set  $\mathfrak S^c$, see \ref{scrc} of pairs $(\er ,\theta)$. We define that  $(\er ,\theta), (\er ',\theta') $ are connected  if $\er,\er'$ are connected and $\theta(k)=\theta,\theta(h)=\theta'$ for some $k,h$ with $\er(k)=\er,\ \er(h)=\er'$.   \end{definition}

In this graph then the eigenvalues $\theta,\theta'$ are either both real or both complex. In the case of complex eigenvalues we must note that if $(\er ,\theta), (\er ',\theta') $ are connected then so are $(\er ,\bar \theta), (\er ',\bar\theta') $.

\begin{lemma}\label{latiii}
If  $(\er ,\theta), (\er ',\theta') $ are connected by a path of edges then they are connected by an edge. Such edge is  black  if the path contains an even number of red edges and  is red  otherwise.
\end{lemma}
\begin{proof}
We only need to prove our statement when the path is made of two edges. Suppose $(\er ,\theta), (\er _1,\theta_1)$ are connected by the edge $\ell$ and $(\er _1,\theta_1)$ is connected to $(\er ',\theta')$ by  $\ell'$. If $\ell$ and $\ell'$ are both black we have $\theta_1- \theta =  (\ome^{(1)}(\xi),\ell)$  and $\theta'-\theta_1=  (\ome^{(1)}(\xi),\ell')$ which implies that $\theta'-\theta =  (\ome^{(1)}(\xi),\ell+\ell')$ an hence $\ell+\ell'\in X$. In the same way:
$$  \sum_i \mathtt j_i \ell_i + \er -\er _1=  \sum_i \mathtt j_i \ell_i' + \er _1-\er '= 0\;\to \sum_i \mathtt j_i (\ell_i+\ell_i') + \er -\er '=0 ,$$ same for the quadratic formulae. We have proved that $(\er ,\theta), (\er ',\theta') $ are connected by a black oriented edge marked $\ell+\ell'$.
Similarly if $\ell,\ell'$ are red then  $(\er ,\theta), (\er ',\theta') $ are connected by a black oriented edge marked $\ell-\ell'$. Finally if $\ell$ is red and $\ell'$ is black  $(\er ,\theta), (\er ',\theta') $ are connected by a red edge marked $\ell-\ell'$.
\qed\end{proof}  We need now some further constraints of the frequencies $\mathtt j$.  Since  we have seen that this final graph does not depend on $S$ we can impose the condition  that $(\er ,\theta), (\er ,\theta'),\ \theta\neq\theta' $ is never connected by an edge, since we may assume that, if  $\theta\pm \theta'=(\ome^{(1)}(\xi),\ell)$  we have always  $\ell\cdot\mathtt j\neq 0$.  

Then we have 
\begin{proposition}\label{prp} The projection $(\er ,\theta)\mapsto \er$  is a map of  the graph $ \Gamma^{(f)}_S$ to $ \Gamma_{S,\, X_q\cup Y}$. It maps thus a connected component injectively into a connected component.

\end{proposition}

In each connected component $\mathfrak D$ of $ \Gamma^{(f)}_S$ we now want to choose a root,  we do this by choosing as root a pair $(\er,\theta)$ with $\er$ a root of a components $A$ of  $\Gamma_S$, in other words we choose a root of the component of the {\em root graph},  in order to distinguish this from the other pairs appearing we shall use the symbol  $ \mathfrak t:=(\er_{  \mathfrak t},\theta_{  \mathfrak t})$. By the previous  remarks  we can make this choice {\em translation} invariant for the components with black edges, we always make this choice.
 We shall denote this last set of {\em roots} by $\mathfrak T$ which therefore is an indexing set for the connected components of $ \Gamma^{(f)}_S $. 
  
  It is important to remark that we can divide $\mathfrak T$ into a finite set $\mathfrak T_f$ (which contains all the bad blocks) and the infinite {\em good} set $ \mathfrak T_g$ which by definition is the part which projects to  the part of  $\tilde \Gamma_{S,\, X_q\cup Y}$ made of  the   components  containing only black edges.
%
%
 
  By Lemma \ref{latiii} each vertex $(\er,\theta)$ of the chosen component  $\mathfrak D$ is connected to the root
 $ \mathfrak t  $ by an edge $\ell(\er,\theta)$. 
 
 Given a root $\mathfrak t  \in \mathfrak T$ let us       denote by $\mathfrak D_{\mathfrak t}$ the corresponding component of $ \Gamma^{(f)}_S$ and by $\dodo_\et$, the {\em support of $\mathfrak D_{\mathfrak t}$}  to be       the set of the $k\in S^c$ such that $(\er(k),\theta_k)\in \mathfrak D_{\mathfrak t}.$   Notice that by construction this is a union of some of the subsets ${ S^c_{\er,\theta}}$  defined by the Fitting decomposition.    If $k\in\dodo_\et$ we set $\mathfrak t(k)=(\er_{  \mathfrak t}(k),\theta_{  \mathfrak t}(k)):=\mathfrak t$. In general $\er_{  \mathfrak t}(k)\neq\er(k),\  \ \theta_{  \mathfrak t}(k)\neq \theta(k)$.
  \begin{corollary}\label{findec}
The sets $\dodo_{\mathfrak t}$ decompose  the connected components of     $ \tilde\Gamma_{S,\, X_q\cup Y}$.     This decomposition of the components with only black edges is invariant under translations.
\end{corollary}       
 
 Set $\ell_k  $ through the Formulas  \begin{equation}\label{alcf}
 \mathtt j \cdot\ell_k+\er(k)=(\eta(\ell_k)+1)\er_{  \mathfrak t}(k),\   \omega^{(1)}(\xi)\cdot \ell_k   +\theta(k)=(\eta(\ell_k)+1)\theta_{  \mathfrak t}(k).
\end{equation} We also define a new color $s(k):=\sigma(k)(\eta(\ell)+1)=\pm 1$. Note that we can reformulate the Formulas \eqref{defL}, in terms of  $L_k, \ell_k, s(k) , \mathfrak t(k)$   by substituting in those Formulas the expressions given in Formulas \eqref{ner}, \eqref{ros}.
 
 We see now that we have a parallel with \ref{pl}.  Take a degree 2 monomial $\mathfrak m$ product of two eigenvectors as in the previous discussion.
\begin{proposition}
A monomial $\mathfrak m=e^{\ii \nu\cdot x}z_h^{\sigma}z_k^{\sigma'}$  with $\nu\neq 0$ is in the kernel of $ad(\mathcal N^s)$ if and only if $(\er(h),\theta_h),(\er(k),\theta_k) $   are connected by an edge $\ell$ and $\nu=\pm \ell$ (here the sign is determined by $\s,\s'\s(k),\s(h)$). 

In Formula  \eqref{caso1} the edge is a red edge while in   Formula  \eqref{caso2} it is a black edge.\end{proposition}
\begin{proof}
We apply the commutation rules; we only need to show that the possible monomials have $\nu\neq -2e_i$ for all $i$. Indeed in this case we see that the equations  for the conservation of $\mathbb M$ and $\mathbb K$ have the unique solution $r(h)= \mathtt j_i\notin S^c$, a contradiction.
\qed\end{proof}

\subsection{A new phase shift\label{fincoo}}
 At this point we can perform a new symplectic change of variables  which is  done with the same method of phase shift as in Formula \eqref{labella}.

We now define  $\Psi$:
$$ 
z_k = e^{-\ii \s(k)\ell_{k} \cdot x}z_k'\,,\quad \forall  k\in S^c \,,
$$
$$
 x=x'\,,\quad y= y'+ \sum_{k\in S^c} \s(k)\ell_{k}|z'_k|^2 \,.
$$

\begin{proposition}\label{ilcambio}
i) \;$\Psi$ is symplectic analytic and leaves $\mathbb M, \mathbb L,\mathbb K,\mathbb K_1$ and $\mathcal N$  in normal form.
ii)\;  For any function $f\in \mathcal R_{s,r}$  of degree $\leq 2$  and such that $f$ Poisson commutes with $\mathcal N^s$ we have that $f\circ \Psi$ does not depend on $x$.
\end{proposition}
\begin{proof}
We first prove item ii). Since $f$ is an absolutely convergent sum of monomials we only need to prove our claim on single monomials $\mathfrak m$. If the degree in $w$ is zero or one the statement follows trivially (by substituting the new variables). In case of degree two we substitute and  get in one case
$$e^{\ii \s\s(h)\nu\cdot x} z_h^\s z_k^{\s'}=  e^{\ii( \s\s(h)\nu -\s\s(h) \ell(h) -\s'\s(k) \ell(k))\cdot x}( z')_h^\s (z')_k^{\s'} =  ( z')_h^\s (z')_k^{\s'}$$
by applying Lemma \ref{latiii} to the vertices $(\mathfrak R,\mathfrak t),(\er(h),\theta_h),(\er(k),\theta_k)$ and recalling that the last two vertices are joined by an edge marked $\nu$. The other case  is identical.

It is possible that $f$ contains monomials   with $\nu=0$. Then by Remark \ref{argo} $ (\er(h),\theta_h)=(\er(k),\theta_k)$ and $\s \s(h)=\s'\s(k)$ and such term remains $x$ independent.

i)\; We can compute the new Hamiltonians  explicitly by substitution:
\begin{eqnarray*}
 \mathbb L&=& \sum_i  y_i + \sum_k  \s(k) |z_k|^2 =\sum_i  y'_i + \sum_{k\in S^c} \s(k)\eta(\ell_{k})|z_k|^2+ \sum_k  \s(k) |z_k|^2\\ &=&\sum_i  y'_i + \sum_k   s(k) |z'_k|^2 
 \\ \mathbb M &=&   \mathtt j \cdot y  + \sum_k \s(k) \er(k) |z_k|^2 =    \mathtt j \cdot y' + \! \sum_{k\in S^c} \s(k)  \mathtt j \cdot\ell_{k}|z'_k|^2+\! \sum_k  \s(k)   \er(k)  |z'_k|^2 
 \\ &=& \mathtt j \cdot y' +\sum_k   s(k)   \er_{  \mathfrak t}(k)  |z'_k|^2= \sum_{\mathfrak t\in \mathfrak T} \er_{\mathfrak t}(\sum_{k\in \dodo_\et} s(k)  |z'_k|^2).
\end{eqnarray*}
As for the nilpotent part $\tilde Q^n$ in $\mathcal N$ we notice that it is a sum of monomials as in Remark \ref{argo}, so  $\tilde Q^n$ remains block diagonal and $x$ independent by the same argument as in item ii).
\qed\end{proof}
In order to simplify the notations we now drop the apex $'$ and write $y,z$ the new coordinates.
\begin{corollary}\label{diagp}
For $\xi\in \mathcal A_\alpha$,  in the final coordinates the  normal form of the Hamiltonian has the form 
\begin{equation}\label{FF}
\mathcal N= \mathbb K+\mathbb K^1\,,\quad \mathbb K^1=  \omega^{(1)}(\xi)\cdot y +  \sum_{\mathfrak t\in \mathfrak T_f}  {\mathcal Q}_{\mathfrak t}+\sum_{\mathfrak t\in \mathfrak T_g} \theta_{\mathfrak t}(\sum_{k\in \dodo_\et}| z_k|^2)\end{equation}   \begin{equation}\label{MLK}
\quad \mathbb K = \sum_i |\mathtt j_i|^2 y_i +\sum_{\mathfrak t\in \mathfrak T} |\er_{\mathfrak t}|^2(\sum_{k\in \dodo_\et} s(k)  |z_k|^2) .
\end{equation} 
As for the finitely many ${\mathcal Q}_{\mathfrak t}$  each corresponds to a Hamiltonian which can be split  in Jordan decomposition, with a {\em scalar} and a nilpotent part ${\mathcal Q}^{\rm nil}_{\mathfrak t}$:
$$ {\mathcal Q}_{\mathfrak t}={\mathcal Q}^{\rm nil}_{\mathfrak t}+\theta_{\mathfrak t}(\sum_{k\in \dodo_\et}s(k)| z_k|^2),\quad \forall \mathfrak t\in \mathfrak T_f. $$ The associated matrix corresponds to a map on two Lagrangian blocks each with a single  non--zero eigenvalue $\pm \ii \,\theta_{\mathfrak t}$ but contains possibly also some nilpotent part.
 \end{corollary}\begin{proof}
When we substitute in Formula \eqref{secN}, setting $\mathfrak S^c_f$ the support of  the finitely many blocks in $ \mathfrak T_f$
$$ \mathbb K^1=  \omega^{(1)}(\xi)\cdot y +  \sum_{(\er,\theta)\in \mathfrak S^c_f} \tilde{\mathcal Q}_{\er,\theta}+  \sum_{(\er,\theta)\notin \mathfrak S^c_f} \theta(\sum_k|z_k|^2)$$
$$=  \omega^{(1)}(\xi)\cdot y' + \sum_{k\in S^c}\s(k)  \omega^{(1)}(\xi)\cdot\ell_{k} |z'_k|^2 +\!\! \sum_{(\er,\theta)\in \mathfrak S^c_f} \tilde{\mathcal Q}_{\er,\theta}+  \sum_{(\er,\theta)\notin \mathfrak S^c_f} \theta(\sum_k|z'_k|^2).$$
For the $k$ in the support of  $\mathfrak T_g$ we have   $\sigma(k)=1,\ \eta(\ell_k)=0$, the coefficient of   $|z'_k|^2$  comes from two terms,   $\omega^{(1)}(\xi)\cdot\ell_{k}$ and $\theta_k$ the sum is  by Formula \eqref{alcf}, $\theta_{  \mathfrak t}(k)=\theta_{\mathfrak t}$ which gives the last term of Formula \eqref{FF}, Formula \eqref{MLK} is similar.

 The remaining terms are collected in the finite sum  namely, for a given   $\mathfrak t\in \mathfrak T_f$  the contribution to  $ {\mathcal Q}_{\mathfrak t}$ is the sum of the sum  $\sum_{(\er,\theta)\in D_{\mathfrak t}} \tilde{\mathcal Q}_{\er,\theta}$ plus the sum $ \sum_{k\in \dodo_\et}\s(k)\, \omega^{(1)}(\xi)\cdot\ell_{k}\,|z'_k|^2$. Formula \eqref{alcf}  gives the statement on the eigenvalues.
\qed\end{proof}

Remark also that, if $k\in \dodo_{\mathfrak t},\ \mathfrak t\in \mathfrak T_g$ then $s(k)=1$. Thus $\mathcal N^s$ is the sum of the terms
\begin{equation}\label{itermini}
\sum_i (|\mathtt j_i|^2+\omega^{(1)}_i(\xi))y_i+\sum_{\mathfrak t\in \mathfrak T } ( |\er(\mathfrak t)|^2+\theta_{\mathfrak t})(\sum_{k\in \dodo_\et}s(k)| z_k|^2 ).
\end{equation}That is  in $y$, given by $\sum_i (|\mathtt j_i|^2+\omega^{(1)}_i(\xi))y_i$, the infinite real diagonal part  given by $\sum_{\mathfrak t\in \mathfrak T_g} (|\er(\mathfrak t)|^2+\theta_{\mathfrak t})(\sum_{k\in \dodo_\et}| z_k|^2 )$ and finally the finite semisimple part involving the  {\em bad} sites  $\sum_{\mathfrak t\in \mathfrak T_f} ( |\er(\mathfrak t)|^2+\theta_{\mathfrak t})(\sum_{k\in \dodo_\et}s(k)| z_k|^2 )$.  These finitely many blocks may also contribute to the nilpotent part of $\mathcal N$.

Recall that $\sum_{k\in \dodo_\et}s(k)| z_k|^2$ acts as the scalar $\ii,-\ii$ on the two corresponding Lagrangian spaces, since a basis of one of them is $z_k^{s(k)}$ (and the other the conjugate).
\paragraph{$F^{0,1}$ and the algebra $\mathcal F_{{\rm ker}}$.}  We now decompose $F^{0,1}$     as orthogonal sum (with respect to the symplectic form),  with  blocks  each decomposed as sum of two conjugate Lagrangian subspaces, $ (\dodo_{\mathfrak t},\nu)^{+}\oplus (\dodo_{\mathfrak t},\nu)^{-}$ with $\et\in\mathfrak T$ and $\nu\in\Z^n$ constrained by the conservation laws.

For each index  ${\mathfrak t},\nu$ the constraints $$\pi(\nu)+\er_{\mathfrak t}=\nu\cdot\mathtt j+\er_{\mathfrak t}=0,\quad \eta(\nu)+1=0$$ express respectively   the conservation of momentum and   mass. We have as basis for  $(\dodo_{\mathfrak t},\nu)^{+}$ the elements  $e^{\ii   \nu\cdot x}z_k^{s(k)}$  where  $k\in \dodo_\et $ and for  $(\dodo_{\mathfrak t},\nu)^{-}$ the conjugate elements.

These blocks are obviously stable  under $ad(\mathcal N)$ and the action is deduced from corollary \ref{diagp}, note that on each block this action is invertible.
\smallskip

The product of two blocks $(\dodo_{\mathfrak t_i},\nu_i)^{\pm 1},\ i=1,2$  produces a quadratic block in $F^{0,2}$ stable under  $ad(\mathcal N)$ with basis  the products of basis  elements. In order to avoid repetitions we index different quadratic blocks by $(\nu,\et_1,\et_2,\s_1,\s_2)$  where the conservation laws are:
$$
 \eta(\nu)+1+\s_1\s_2=0\,,\quad \pi(\nu)+\er_{\mathfrak t_1}+ \s_1\s_2 \er_{\et_2}=0\,
$$
while the basis elements are  $e^{\ii  \s_1 \nu\cdot x}z_k^{\s_1s(k)}z_h^{\s_2s(h)}$ where  $k\in \dodo_{\et_1} $ and  $h\in \dodo_{\et_2} $.

For $\s_1\s_2$ fixed the blocks come in conjugate pairs and we usually exhibit the one with $\s_1=1$.

\begin{proposition}\label{secM}
The action of $ad(\mathcal N)$ is invertible on all the blocks different from $(0,\et,\et,\s,-\s)$.
\end{proposition}
\begin{proof}
It is enough to prove that the action of the semisimple part $ad(\mathcal N^{ s })$ is invertible on all these blocks.

Now from Formula \eqref{itermini} on a block  $(\nu,\et,\et',\s, \s')$   the action is via the scalar  function 
$$\sum_i (|\mathtt j_i|^2+\omega^{(1)}_i(\xi))\nu_i+\s  ( |\er(\mathfrak t)|^2+\theta_{\mathfrak t}) +\s'  ( |\er(\mathfrak t')|^2+\theta_{\mathfrak t'}) $$ so  we only need to observe that if   $(\nu,\et,\et',\s, \s')$ is different from $(0,\et,\et,\s,-\s)$ this function is non zero.   If this expression equals 0  we must have  separately
$$\sum_i  |\mathtt j_i|^2 \nu_i+\s   |\er(\mathfrak t)|^2  +\s'    |\er(\mathfrak t')|^2 =0,\quad \sum_i  \omega^{(1)}_i(\xi) \nu_i+\s   \theta_{\mathfrak t}  +\s'   \theta_{\mathfrak t'}  =0. $$ Now we also have conservation of momentum  $\pi(\nu)+\s \er(\mathfrak t)+\s' \er(\mathfrak t')=0$.

These properties then imply that the two roots $\er(\mathfrak t), \er(\mathfrak t')$ are joined by the edge $\nu$  but in the final graph this implies that $\nu=0$ and $\er(\mathfrak t)=\er(\mathfrak t')$ the final graph was built   so that these equalities  imply  that $(\nu,\et,\et',\s, \s')=(0,\et,\et,\s,-\s)$. 
\qed\end{proof}
This  Proposition allows us to define the decomposition of $F^{(0,2)}$ given by Definition \ref{kerrg}
into the subalgebra $\mathcal F_{{\rm ker}}$ and its unique  complement $\mathcal F_{{\rm rg}}$ stable under adjoint action of $\mathcal F_{{\rm ker}}$.  
 \appendix
\section{Measure estimates}
\begin{lemma}\label{tecnico}
Consider a $C^q$ function $f$  of $n$ variables in a domain $A$ contained in some hypercube 
 of side of length $\zeta$.  If for some $k\in \N^n$ such that $|k|_1\leq q$  one has
$$
\inf_{x\in A}| \partial_x^k f|\geq  c^{|k|}>0
$$  
then for all $\alpha\geq 0$  we have, setting:
 \begin{equation}\label{measd}
A_0:=\{x\in A\,:\quad |f(x)|\leq  \alpha^{|k|}\} ,\quad {\rm meas}(A_0)\leq 2|k| \zeta^{n-1} \alpha c^{-1}.
\end{equation}
  
\end{lemma}
\begin{proof}
 By induction.  If $|k|=1 $ this is a standard argument (which one can develop also  for Lipschitz functions). Without loss of generality let us suppose that $k=e_1$. We consider  the map $F : x \mapsto (f(x),x_2,\ldots,x_n)$  which, since $|\partial_{x_1}f|>0$,  is a diffeomorphism   from $ A_0$   to some set $B $  which is contained in the product  of an $n-1$--dimensional hypercube  with size of length $\zeta$ times a segment of length  $2  \alpha$, hence the volume of  $F(A_0)$ is $\leq {\rm meas}(B)= 2  \alpha\zeta^{n-1}$. 
 
 The  Jacobian of $F^{-1}$  is $  \partial_x  f^{-1}$ so in absolute value is   $\leq  c^{-1} $. Therefore the volume of the set    $A_0=F^{-1}(F(A_0))$    can be estimated by $2  \alpha\zeta^{n-1} c^{-1}$.
 
 Let us now write (again without loss of generality)  $\partial_x^k = \partial_x^h \partial_{x_1}$ (where $|h|= |k|-1$) and set $g(x):= c^{-1}\partial_{x_1} f(x)$. 
 So that we know
 $$
\inf_{x\in A}| \partial_x^h g|\geq  c^{|h|}.
$$  
 Then by the inductive hypothesis
 $$
A_1:= \{x\in A\,:\quad |g(x)|<   \alpha^{|h|}\},\quad {\rm meas}(A_1) \leq 2|h| \zeta^{n-1} \alpha c^{-1}. 
$$
On the region $A \setminus A_1$  we have $c|g(x)|=|\partial_{x_1}   f(x)|\geq    c\alpha^{|h|} $. Then by case $|k|=1$  we have
$$
A_2:=\{x\in A \setminus A_1\,:\quad |f(x)|\leq   \alpha^{|k|}\},\quad {\rm meas}(A_2)\leq 2\zeta^{n-1}\alpha^{|k|}c^{-1}\alpha^{-|h|} =2\zeta^{n-1}\alpha c^{-1}
$$
and this  concludes the proof since $A_2=A_0\setminus A_1$ so  $A_0\subset A_1\cup A_2$ and:
$$
{\rm meas}(A_0)\leq {\rm meas}(A_1)+{\rm meas}(A_2)  \leq 2|h| \zeta^{n-1} \alpha c^{-1}+ 2\zeta^{n-1} \alpha c^{-1}= 2|k| \zeta^{n-1} \alpha c^{-1}.
$$
\qed\end{proof}

We shall often apply this  Lemma as follows, we have our domain $\mathcal O$ in the coordinates $\xi$ and its image $\tilde{\mathcal O}$ in the coordinates  $\omega$. We want perform an estimate  of a set $A_0\subset \mathcal O$ as before. We first perform the estimate for the corresponding $B_0\subset \tilde{\mathcal O}$  which satisfies the hypotheses of Lemma \ref{measd} with $\zeta=  \e^{2q}  $. The dilation factor  of the inverse of  the transformation $\xi\mapsto \omega$  is bounded by  $L^n\e^{-2n(q-1)}$ (Remark \ref{dilf}). In all the estimates which we shall perform we have  $\alpha=\e^{2q}\beta$ where usually $\beta=K^{-c\rho}$.  We deduce
\begin{corollary}\label{stixo} If the measure of $B_0=\omega(A_0)$ is bounded by  $C\zeta^{n-1}\alpha= C\e^{2q(n-1)}\e^{2q}\beta = C\e^{2q n  } \beta$ that of $A_0 $ is bounded by  $C\beta L^n\e^{-2n (q-1)   }  \e^{2q n  }\leq C' \e^{ 2n   }\beta  .   $

\end{corollary}

Let $f,g$ be inverse functions between two domains in  $\R^n$, i.e. $g\circ f=1$  and both diffeomorphisms of some class $C^p$, we write in coordinates  $\omega_i=f_i(\xi_1,\ldots,\xi_n)$  and $\xi_i=g_i(\omega_1,\ldots,\omega_n)$. Then there are universal formulas due to Fa\`a di Bruno, to express the   derivatives $b_\alpha:=\partial^\alpha_\omega g$  as polynomials in terms of   $a_\beta:= \partial^\beta_\xi f  $ and  the first derivatives $b_i:=\partial_{\omega_i} g$  (computed at the corresponding points, $\xi,\ \omega=f(\xi)$).

We do not need the explicit formulas but only to know that the nature of such a polynomial is the following. We give a {\em weight}  $|\beta|$ to $a_\beta$ and $-1$ to  $b_i$. Then the polynomial expressing $b_\alpha$ is a linear combination with rational coefficients  of monomials of weight $-1$.

So let $h=|\alpha|$,  in each monomial we have some number $k $  of factors $b_i$ which contribute $-k$ to the weight and other factors  of type $a_\beta$ of total weight  $k-1$, the total number $k$ of factors $b_i$  can be seen to be  bounded by  $2h-1$.

  We assume to have some estimates  $|\partial_\xi^{\beta} \omega|=|a_\beta|\leq \e^{2q-2|\beta|}  {M} ^{|\beta|}$, $|\partial_{\omega_i} \xi|=|b_i|\leq \e^{2 -2q}L$  and $LM<c,\ c>1$. In our setting these estimates   come from Proposition \ref{coste}  and are  justified by the fact that we work on a domain of width $\e^2$ and we start the algorithm from $\omega$ homogeneous of degree $q$. 
  
  Under these estimates we obtain for $b_\alpha=\partial^\alpha_\omega\xi$ an estimate, with $h=|\alpha|$, (and $\e<1$) 
 \begin{equation}\label{sttt}
|b_\alpha|\leq C   \e^2 L \e^{-2hq } (LM)^{2(h-1)}\leq    C  \e^{2(1 - h q)}L  c ^{2h-2}  
\end{equation} where the order comes from estimating the worst monomial  in the polynomial expression of $b_\alpha$. The constant  $C $ is some positive number, sum of the absolute values of the coefficients of this combinatorial expression and so depending only  on $n$ and $\beta$.  Summarising
\begin{lemma} 
 Let $P(\xi)$ be such that $|\partial_\xi^h P(\xi)| \leq \e^{2(q-|h|)}M^{|h|} $  and consider the map $\xi\to \omega(\xi)$ with derivatives
 $|\partial_\xi^{h} \omega|\leq  \e^{2(q-|h|)} M^{|h|}$ and its inverse $\xi(\omega)$ with $|\partial_\omega \xi| \leq \e^{2(1 - q)}L$.
 Then if $LM\leq c$ 
\begin{align}\label{duedis}
|\partial_\omega^j \xi| &\leq C  \e^{2(1 - |j|q)}L  c ^{2|j|-2} \nonumber
\\
 |\partial_\omega^jP(\xi(\omega))|&\leq C'\e^{2q(1 - |j|)}  c ^{2|j|-2}   
\end{align}
 \end{lemma}
\begin{proof} The first is formula   \eqref{sttt} and the second follows 
 from the chain rule.
 \qed\end{proof}
\begin{lemma}\label{tecni2}
 Consider a $\pp\times \pp$ matrix $A(\omega)$ and given $\nu\in \Z^n$, set 
 $$f(\omega)={\rm det}( \omega\cdot \nu I + A(\omega) )$$
 and suppose that $|\omega\cdot \nu|\leq \pp  M \e^{2q}$
    and    \begin{equation}\label{stinf}
 |\partial_{\omega_1}^j A|_{\infty} \leq C_1 \e^{2q(1-|j|)}
\end{equation}
   for some $ C_1\geq \pp M$.
   Then  there exists $C_2 $ (depending on $C_1,q,\pp $)  such that if $|\nu_1| > C_2   $  we have 
   $$
  | \partial_{\omega_1}^\pp  f| > C_2^\pp 
   $$
 \end{lemma}
 \begin{proof}
 Clearly
 $$
 f(\omega)= P_A(\omega\cdot \nu)= (\omega\cdot \nu)^\pp  + (\omega\cdot \nu)^{\pp -1}{\rm tr}(A)+\dots+det(A).
 $$
For the derivative we have that $\partial_{\omega_1}^\pp  [ (\omega\cdot \nu)^{\pp -i} \sigma_i(A)]$ is a sum of terms       $\partial_{\omega_1}^h   (\omega\cdot \nu)^{\pp -i} \partial_{\omega_1}^{\pp -h}\sigma_i(A) ,$ with $ |h|\leq \pp -i.$  We estimate
 $$|\partial_{\omega_1}^h   (\omega\cdot \nu)^{\pp -i}|=|h! \nu_1^h   (\omega\cdot \nu)^{\pp -i-h} |\leq h! \nu_1^h (\pp M \e^{2q})^{\pp -i-h}.$$
 As for $\partial_{\omega_1}^{\pp -h}\sigma_i(A) $  we have that $\sigma_i(A) $ is a sum, with signs, of $\binom{\pp }{i}$ products  of $i$ entries of $A$. For  such  a product  $\partial_{\omega_1}^{\pp -h}$ is again a sum of some $C(h,i)$ products  where each factor  has been derived  some $u_i\geq 0$ times with $\sum u_i=\pp -h$. For these  factors we apply the estimate \eqref{stinf} getting an estimate $C_1^i \e^{2q(i-\pp +h)}$ for their product. Thus for each $|\partial_{\omega_1}^h   (\omega\cdot \nu)^{\pp -i} \partial_{\omega_1}^{\pp -h}\sigma_i(A)|  $ we have an estimate $h! \binom{\pp }{i}\nu_1^h (\pp M)   ^{\pp -i-h}C_1^i\leq h! \binom{\pp }{i}\nu_1^h  C_1  ^{\pp  -h}$. Then:
 $$ | \partial_{\omega_1}^\pp  f|  \geq \pp !|\nu_1|^\pp  - \sum_{|j|\leq \pp -1}\kappa_j|\nu_1|^j C_1^{\pp -|j|} \geq (C_2)^\pp  \pp !- \sum_{j}\kappa_jC_2^{|j|}C_1^{\pp -|j|}).
 $$
The coefficients $\kappa_j$ are purely combinatorial so given $C_1$ we just need to  choose $C_1/C_2$ so that 
$$ \sum_{j}\kappa_j(\frac{C_1}{C_2})^{\pp -|j|}<\pp !/2.$$
  \qed\end{proof}
  
   In our compact set $\mathcal O\subset \e^2 \Lambda\subset \R^n$ we now consider the   pointwise $\lambda$ norm associated to the $C^\ell$ maps from $\mathcal O\to \mathbb C^\pp$ with the $L_\infty$ norm, $\Vert b(\xi)\Vert_\lambda:=\sum_{|\alpha|\leq \ell}\lambda^{|\alpha|} |\partial_\xi^{\alpha} b(\xi)|_\infty.$  

\begin{lemma}\label{tecni4}
Take a $\pp\times \pp$ matrix $A(\xi)$  satisfying the bounds:
 $$
 |\partial_\xi^{\alpha} A|_\infty \leq \lambda^{-|\alpha|} \e^{2q} \,,\quad \forall |\alpha|\leq \ell
 $$
 and a $\pp$ vector $b(\xi)$ with finite $\lambda$ norm.
 Suppose now that  $A$ is invertible for all $\xi\in \mathcal O$ and that $|A^{-1}|_2 \leq \e^{-2q} K^\varrho$ then  $y(\xi)=A^{-1}(\xi) b(\xi)$ satisfies the bounds 
 $$
 \|y\|_\lambda \leq  \ell^n 2^{\ell^2}\pp^{3\ell+1} \e^{-2q} K^{(\ell+1) \varrho} \|b\|_\lambda.
 $$  
\end{lemma} 
\begin{proof}  We preliminarily note that $$
 |\partial_\xi^{\alpha}b|_\infty \leq \lambda^{-|\alpha|}\|b\|_\lambda.
 $$
  Then we derive $\alpha$ times the equation $Ay=b$ and obtain 
 $$
 A\partial_\xi^\alpha y= \partial_\xi^\alpha b- \sum_{ j_1+j_2=\alpha,\atop j_1\neq 0} c_{j_1,j_2} (\partial_\xi^{j_1} A)\, \partial_\xi^{j_2} y\,,\quad 1+\sum_{j_1+j_2=\alpha,\atop j_1\neq 0} c_{j_1,j_2} = 2^{|\alpha|}.
 $$
 This gives the bound 
 $$
 |\partial_\xi^\alpha y|_\infty \leq  |A^{-1}|_2\Big(\pp |\partial_\xi^\alpha b|_\infty+ \pp^2\sum_{j_1+j_2=\alpha,\; j_1\neq 0} c_{j_1,j_2} |\partial_\xi^{j_1} A|_\infty |\partial_\xi^{j_2} y|_\infty \Big) $$ We now prove by induction that
 $$
 |\partial_\xi^\beta y| \leq \e^{-2q} \lambda^{-|\beta|} 2^{\ell|\beta|}\pp^{3|\beta|+1} K^{(|\beta|+1)\varrho}\|b\|_\lambda
$$ 
 by simply substituting in the right hand side of the formula all our bounds 
 
 \begin{align*}
 &|\partial_\xi^\alpha y|_\infty  \leq 
 \e^{-2q} K^\varrho\Big(\pp \lambda^{-|\alpha|}\|b\|_\lambda+\\ & \pp^{2}\sum_{j_1+j_2=\alpha\atop j_1\neq 0} c_{j_1,j_2}  \lambda^{-|j_1|} \e^{2q}\e^{-2q} \lambda^{-|j_2|} 2^{\ell |j_2|}\pp^{3|j_2|+1} K^{(|j_2|+1)\varrho}\|b\|_\lambda\Big)\stackrel{|j_1|+|j_2|=|\alpha|}{=}
 \\
&\e^{-2q}  \lambda^{-|\alpha|}\|b\|_\lambda K^\varrho \Big(\pp + \sum_{ j_1+j_2=\alpha\atop j_1\neq 0}c_{j_1,j_2}   2^{\ell |j_2|} \pp^{3(|j_2|+1)}K^{(|j_2|+1)\varrho}   \Big) \stackrel{|j_2|\leq |\alpha|-1}{\leq}
 \\
 &\leq \e^{-2q}  \lambda^{-|\alpha|}\|b\|_\lambda K^{(|\alpha|+1)\varrho}  (\pp + \pp^{3|\alpha|}2^{\ell (|\alpha|-1)}\sum_{ j_1+j_2=\alpha\atop j_1\neq 0}c_{j_1,j_2} ) \leq 
 \\
& \leq\e^{-2q}  \lambda^{-|\alpha|}\|b\|_\lambda K^{(|\alpha|+1)\varrho}  (\pp + \pp^{3|\alpha|} (2^{\ell |\alpha|}-1))\leq
\e^{-2q}\lambda^{-|\alpha|}2^{\ell |\alpha|} \pp^{3|\alpha|+1}  K^{(|\alpha|+1)\varrho} \|b\|_\lambda.
 \end{align*}
multiplying by $\lambda^{|\alpha|}$ and summing over $\alpha$ (at most $\ell^n$) we obtain the desired estimate.
 \qed\end{proof}
\subsubsection{NLS estimates}
Consider the   NLS equation restricted to some domain  $ \mathcal R_\alpha\cap \e^2\Lambda$. This defines $\omega$ and $\Omega_\et$ by setting $\tilde\Omega_\et=0$. By \cite{PP3} we know that the perturbation $P$ is quasi--T\"oplitz for some parameters $K_0,\theta_0,\mu_0$.

\begin{lemma}\label{2melquant}
\begin{enumerate}
\item There exist constants $L_0,M_0$ so that formul\ae\ \eqref{omelip} hold.

\item Given $S_0 $,  we can choose the domain $\mathcal O_0$ and a constant $a_0$ so that Formul\ae\ \eqref{piffedue}  hold  for $\ome,\Omega_\et$ defined by $\mathcal N$  with  $\mathcal O= \mathcal O_0$ and $a=a_0$.

\end{enumerate}
\end{lemma}
\begin{proof}
i) We have that the  functions $\omega-\mathtt j^{(2)}$ and its inverse are homogeneous and invertible, moreover the functions $\theta_\et$ and $\Omega_\et^{nil}$ are also homogeneous and from a finite list   so we satisfy these formulas by  taking the appropriate maximum on   a compact domain of the unit sphere contained in some $\mathcal R_\alpha$ and such that its cone from the origin contains the domain $\mathcal O_0$.  

ii) By the Melnikov conditions, Proposition \ref{secM}, the matrices under consideration are all invertible as functions outside the algebraic hypersurfaces where the determinant is 0.  All these matrices evaluated at $\xi=0$ are an integer multiple of the identity, the remaining part is homogeneous and runs  in a finite list. If this integer is non--zero then   the estimates  hold provided $\e$ is small enough. Otherwise we  have to restrict the domain so to avoid the finite number of hypersurfaces where one of these determinants vanish and estimate on the unit sphere.
\qed\end{proof}
As for $S_0$ we need to fix it as $n C_2$ so to apply   lemma \ref{tecni2} at all steps where $A(\omega)$ are the matrices, appearing in Formula\ae\  \eqref{piffe} in the coordinates $\omega$.

\begin{proof}[Conclusion of the proof of Lemma \ref{bababa}] By the non-degeneracy condition insured by Lemma \ref{2melquant}, we may assume $|\nu|> S_0$ so one of the coordinates $\nu_i$ has $|\nu_i|>S_0/n$  and we may assume this happens for $\nu_1$.  We will fix $S_0$ sufficiently large as seen in the course of the proof.

If $\s=\s'=0$ we  claim that we remove from the set $\mathcal O$ a region of order $\e^{2n} K^{-\varrho}$. For this follow the strategy of Remark \ref{dilf}.  Namely, apply Lemma  \ref{tecnico} to $f(\xi)=\omega\cdot \nu $ we then estimate the measure of the set $A_0$ where $|\omega\cdot \nu |<\e^{2q}K^{-\varrho}$ as follows.
 
 First  by  Formula   \eqref{measd}  the  measure of the image $B_0$ of $A_0$ under $\omega$  is bounded by  $2\zeta^{ n-1}\e^{2q}K^{-\varrho} c^{-1}$ where $\zeta$  is the size of the side of a hypercube containing  $B_0$ and $c$  is a lower bound for one of the derivatives    $\nu_i=\partial_{\omega_i}(\omega\cdot \nu)$.   We may take for $c$ the maximum of the absolute values of the $\nu_i$ which is  $>  S_0/n$.  As for $\zeta  $ we have $\zeta\leq M \e^{2q}$  so we have a bound
 $$meas(A_0)\leq L^n\e^{2 n(q-1) } meas(B_0)\leq  L^n\e^{-2 n(q-1) }   (M\e^{2q})^{n-1} K^{-\varrho} \e^{2q} c^{-1}\lessdot \e^{2n} K^{-\varrho}  .$$
If $\et,\et'$ are both in $\mathfrak T_g$ we  set for $\s=0,\pm 1$
$$ 
f(\omega)= \omega\cdot \nu + \lambda^{(i)}_\et+ \s \lambda^{(j)}_{\et'}.
$$

By definition of $M$ we have $|\partial_{\xi} \lambda^{(j)}_{\et}|)\lessdot \;\e^{2q-1} M$ so by applying Lemma A.2.    we know that $|\partial_{\omega_1} f(\omega)|=|\nu_1+ \partial_{\omega_1}(\lambda^{(i)}_\et+ \s \lambda^{(j)}_{\et'})|  \geq |\nu_1|-| \partial_{\omega_1}(\lambda^{(i)}_\et+ \s \lambda^{(j)}_{\et'})|  \geq S_0/n- C\e^{2q-1}    $ for some constant $C$, so taking $S_0$ sufficiently large  the proof is concluded by applying Lemma \ref{tecnico} with $x $ replaced by $ \omega$. 

We get that the resonant set in the variables $\omega$ has measure of order $\e^{ 2q(n-1)+2q} K^{- \varrho}=\e^{ 2q n } K^{- \varrho}$. In order to obtain the estimates in the variables $\xi$ we multiply by the dilation factor $L^n\e^{2n-2nq}$ getting $\lessdot\, \e^{ 2  n } K^{- \varrho}$.\smallskip

When either $\et$ or $\et'$ are  in $\mathfrak T_f$ we need to study
$$
f(\omega)={\rm det}( (\omega\cdot \nu) I + L(\Omega_\et  )+ \s R(\Omega_{\et'} ) ).$$
This is a matrix of dimension $\pp=\d_\et \d_{\et'}$  (resp. of dimension $\pp=\d_\et$ if $\s=0$). By hypothesis setting $F(\xi)=  L(\Omega_\et(\xi)) + \s R(\Omega_{\et'}(\xi))$ we get $|\partial^\alpha_\xi F|_\infty \leq 2\e^{2q-2|\alpha|} M$. 

We have that $A(\omega)= F(\xi(\omega))$ satisfies estimates as in \eqref{stinf} and under the  hypothesis $|\omega\cdot \nu|\leq \pp M \e^{2q}$  we can apply  Lemma A.2.  Since we are assuming  $|\nu_1|>S_0/n$ we only need to choose $S_0$ large enough in order to prove $|\partial_{\ome_1}^p f| \geq C_2$ some positive constant. 

By our definitions the resonance condition is $| f(\omega)|\geq  \e^{2q \pp} K^{-\varrho+1 } $ where $\pp=d_\et d_{\et'}$. Then we apply Lemma \ref{tecnico} with $|k|=\pp$ and $ \alpha= \e^{2q  } K^{-\varrho/\pp } $. 
 We  get that the resonant set in the variables $\omega$ has measure of order $2\pp \zeta^{n-1} \e^{2q} K^{(- \varrho+1)/\pp }$. So by Corollary \ref{stixo}  the resonant set in the variables $\xi$ has measure of order $  \e^{2n-2q} \e^{2q} K^{- \varrho/\pp }= \e^{2n } K^{(- \varrho+1)/\pp }$.\qed\end{proof}

Consider a normed space of sequences $v:=(v_i)$  we say that $v\leq w$ if and only if $v_i\leq w_i,\  \forall i$. 

Given $k$ sequences  $v^{(j)},\ j=1,\leq m$ we have by definition  $\sup_j  v^{(j)}:=(\sup_j  v^{(j)}_i)$. Assume that the norm $|\cdot |$ satisfies
$$0\leq v\leq w\implies |v|\leq|w|.$$
\begin{lemma}\label{scambio}  If $v^{(j)}\geq 0,\ j=1,\ldots, m$ are $m$ positive sequences we have then $v^{(j)}\leq \sup_j  v^{(j)}\leq \sum_j  v^{(j)}$ hence
\begin{equation}\label{scaa}
 \sup_j  |v^{(j)}|\leq |\sup_j  v^{(j)} | \leq \sum_j  |v^{(j)}|\leq m  \sup_j  |v^{(j)}|.
\end{equation}

If $a_i$ are positive numbers 
\begin{equation}\label{pos}
|\sum_ja_jv^{(j)}|\leq  \sum_ja_j|v^{(j)}|\leq  m|\sup_ja_jv^{(j)}|\leq  m|\sum_ja_jv^{(j)}|.
\end{equation}
\end{lemma}
\begin{lemma}\label{domino}
Consider a Hamiltonian $$Q(x,w)= \sum_{|\nu|\leq K,\alpha,\beta\atop
|\alpha|+|\beta|\geq 1}e^{\nu\cdot x}Q_{\nu,\alpha,\beta}z^\alpha\bar z^\beta,$$ (i.e. independent of $y$, of degree at least one in $w$ and in $x$  a trigonomentric polynomial of degree $|\nu|_1\leq K$). Denote  its associated vector field by
$$
X_Q= X_Q^{(y)}+X_Q^{(w)},\quad  X_Q^{(y)}= \partial_xQ(x,w)\cdot \partial_y\,,\quad X_Q^{(w)}=-\ii \partial_{z}Q(x,w)\cdot \partial_{\bar z}+\ii \partial_{\bar z}Q(x,w)\cdot \partial_{ z}\,,
$$
then one has that \begin{equation}\label{stimona}
\| X_Q^{(y)}\|_{s,r}\leq  K \| X_Q^{(w)}\|_{s,r}
\end{equation}
\end{lemma}
\begin{proof}
We can write (collecting is some arbitrary way the monomials)
$$
Q=\sum_{k} A_k(x,w) z_k + B_k(x,w)\bar z_k.
$$ 
Then, by definition
$$
\| X_Q^{(y)}\|_{s,r}= r^{-2}\sup_{\|w\|_{a,p}\leq r}\sum_{j=1}^n |\sum_k|M\partial_{x_j}A_k(x,w)|z_k+ |M\partial_{x_j}B_k(x,w)|\bar z_k|
$$
now we know that $\sum_{j=1}^n |M\partial_{x_j}A_k(x,w)|\leq K |M A_k(x,w)|$ then by Cauchy-Schwartz we have
$$
\| X_Q^{(y)}\|_{s,r}\leq r^{-2}K\sup_{\|w\|_{a,p}\leq r}\sqrt{\sum_k |z_k|^2+|\bar z_k|^2}\sqrt{\sum_k |M A_k|^2+|M B_k|^2}
$$ 
and we have that $|M A_k|\leq |M \partial_{z_k} Q|$ (same for $B$)
since $\sum_k |z_k|^2+|\bar z_k|^2\leq r^2$ and the $\ell^2$ norm is dominated by the $\|\cdot\|_{a,p}$ norm the claim follows.
\qed\end{proof}
Now  in $\ell_{a,p}$ we can define the norm
$$v= \{v_k\}_{k\in \Z^d}\qquad   | v|^2_{a,p}= \sum_{\et} e^{2a|\er_\et|}|\er_\et|^{2p} \sup_{k\in \dodo_\et}|v_k|^2.
$$
\begin{remark}\label{normeeq}
This norm is built so that if for $v,w\in \ell_{a,p}$ we have for all $\et\in \mathfrak T$ that  $\sup_{k\in\dodo_\et}|v_k|\leq \sup_{k\in\dodo_\et}|w_k|$ then $|v|_{a,p}\leq |w|_{a,p}$.
One sees that the norms $\|\cdot\|_{a,p}$ and $|\cdot|_{a,p}$ are equivalent, namely  $c_{a,p} | v|_{a,p}\leq \| v\|_{a,p}\leq C_{a,p} | v|_{a,p}$, with:
$$
C_{a,p}:= \sup_{\et,k\in \dodo_\et} \sqrt{d_\et} e^{a( |k|-|\er_\et| )}\frac{|k|^{p}}{|\er_\et|^{p}}\,,\quad c_{a,p}= \inf_{\et,k\in \dodo_\et} e^{a( |k|-|\er_\et| )}\frac{|k|^{p}}{|\er_\et|^{p}}.
$$
\end{remark}
  \begin{lemma}\label{orrore}
 Given a hamiltonian $H=N+P$ with $P$ a  quasi--T\"oplitz  function with parameters $(K,\theta,\mu)$,  and $N$ as in \eqref{notazioni}. Consider the linear operator
 $$
L:= {\rm ad}(N )+ \Pi_{{\rm rg},\leq K} {\rm ad}(\Pi_{> 2}P ) 
 $$
 and the equation $L F= \Pi_{{\rm rg},\leq K}P$. In the set of $\xi$ for which formul\ae \eqref{piffe} hold, for all $|\nu|<K$, $\et,\et'\in \mathfrak T$, 
 such equation admits a unique solution $F=\Pi_{{\rm rg}}F$ which is a quasi--T\"oplitz  function  with parameters $(K,\theta,\mu)$ and satisfies the bounds
 $$
 \|X_F\|^T_{\overrightarrow p^+}\lessdot \delta^{-2} K^{3(\ell+4)\varrho}\|X_{P_{{\rm rg}}}\|^T_{\overrightarrow p}\,,$$
 where $ \overrightarrow p^+=( s-2\delta s ,r-2\delta r,K,\theta+2\delta\theta,\mu-2\delta\mu)$ and  $\overrightarrow p=(s,r,K,\theta,\mu)$.
\end{lemma}
\begin{proof}
 We discuss in detail the estimates on the $\lambda$ norm, the ones on the T\"oplitz norm follow verbatim from Proposition 10.16 of \cite{PP3}.
 
  We recall that by remark \ref{nilp},
setting $D=  {\rm ad}(\Pi_{{\rm ker}}P )$ and $A= D^{-1} \Pi_{{\rm rg},\leq K} {\rm ad}(\Pi_{> 2}P ) $ we have
$$
L=D(1+A)\quad\longrightarrow\quad L^{-1}= D^{-1}( 1- A +A^2).
$$
Let us now work in the set of $\xi$ where \eqref{piffe} hold. We first prove that for any quadratic Hamiltonian $b\in \mathcal F_{{{\rm rg}},\leq K}$, setting $Y= D^{-1}b$, we have
\begin{equation}\label{alleluia}
\|X_Y\|^\lambda_{s,r}\leq \e^{-2q}K^{(\ell+3) \varrho}\|X_b\|^\lambda_{s,r}.
\end{equation}

%
%

 As a first step we notice that if we divide $Y= Y^{(0)}+Y^{(1)}+Y^{(2)}$ w.r.t the degree in $w=z,\bar z$ then it is sufficient to bound each $\|X_{Y^{(i)}}\|_{s,r}^\lambda$ in terms of the corresponding $\|X_{b^{(i)}}\|_{s,r}^\lambda$.

The term $Y^{(0)}$ is trivial. 
We now study $Y,b$ of degree one.

In studying linear real Hamiltonians of the form $h= \sum_{\s,k,\nu} h_{\s,k}(\nu)e^{\ii \nu\cdot x} {z_k^\s}$, by Lemma \ref{domino}, we only need to bound the $w$ component of the vector field, namely
$$
 \| X_{h}\|_{s,r} \leq 2(K+1) r^{-1}\big( \sum_{k} (\sum_\nu | h_{+,k}(\nu)| e^{s|\nu|} )^2 e^{2a |k|}|k|^{2p}\big)^{\frac12}
$$
Now we use this formulas with $ h=\partial_\xi^\alpha Y^{(1)}$ , $|\alpha|\leq \ell$.
Then, by using \eqref{pos} with $a_j=\lambda^{|\alpha|}, m\leq (\ell+1)^n, v_j= \{\sum_\nu e^{s|\nu|}|\partial_\xi^\alpha h_{s(k),k}(\nu)|\}_{k\in \Z^d}$ and $|\cdot|=\|\cdot\|_{a,p}$, we have  that
\begin{eqnarray}\label{prima}
\| X_{Y}\|^\lambda_{s,r}&\leq& (2   K +2)(\ell+1)^n  r^{-1}  \| M^\lambda Y\|_{a,p}, \nonumber\\ M^\lambda h &:=& \{\sum_\nu e^{s|\nu|}\sum_{|\alpha|\leq \ell} \lambda^{|\alpha|} |\partial_\xi^\alpha h_{s(k),k}(\nu)|\}_{k\in \Z^d}.
\end{eqnarray}

Now we have  by Lemma \ref{tecni4} that for each $\et$ and for each $k\in \dodo_\et$
$$
\sum_{|\alpha|\leq \ell} \lambda^{|\alpha|}\sup_{k\in \dodo_\et} |\partial_\xi^\alpha Y_{s(k),k}(\nu)| \leq \e^{-2q} K^{(\ell+2)\varrho}\sum_{|\alpha|\leq \ell} \lambda^{|\alpha|} \sup_{k\in \dodo_\et}|\partial_\xi^\alpha b_{s(k),k}(\nu)| 
$$
We exchange the $\sup$ with the sum over $\nu$ using  \eqref{scaa} with $|\{v_\nu\}|= \sum_\nu e^{s|\nu|}|v_\nu|$, then we have that
$$
\sup_{k\in \dodo_\et}|(M^\lambda Y)_k|\leq d_\et \e^{-2q} K^{(\ell+2)\varrho} \sup_{k\in \dodo_\et}|(M^\lambda b)_k|
$$

 By Remark \ref{normeeq} we conclude that
$| M^\lambda Y|_{a,p} \leq \ell \e^{-2q} K^{(\ell+2)\varrho} | M^\lambda b|_{a,p}$
and
\begin{equation}\label{ultima}
\| M^\lambda Y\|_{a,p} \leq \ell C_{a,p}c_{a,p}^{-1}\e^{-2q} K^{(\ell+2)\varrho} \| M^\lambda b\|_{a,p}.
\end{equation}
Then by \eqref{pos} $$\| M^\lambda b\|_{a,p}\leq \sum_{|\alpha|\leq \ell} \lambda^{|\alpha|} \| \{\sum_{\nu}e^{s|\nu|}|\partial_\xi^\alpha b_{s(k),k}(\nu)|\}_{k\in \Z^d} \|_{a,p}\leq r \|X_{b}\|^\lambda_{s,r}. $$
Substituting  \eqref{ultima}  in \eqref{prima} and then using this last bound we obtain 
$$
\| X_{Y}\|^\lambda_{s,r}\leq (2   K +2)(\ell+1)^n \ell C_{a,p}c_{a,p}^{-1}\e^{-2q} K^{(\ell+2)\varrho}\|X_{b}\|^\lambda_{s,r}\leq \e^{-2q}K^{(\ell+3)\varrho}\|X_{b}\|^\lambda_{s,r},
$$
provided $K$ is large (note that $\varrho>1$).\bigskip

We need to treat now the case of (regular) quadratic Hamiltonians 
 (with conservation of momentum)  which belong to  the range. 
To  any quadratic hamiltonian $Q(x,w)$ we associate an $x$ dependent linear operator $\tilde Q(x)$  on $\ell_{a,p}$ by Poisson bracket. By Formula \eqref{represQ} this amounts to writing $Q(x,w)=-\frac12 w \tilde Q(x) J w^t$. 
Then in the basis $z_k,\bar z_k$ we can define the majorant (of the operator ad$(Q)$ ) in matrix terms as
$$
\widehat Q_{k,h}^{\s,\s'}= \sum_{\nu}e^{s|\nu|} |\tilde Q_{k,h}^{\s,\s'}(\nu)|\,
$$
Similarly to the linear case we only need to control $$\|X^{(w)}_Q\|_{s,r}= \sup_{\|w\|_{a,p}\leq 1}\| \widehat Q w\|_{a,p}\,,\quad w=\{z_k,\bar z_k\}_{k\in\Z^d}.$$
Then  passing to the equivalent norm $|\cdot|_{a,p}$ of Remark \ref{normeeq}, we have
$$
\|X_Q\|_{s,r}\leq (K+1)c_{a,p}^{-1}C_{a,p} \sup_{|w|_{a,p}\leq 1}| Q w|_{a,p}.
$$
 Now as in the linear case we pass to the $\lambda$ norm and get
 \begin{eqnarray}\label{parto}
& \|X_Q\|^\lambda_{s,r}\leq (\ell+1)^{n}(K+1)c_{a,p}^{-1}C_{a,p} \sup_{|w|_{a,p}\leq 1}| \widehat Q^\lambda w|_{a,p}\,, \nonumber\\ & (\widehat Q^\lambda)_{k,h}^{\s,\s'}:=\sum_{\alpha}\lambda^{|\alpha|}\sum_{\nu}e^{s|\nu|} |\partial_\xi^\alpha\tilde Q_{k,h}^{\s,\s'}(\nu)|
\end{eqnarray} 
 We  consider  $Y=D^{-1}b$  where $b$  is a quadratic Hamiltonian in the range.  Since $D$ preserves the range  we have that also $Y$ is in the range.
 Now we claim that for all $w$ and for all $\et$ we have 
 \begin{equation}\label{assurdo}
\sup_{k\in \dodo_\et}|(\widehat Y^\lambda w)_k|\leq \ell \e^{-2q} K^{(\ell+2)\varrho} \sup_{k\in \dodo_\et}|(\hat b^\lambda \tilde w)_k|\,,\quad \tilde w_h= \sup_{k\in \dodo_\et}|w_{k}|\;\forall h\in \dodo_\et.
\end{equation}
If this holds true by remark \ref{normeeq} we have 
$$
|\widehat Y^\lambda w|_{a,p}\leq \ell \e^{-2q} K^{(\ell+2)\varrho}|\hat b^\lambda \tilde w|_{a,p}
$$
then since $|\tilde w|_{a,p}=|w|_{a,p}$ (by definition) we have
$$
\sup_{|w|_{a,p}\leq 1} |\widehat Y^\lambda w|_{a,p}\leq \ell \e^{-2q} K^{(\ell+2)\varrho}\sup_{|u|_{a,p}\leq 1}|\hat b^\lambda u|_{a,p}.
$$
Now we can apply \eqref{pos} to take out the sum over $\alpha$ from the norm $|\cdot|_{a,p}$ and we get
$$
|\hat b^\lambda u|_{a,p}\leq (\ell+1)^n \sum_{\alpha}\lambda^{|\alpha|}|\widehat{ \partial_\xi^\alpha b}\; u|_{a,p}\leq (\ell+1)^n  \|X_b\|_{s,r}^\lambda
$$
$$
\|X_Y\|^\lambda_{s,r}\leq (\ell+1)^{n}(K+1)c_{a,p}^{-1}C_{a,p}\ell \e^{-2q} K^{(\ell+2)\varrho}(\ell+1)^n  \|X_b\|_{s,r}^\lambda\leq 
K^{(\ell+3)\varrho}\e^{-2q}  \|X_b\|_{s,r}^\lambda. $$
For the proof of \eqref{assurdo} we first use Lemma \ref{tecni4} to deduce
$$
\sum_\al \lambda^{|\alpha|}\sup_{k\in \dodo_\et\atop h\in \dodo_{\et'}}|\partial_\xi^\alpha \tilde Y_{k,h}^{s(k),s(h)}(\nu)|\leq \e^{-2q} K^{(\ell+2)\varrho}\sum_\al \lambda^{|\alpha|}\sup_{k\in \dodo_\et\atop h\in \dodo_{\et'}}|\partial_\xi^\alpha \tilde b_{k,h}^{s(k),s(h)}(\nu)|.
$$
By repeated use of \eqref{scaa} and \eqref{pos} and using the fact that $d_\et d_{\et'}\leq \ell$ the formula follows.
This completes the proof of \eqref{alleluia}.
Now we turn to  $F= L^{-1}P_{{\rm rg}}$. For any regular hamiltonian $h$ and for all $s'<s,r'<r$, we can bound $\|X_{\{P, h\}}\|_{s',r'}^\lambda$ by Cauchy estimates recalling that
$\|X_P\|_{s,r}\leq \Theta	\leq \frac32 \Theta_0$ (here $\Theta$ is one of the telescopic parameters defined in section \ref{quaes}). Then by using  \eqref{alleluia} at most three times we get
\begin{equation}\label{alle2}
\|X_F\|^\lambda_{s-2\delta s ,r-2\delta r}\lessdot \delta^{-2} K^{3(\ell+3)\varrho}\|X_{P_{{\rm rg}}}\|^\lambda_{s,r}.
\end{equation}
Then passing to the T\"oplitz norm we follow word by word \cite{PP3}, Proposition 10.16.  

By definition the quasi--T\"oplitz property for $F$ depends only on the quadratic part  $F^{(2)}$.  Then proving that $F$ is quasi-T\"oplitz requires that we produce a {\em piece-wise T\"oplitz} approximation and an error. As in Proposition 10.16 for the part of $F$ with $\s\s'=1$ we  take  zero as piece-wise T\"oplitz approximation.

For given $\nu$  we treat the quadratic part as sum of blocks  $F_{\nu,\et,\et'}$.  When  $\s\s'=-1$ by conservation of momentum  $\pi(\nu)+\er_\et-\er_{\et'}=0$  so only $\theta_{\et'}=\theta'$ must be fixed (and must belong to a finite list) then the couple $(\pi(\nu)+\er_\et, \theta')$ identifies (at most) one $\et'$ and hence the block $\dodo_{\et'}$. Then  we identify the parallel strata $Z$ to which $\dodo_\et$ belong and the same for $\dodo_{\et'}$. We define  the quasi-T\"oplitz approximation as follows.  Recall that  $F_{\nu,\et,\et'}$ solves the equation  discussed in remark \eqref{nilp}  inverting the operator $L_m$. This solution is a     3 step algorithm in each step we either perform a Poisson bracket, which preserves the Quasi--T\"oplitz property  or  we solve an equation of type
\begin{equation}\label{sullos}
\Big((\omega\cdot\nu)-2  \pi(\nu)\cdot \er_\et   -|\pi(\nu)|^2+\theta_\et+\tilde \Omega_\et-\theta_{\et'}- \tilde \Omega_{\et'}\Big) F_{\nu,\et,\et'}=P_{\nu,\et,\et'},
\end{equation} by induction $P_{\nu,\et,\et'}$ has a  quasi--T\"oplitz  approximation  $P_{\nu,\et,\et'}^{qt}$ same for $\tilde \Omega_\et$ and $\tilde \Omega_{\et'}$.  We distinguish two cases, if $   \pi(\nu)\cdot x   $ is not constant on $\Sigma_\et$  we have seen   that $  \pi(\nu)\cdot \er_\et   $  is large hence $F_{\nu,\et,\et'}$ is quite small and   we can again take zero as its piece-wise T\"oplitz approximation.  

Otherwise  in Formula \eqref{sullos}, the only terms  which are not constant  on the stratum  are  $\tilde \Omega_\et- \tilde \Omega_{\et'}$. By induction they have a quasi--T\"oplitz  approximation  $\tilde \Omega_\et^{qt}- \tilde \Omega_{\et'}^{qt}$ and the final error will produce the error of the quasi-T\"oplitz approximation.

That is  we take the solution $ \bar F_{\nu,\et,\et'}^{qt}$ of  
\begin{equation}\label{sullos1}
\Big((\omega\cdot\nu)-2  \pi(\nu)\cdot \er_\et   -|\pi(\nu)|^2+ \theta_\et+\tilde \Omega^{qt}_\et-\theta_{\et'}- \tilde \Omega^{qt}_{\et'} \Big) \bar F_{\nu,\et,\et'}^{qt}=P_{\nu,\et,\et'}^{qt}.
\end{equation} as   quasi-T\"oplitz approximation of $F_{\nu,\et,\et'}$ on this stratum.

The final point is now to show that  we can bound the inverse of the matrix on the left-hand side of \eqref{sullos1} by $K^{\rho}$.
on this stratum the constraint \eqref{reson} on a single point is sufficient to constraint  on all points of the stratum. For this one has to compute  the value of $\varrho$  to be used on the stratum  and compare it with the error term  in the   quasi--T\"oplitz  approximation.
\begin{eqnarray}\label{alle3}
& \|X_F\|^T_{\overrightarrow p^+}\lessdot \delta^{-2} K^{3(\ell+4)\varrho}\|X_{P_{{\rm rg}}}\|^T_{\overrightarrow p}\,,\nonumber \\ & \overrightarrow p^+=( s-2\delta s ,r-2\delta r,K,\theta+2\delta\theta,\mu-2\delta\mu)\,,\; \overrightarrow p=(s,r,K,\theta,\mu)
\end{eqnarray}
\qed\end{proof}
\section{Quasi T\"oplitz structure}\label{prilone}
{\em We group here an informal description of the main properties of quasi-T\"oplitz functions needed through the KAM algorithm. Recall that the notion of quasi-T\"oplitz is given through some parameters $K,\vartheta,\mu$.

First one has to identify,  for each $N\geq K$ a natural affine structure. This is done by introducing the notion of optimal presentation for a  point $m$ and the notion of cut.}
\paragraph{The stratification by cuts\label{bycuts}.}  The optimal presentation of  a point $m\in \Z^d$ is a combinatorial notion dependent on   $N$. One presents the point $m$ as the intersection of  $d$ hyperplanes    $(v_i,x)=p_i$ with $v_i$ of norm $\leq N$ and such  that the list $p_i,v_i$ is {\em minimal} in an explicit lexicographic order. This presentation, called {\em optimal}  is unique and denoted   $m\frec{N}[v_i;p_i]$.

Then, given an open interval $I=(a,b)$,  one says that a point $m\frec{N}[v_i;p_i]$ has a cut at $I$ if no coordinate $p_i$ lies in $I$.  If it has such a cut  then we have some $p_\ell<a,\ p_{\ell+1}>b$  and the first $\ell$ equations of the optimal presentation of $p$  define some affine space of codimension $\ell$, denoted $ [v_i;p_i]_\ell$.  This defines a linear stratification at least on the set of points which have such a cut.

If one gives  $d+1$ disjoint increasing intervals $I_j,\ j=1,\ldots,d+1$   clearly for any point $m$ there is a minimum $j$ such that $m$ has a cut for $I_j$.  To this cut is then associated an affine space and   in this way one constructs the desired stratification as explained in remark \ref{lst}.  
We introduce a parameter $\varrho_0$ which will be fixed at the end.  We start from   intervals $(\varrho_i,\varrho_{i+1})$   in the scale $\log_N$   with
$$\varrho_1=4d\varrho_0,\quad \varrho_{i+1}=4d \varrho_i$$  and impose for the condition that $m\frec{N}[v_i;p_i]$ has a cut at $\ell$ if there exists $j$ such  that,  $\log_N 2 p_\ell< \varrho_j,\ \log_N \frac{p_{\ell+1}}{4}\geq \varrho_{j+1}$. 
This indeed is a {\em cut} associated to $d+1$ disjoint intervals and hence provides a stratification $\Sigma^{N,\varrho}$, see section 3.1 of \cite{PX}.  We shall usually drop the symbol $ \varrho $ and write just $\Sigma^{N }$.

 Thus given any point $m$ this construction determines the stratum $\Sigma^N_m$ to which it belongs to and the linear space, of some codimension $\ell$, spanned by the stratum.  If $v$ is a vector with $|v|_1\leq N$  consider the scalar product $(v,x),\ x\in \Sigma^N_m$. If this is not constant on the stratum by the definitions it is {\em large with $4N^{\varrho_{j+1}}$}.   
 \begin{lemma}\label{ecchec}
 The number of strata with a cut at $\varrho_j$  is less than $N^{ (2d-1)\varrho_j}$.
\end{lemma}
\begin{proof}
 This is Remark 7.7 of \cite{PP3} with $p= 4 N^{\varrho_j}$, indeed summing over the codimension
 $$\sum_{\ell=0}^d(2 \kappa N)^{\ell d}(4N^{ \varrho_j})^\ell \leq N^{d(\varrho_j+d+1)} \leq N^{(2d-2)\varrho_j}$$ provided that
 $ \varrho_0> d(d+1)/(d-2)$.
\qed\end{proof}
The stratification $\Sigma^N$, provided that $N$ is large enough, refines the  stratification  given by Definition \ref{basstra} (cf.Theorem 5 of \cite{PP3}).  In fact  given  any element $\et\in \mathfrak T_s$, this determines the set $\dodo_\et$ and the root $\er_\et$ (which in general is NOT in the set $\dodo_\et$).

We then have that, the stratum through  $\er_\et$ and the ones passing through each individual  point of  $\dodo_\et$ are all parallel to each other and  the union of these strata is also a union of  some family  of sets $\{\dodo_\et,\er_\et\}$.\footnote{strictly speaking one needs a slight refinement  of  the definitions adding some more strata but with the same estimate on the number of strata.}

\begin{remark}\label{strT}
In other words  we have a decomposition of  the set  $\mathfrak T_s$ into    sets  $\mathfrak T_s(i)$ (which are the ones introduced in Proposition \ref{prinor}), so that   each stratum   $\mathfrak T_s(i)$ is the indexing set of the  strata through the points in $\dodo_\et$ and $\er_\et$ for all  $\et\in \mathfrak T_s(i)$.
\end{remark}  \paragraph{$\tau$--bilinear and quasi--T\"oplitz functions}  The quasi-T\"oplitz property can not be given only on the quadratic terms of a Hamiltonian since we need  a class of functions closed w.r.t. Poisson brackets, so we first relax the notion of quadratic Hamiltonian. \smallskip

Given $\frac12\leq \mu,\vartheta\leq 4$ and $\varrho_1\leq 4 d \tau\leq \varrho_{d+1}$ we define (cf. 8.6 of \cite{PP3}) the {\em $(N,\tet1,\mu,\tau)$--bilinear  functions} to be  the    functions which are bilinear in the {\em high variables} $z_m^\s,z_n^{\s'}$. By this  we mean that $|m|,|n| > \vartheta N^{a_{d+1}}$ and both $m$ and $n$ have a cut at $(\mu N^\tau,\tet1 N^{4d \tau})$. These functions may depend on  $x,y$ and on the   {\em small  variables} $z_j^\s$ with $|j|<\mu N^3$ in a possibly complicated way but with  the constraint that the coefficients have {\em low momentum} (cf. 8.2 of \cite{PP3}).

Finally we define  the {\em piecewise T\"oplitz} functions as those   $(N,\tet1,\mu,\tau)$--bilinear functions which are constant when restricted to  each stratum (cf. 8.10 of \cite{PP3}).

We can now define the  $(K,\tet1,\mu)$--quasi--T\"oplitz
functions.  Informally speaking given a function    $f$, for all $N>K, \varrho_1\leq 4d \tau\leq \varrho_{d+1}$, we project it   on the   $(N,\tet1,\mu,\tau)$--bilinear functions   and  we say that $f$ is quasi--T\"oplitz if all these projections are {\em well approximated} by a piecewise T\"oplitz function.
To be more precise, $\tau$ controls the size of the error function, namely    the $(N,\tet1,\mu,\tau)$--bilinear  part of $f$ is approximated by a piecewise T\"oplitz function with an error of the order $N^{-4d \tau}$,  for all $N\geq K$.

The role of the parameters $K,\tet1,\mu$ is to ensure  that if $f,g$ are quasi--T\"oplitz with parameters $K,\tet1,\mu$ then $\{f,g\}$ is  quasi--T\"oplitz  for all $\vartheta'>\vartheta$ and $ \mu'<\mu$ provided  $K'>K$ is large enough.

The stratification   $\Sigma^N$ described above is connected to the quasi-T\"oplitz structure, indeed if $m$ belongs to a stratum $\Sigma_m^N$  of codimension $\ell$ then, by construction, $m$ has a cut  at $(\mu N^{\varrho_j},\tet1 N^{4d\varrho_{j}})$ for all $\frac12\leq \mu,\vartheta\leq 4$ where $\varrho_j$ is fixed by $\Sigma^N$ and we shall denote it by $\varrho_{_\Sigma}$.  Then one can show that all points $n$ close to $m$ have a similar cut, see Lemmata 7.21 to 7.24 of \cite{PP3}. 
In \S  8.12  of \cite{PP3}  there are several properties of  diagonal quasi-T\"oplitz functions which are needed in the KAM algorithm. These properties are in fact  also valid  with the obvious modifications for the block diagonal  quasi-T\"oplitz functions  in the subalgebra $\mathcal F_{{\rm ker}}$ of Definition \ref{kerrg}.

This is due to the fact that,   provided we start  from $N\geq K_0$ and $K_0$ is large enough    the corresponding stratifications refine the affine stratification  of  Proposition  \ref{teo2}, part iii). 

When we apply it to  the  hamiltonians in the KAM algorithm we exploit the fact that,   by part iv) of the same proposition the corresponding function $\theta_\et$  is the same on all points of the stratum.
Moreover we need the following
\begin{lemma}\label{eccec0}
 Take a quasi--T\"oplitz  function $F$ quadratic in $w$ and its projection $\Pi_{{\rm ker}}F=\sum_{\et} F_\et. $ Then for each  $N\geq K$  consider the stratification $\Sigma(N)$. If $\er_\et,\er_{\et'}$ belong to the same stratum $Y$ (of codimension $\ell$) and $\dodo_\et= \dodo_{\et'}+ T_Y$ (as in Proposition \ref{teo2} {\it iii}.) then  $\theta_\et=\theta_{\et'}$ and
 $$| F_\et-F_{\et'}|^\lambda_\infty \leq N^{-4d \varrho_{_{\Sigma}}} \|X_F\|_{\overrightarrow p}^T $$
\end{lemma}
\begin{proof}
This follows word by word as the proof of Lemma 8.20--Formula (92) of \cite{PP3}.
 \qed
\end{proof}

%
\end{document}